\documentclass[final,1p]{elsarticle}
\usepackage{amssymb,latexsym,amsmath}
\usepackage{amsthm}
\usepackage{color,caption,cases}
\usepackage[colorlinks,linkcolor=blue, anchorcolor=blue, citecolor=blue]{hyperref}
\usepackage{bm}
\usepackage{graphicx,graphics}
\usepackage{subfigure}
\usepackage{algorithm}      
\usepackage{algpseudocode}
\usepackage{booktabs}
\usepackage{showlabels}
\usepackage{tabularx}
\usepackage{stmaryrd}
\usepackage[T1]{fontenc}
\usepackage{makecell}
\usepackage{multirow}
\usepackage{epstopdf} 
\biboptions{numbers,sort&compress}
\usepackage{chemfig}
\usepackage{float}
\allowdisplaybreaks

\newtheorem{theorem}{Theorem}[section]
\newtheorem{lemma}[theorem]{Lemma}
\newtheorem{corollary}[theorem]{Corollary}

\theoremstyle{definition}

\newtheorem{example}[theorem]{Example}

\theoremstyle{remark}
\newtheorem{remark}{Remark}
\numberwithin{equation}{section}

\newcommand{\f}{\frac}
\newcommand{\df}{\dfrac}
\newcommand{\p}{\partial}

\newcommand{\mal}{\mathcal}
\newcommand{\De}{\Delta}
\newcommand{\de}{\delta}

\newcommand{\ovl}{\overline}

\begin{document}
	\begin{frontmatter}
		\title{A linear, mass-conserving, multi-time-step compact block-centered finite difference method for incompressible miscible displacement problem in porous media}
		
		\author[ouc]{Xiaoying Wang} \ead{wxy7121@stu.ouc.edu.cn}
            \author[sdu]{Hongxing Rui} \ead{hxrui@sdu.edu.cn}
		\author[ouc,lab]{Hongfei Fu\corref{Fu}}\ead{fhf@ouc.edu.cn}
		
		\address[ouc]{School of Mathematical Sciences, Ocean University of China, Qingdao, Shandong 266100, China}
            \address[sdu]{School of Mathematics, Shandong University, Jinan, Shandong 250100, China}
		\address[lab]{Laboratory of Marine Mathematics, Ocean University of China, Qingdao, Shandong 266100, China}
		\cortext[Fu]{Corresponding author.}
		
\begin{abstract}
  In this paper, a two-dimensional incompressible miscible displacement model is considered, and a novel decoupled and linearized high-order finite difference scheme is developed, by utilizing the multi-time-step strategy to treat the different time evolutions of concentration and velocity/pressure, and the compact block-centered finite difference approximation for spatial discretization.
  We show that the scheme is mass-conserving, and has second-order temporal accuracy and fourth-order spatial accuracy for the concentration $c$, the velocity $\bm{u}$ and the pressure $p$ simultaneously. The existence and uniqueness of the developed scheme under a rough time-step condition is also proved following the convergence results. Numerical experiments are presented to confirm the theoretical conclusions. Besides, some ’real’ simulations are also tested to show good performance of the proposed scheme, in particular, the viscous fingering phenomenon is verified. 
\end{abstract}
  \begin{keyword}
	Incompressible miscible displacement \sep Mass-preserving  \sep  Compact BCFD method \sep Multi-time-step \sep Error analysis  \sep Numerical experiments
			
           
   \end{keyword}
\end{frontmatter}
	
  \section{Introduction}\label{sec:intro}
    The miscible displacement model characterizes the process of one fluid displacing another within a porous media, and has significant applications in fields such as oil reservoir engineering and groundwater contamination remediation \cite{Ewing83,Bear72}.  This model effectively captures the dynamic interactions between the invading fluid (e.g., injected solvents or contaminants) and the resident fluid (e.g., oil or groundwater), driven by pressure gradients and fluid mixing processes.
    
    The mass conservation of the fluid mixture, combined with the incompressibility condition, Darcy's law, and the mass conservation of the invading fluid, leads to a coupled nonlinear system of partial differential equations describing the fluid pressure $p$, velocity $\bm{u}=(u^{x},u^{y})^\top$, and concentration $c$ of the invading fluid within the porous media \cite{Bear72}. Specifically, the model is formulated as follows:
\begin{equation}\label{mod:PDE}
    \left\{
		\begin{aligned}
			&\phi(\bm{x}) \p_{t}c  + \nabla \cdot ( \bm{u}c - \bm{D} \nabla c )  = \tilde{c}q, \quad  &&\text{in} ~ \Omega \times (0,T],\\
			& \bm{u} = - \f{k(\bm{x})}{\mu(c)} \nabla p, \quad  &&\text{in} ~ \Omega \times [0,T],\\
			& \nabla \cdot \bm{u} = q, \quad  &&\text{in} ~ \Omega \times [0,T],
		\end{aligned}
        \right.
\end{equation} 
where $\bar{\Omega} := [x_{L},x_{R}]\times [y_{L},y_{R}]$ with its boundary denoted by $\partial \Omega$, $\Omega=\bar{\Omega}\setminus \partial \Omega$. 
Model \eqref{mod:PDE} is enclosed with periodic or no-flow/homogeneous Neumann boundary conditions
\begin{equation}\label{mod:bc}
         \bm{u} \cdot \bm{n} = 0, \quad ( \bm{u}c - \bm{D} \nabla c ) \cdot \bm{n} = 0,\quad  \text{on} ~ \partial \Omega \times [0,T],
    \end{equation}
and the initial condition is given as 
    \begin{equation}\label{mod:initial}
        c(\bm{x},0) = c^{o}(\bm{x}),\quad \text{in} ~ \bar{\Omega}.
    \end{equation}     
The specific physical parameters and their meanings are as follows \cite{Russell84,Douglas83-1,Russell85,WH00,WH08}: 
    \begin{itemize}
        \item $\phi(\bm{x})$, the porosity of the porous media; 
        
        \item $k(\bm{x})$, the permeability coefficient of the porous media;
        
        \item $\mu(c)$, the concentration-dependent viscosity, which can be experimentally determined. A commonly adopted relationship for this parameter is the quarter-power rule:
    \begin{equation}\label{mod:mu}
        \mu(c) = \mu_0 \left[\mal M^{1/4} c + (1 - c) \right]^{-4}.
    \end{equation}
    Here the mobility ratio $\mal M := \mu_0 / \mu_s$, where $\mu_s$ denotes the viscosity of either the invading fluid/solvent in petroleum recovery processes, or the solute/solvent in subsurface contaminant transport scenarios, and $\mu_0$ represents the viscosity of the resident fluid (typically crude oil in petroleum reservoir simulations);
    
    \item $q$,  the external flow rates, commonly a linear combination of production $q_{P}$ ($ < 0 $) and injection $q_{I}$ ($ > 0 $), i.e., $q = q_{I} + q_{P}$, and $\tilde{c} = c_{I}$ is specified at injection well and $\tilde{c}=c$ at production well. Compatibility requires that 
    $
		\int_{\Omega} q(\bm{x},t) {\rm d} \bm{x} = 0$ for all $t\in[0,T]$;
    
    \item $\bm{D}$, the well-known Bear--Scheidegger diffusion-dispersion tensor which has the form
     \begin{equation}\label{mod:diffu}
         \bm{D}(\bm{x},\bm{u}) = \phi(\bm{x}) \left( \alpha_{m} \bm{I} + \alpha_{l}|\bm{u}| \bm{E}(\bm{u}) +  \alpha_{t} |\bm{u}| \bm{E}^{\perp}(\bm{u}) \right),
     \end{equation}
     where $\bm{E}(\bm{u}) := \frac{\bm{u}\bm{u}^\top}{|\bm{u}|^2}$ is a $2 \times 2$ tensor representing the orthogonal projection along the velocity vector, $|\bm{u}|^{2} = \bm{u} \cdot \bm{u}$ is the Euclidean norm of $\bm{u}$, and $\bm{E}^{\perp} = \bm{I} - \bm{E}$ is the orthogonal complement of $\bm{E}(\bm{u})$. The parameters $\alpha_{m}$, $\alpha_{l}$ and $\alpha_{t}$ are the coefficients of molecular diffusion, longitudinal and transverse dispersion, respectively. For clarity of presentation, the argument $\bm{x}$ is usually omitted below, i.e., we simply denote $\bm{D}(\bm{u})$ to represent $\bm{D}(\bm{x},\bm{u})$.
    \end{itemize}
    
   It is well-known that the analytical solutions for the coupled nonlinear systems are usually not available. Therefore, the challenge faced by the computational mathematics community is to design effective and efficient numerical methods to provide highly accurate numerical solutions. Over the past few decades, a variety of numerical methods and theoretical analysis of the miscible displacement model have been extensively studied, including the Galerkin finite element (FE) method \cite{Ewing80,Russell85,SunWu19}, mixed finite element (MFE) method \cite{Douglas83-1,Douglas83-2,Russell84,LBY13,LBY14,Rui13,CYP16}, finite difference (FD) method \cite{Rui20,LA22,LA24}, discontinuous Galerkin (DG) method \cite{Beatrice11,LR15,CXY19,FGKY22,KXY24} and so on. Notably, a critical observation in the modeling of incompressible miscible displacement problem \eqref{mod:PDE} is that the primary physical quantity of interest is the concentration field $c$, while pressure $p$ only indirectly influences the dynamics through the velocity field $\bm{u}$.  
   Therefore, numerical schemes that directly approximate the velocity $\bm{u}$ rather than relying on the computations of the pressure $p$ first and the velocity $\bm{u}$ second are theoretically and computationally more suitable for such problems.  
    
    Fortunately, the classical MFE method offers a significant advantage in accurate approximations of both pressure $p$ and velocity $\bm{u}$ simultaneously, and it has attracted considerable attention over the past few decades. For example, Douglas et al. \cite{Douglas83-1,Douglas83-2} proposed a Galerkin-MFE hybrid scheme that approximates the concentration with Galerkin method, while approximates the pressure and velocity with MFE method. This approach 
    effectively reduces the numerical difficulties arising from the differentiation of pressure $p$ multiplied by a rough coefficient to calculate velocity $\bm{u}$ \cite{Russell83}. Building upon this foundation, plenty of numerical schemes that combines the MFE method have been developed, such as the MMOC-Galerkin-MFE method \cite{Russell84,LBY14} and ELLAM-MFE method \cite{WH00,WH08}. However, the classical MFE method still requires solving a saddle-point problem, with requirement of special choices of finite element space pair, which may cause serious computational cost for large-scale modeling and simulations. 
    
    The block-centered finite difference (BCFD) method, alternatively referred to as the cell-centered finite difference method, can also be viewed as an MFE method with special numerical quadrature \cite{Arbogast97}. It provides second-order accurate approximations for both the primal variable and its associated flux simultaneously \cite{Weiser88}. Significantly, different from the classical MFE method \cite{Raviart77}, the BCFD method inherently ensures local mass conservation and yields a symmetric positive definite linear system. These advantages have established BCFD method as an efficient and widely adopted numerical approach for solving various partial differential equations.  For instance, Darcy-Forchheimer flow \cite{forch12,forch18}, miscible displacement problems \cite{Rui20,LA22,LA24}, wormhole propagation \cite{worm18}, Keller-Segel Chemotaxis system \cite{Xu25}, two-phase incompressible flow \cite{Kou22} and compressible gas flow \cite{Kou23}.
    In particular, the BCFD method combined with the backward Euler temporal discretization was developed for incompressible Darcy-Forchheimer miscible displacement problems in \cite{Rui20}, and subsequently, the authors \cite{LA22} proposed a MMOC-BCFD algorithm for compressible models through integration with the modified method of characteristics (MMOC) method. Recently, they also developed a two-grid MMOC-BCFD algorithm for slightly compressible model \cite{LA24}. Both methods were proved to achieve first-order temporal accuracy and second-order spatial accuracy simultaneously for pressure, velocity, and concentration on nonuniform grids. However, all the mentioned BCFD methods \cite{Arbogast97,Weiser88,forch12,forch18,worm18,Xu22,Xu25,Rui20,LA22,LA24} have only second-order accuracy in space! 
   Very recently, there have been some attempts at the construction of fourth-order compact BCFD (abbreviated as CBCFD) methods. Shi et al. \cite{Shi21} proposed a CBCFD method for linear elliptic and parabolic problems, and proved that fourth-order spatial convergence can be achieved for both the primary scalar variable and its flux simultaneously under periodic boundary conditions. Subsequently, such kinds of methods have also been developed for nonlinear contaminant transport equations with adsorption \cite{Shi22}  
   and semilinear Sobolev equation \cite{Wang24}. As mentioned in \cite{KXY24}, high-order numerical methods are preferable to reduce numerical artifacts and mesh dependence in simulating the viscous fingering phenomenon, which is sensitive to mesh orientation and numerical discretization. Therefore, it is of great significance to design high-order accurate BCFD schemes for the incompressible miscible displacement model.  However, the construction and analysis of a high-order accurate BCFD algorithm for model \eqref{mod:PDE} have not been reported. 

  Another key feature of the miscible displacement model is that the velocity/pressure usually evolves less rapidly in time than the concentration in practice \cite{Douglas83-2}.  Therefore, it is reasonable to use a relatively larger time step for approximations of the velocity/pressure than for the concentration. For the first time, Douglas et al. \cite{Douglas83-2} proposed the multi-time-step strategy with small time steps for concentration and large time steps for velocity/pressure, combined the backward Euler time discretization and MFE approximation for the miscible displacement in porous media. Optimal-order error estimates were obtained under a time-step condition
    \begin{equation}\label{cond:1}
        \De t_{c} = o(h),
    \end{equation}
    where $\De t_{c}$ represents the concentration time step and $h$ is the spatial mesh size. Subsequently, the multi-time-step algorithm was also employed combined with the MMOC-Galerkin-Galerkin method \cite{Russell85}, MMOC-Galerkin-MFE method \cite{Russell84}, ELLAM-MFE method \cite{WH00}. In order to obtain optimal-order error estimates, the restriction condition \eqref{cond:1} is also required in \cite{Douglas83-2,Russell85,Russell84,Rui13}, while, Wang \cite{WH08} weakened this condition to $\De t_{c} = \mal{O}(h)$. However, only first-order temporal accuracy was achieved for the concentration in \cite{Douglas83-2,Russell85,Russell84,WH00,Rui13,WH08}.  This inspires us to consider a multi-time-step high-order discretization method for the concentration, and to reduce the requirement of the time-step restriction \eqref{cond:1}. To the best of our knowledge, this appears to be the first paper on the multi-time-step CBCFD method with second-order temporal accuracy and fourth-order spatial accuracy for incompressible miscible displacement problem \eqref{mod:PDE} in porous media. In summary, our new scheme enjoys the following remarkable advantages:
    \begin{itemize}
        \item[-] Mass conservation: The proposed scheme can ensure the mass conservation law, thereby can accurately enhance the physical realism of the simulations.

        \item[-] High efficiency: The decoupling of concentration and velocity/pressure can greatly improve the efficiency of numerical simulation; and the multi-time-step strategy and high-order method can further reduce the computational cost.

        \item[-] Optimal-order convergence: A significant advantage of the scheme is its ability to achieve second-order temporal and fourth-order spatial approximations for the concentration $c$, velocity $\bm{u}$ and pressure $p$, eliminating the order reduction in accuracy caused by the differentiation of pressure $p$ when approximating the velocity in the traditional direct finite difference methods. Rigorous theoretical analysis including the error estimates and unique solvability of the developed scheme are carried out.

        \item[-] 'Real' simulations: The method can accurately simulate the standard five-spot well pattern of oil-water system in porous media. In particular, the viscous fingering phenomenon is observed. 
    \end{itemize}

The outline of the paper is organized as follows. In Section \ref{sec:scm}, the decoupled, linear, multi-time-step high-order compact scheme is developed, and we show the discrete mass conservation of the proposed scheme. In Section \ref{sec:analysis}, a rigorous optimal-order error analysis is discussed for all variables, as well as the existence and uniqueness of numerical solutions.
In Section \ref{sec:nx}, ample numerical experiments are carried out to show the accuracy, high efficiency and mass conservation. In addition, the reliability of the proposed scheme in 'real' simulations is also tested. Concluding remarks are given in Section \ref{sec:conclusion}. Throughout the paper, we use $K$ to denote a generic positive constant, which may have different values at different occurrences.
	
\section{Mass-conserving multi-time-step numerical scheme}\label{sec:scm}
In this section, we shall develop a novel linear, mass-conserving, and efficient multi-time-step high-order finite difference scheme for model \eqref{mod:PDE} under periodic boundary conditions. 

To avoid the inversion of the diffusion coefficient $\bm{D}(\bm{u})$, we adopt the idea of expanded MFE method to introduce two auxiliary variables $\bm{v}=(v^{x}, v^{y})^\top := - \nabla c$ and $\bm{w} =(w^{x}, w^{y})^\top:= \bm{u}c + \bm{D}(\bm{u})\bm{v} $. Let $a(c) := \mu(c)/k(\bm{x})$. Besides, for $\bm{D}$ defined by \eqref{mod:diffu}, we write $\bm{D}(\bm{u})=(\bm{D}^{x}(\bm{u}), \bm{D}^{y}(\bm{u}))^\top$ with $\bm{D}^{x} := (D_{11}(\bm{u}), D_{12}(\bm{u}) )^\top$ and $\bm{D}^{y} := ( D_{21}(\bm{u}), D_{22}(\bm{u}) )^\top$, where $D_{ij}(\bm{u})$ is the element of $\bm{D}(\bm{u})$. Then, the nonlinear system \eqref{mod:PDE} can be rewritten as
	\begin{equation}\label{mod:2D}
		\left\{
		\begin{aligned}
			& \phi \p_{t}c + \p_{x}w^{x} + \p_{y} w^{y}- q_{P}c  = c_{I}q_{I}, &\qquad \text{in} ~ \Omega \times (0,T],\\
            & v^{x} + \p_{x} c =0,~ v^{y} + \p_{y} c =0, &  \qquad \text{in} ~ \Omega \times (0,T], \\
			& w^{x} - u^{x}c - \bm{D}^{x}(\bm{u}) \cdot \bm{v}=0,~ w^{y} - u^{y}c -\bm{D}^{y}(\bm{u}) \cdot \bm{v}=0, & \qquad \text{in} ~ \Omega \times (0,T],\\
			& a(c) u^{x} + \p_{x} p = 0, ~ a(c) u^{y} + \p_{y} p = 0, & \qquad \text{in} ~ \Omega \times [0,T],\\
			& \p_{x} u^{x} + \p_{y} u^{y} = q, & \qquad \text{in} ~ \Omega \times [0,T].
		\end{aligned}
		\right.
	\end{equation}
Furthermore, we suppose that the following assumptions hold:
\begin{itemize}
    \item[(A1)] The exact solutions satisfy
    $$
      \begin{aligned}
          & c \in W^{1,\infty}([0,T] \times\Omega) \cap L^{\infty}(0,T;H^{6}(\Omega)) \cap H^{3}(0,T;L^{\infty}(\Omega)) \cap H^{2}(0,T;W^{1,\infty}(\Omega)),\\
	     & \bm{u} \in  H^{2}(0,T;L^{\infty}(\Omega)) \cap L^{\infty}(0,T;H^{5}(\Omega)),\quad  
           p \in  L^{\infty}(0,T;H^{5}(\Omega)).
      \end{aligned}
    $$
    Besides, there exists a positive constant $K_{1}$ such that
    $$
       \|c\|_{L^{\infty}([0,T]\times \Omega)} + \|\p_{t} c\|_{L^{\infty}([0,T]\times \Omega)} + \|\bm{u}\|_{L^{\infty}([0,T]\times \Omega)} \leq K_{1}.
    $$
    \item [(A2)] There exist positive constants $\phi_{*}$, $\phi^{*}$, $\mu_{*}$, $\mu^{*}$, $k_{*}$, $k^{*}$ and $K_{2}$, such that
	$$
    0 < \phi_{*} \leq \phi(\bm{x}) \leq \phi^{*}, ~
	0 < \mu_{*} \leq \mu(c) \leq \mu^{*}, ~
	0 < k_{*} \leq k(\bm{x}) \leq k^{*},
    $$
    $$
	\max\{\|\mu'\|_{L^{\infty}(\Omega)}, \|q_{P} \|_{L^{\infty}([0,T]\times \Omega)} \} \leq K_{2}.
	$$
\end{itemize}	    

	
	\subsection{Notations}
In the following, we present some basic notations that used in the paper. 
First, we introduce two different temporal girds for the velocity and concentration fields, respectively. Let $N_{p}$ and $N_{c}$ be two positive integers such that $N_{c}=Q N_{p}$ for some positive integer $Q~ (\ge 1)$. The time interval $[0,T]$ is then partitioned for the pressure by $t_{p}^{m} = m \De t_{p}$ with the pressure time step $\De t_{p}=T/N_{p}$, and for the concentration by $t_{c}^{n}=n\De t_{c}$ with the concentration time step $ \De t_{c} =\De t_{p}/Q$. The temporal grids $t_{c}^{n}$ and $t_{p}^{m}$ are illustrated in Figure \ref{fig:grid:time}.
  For temporal grid functions 
  $\left\lbrace \theta_{c}^{n}=\theta(t_{c}^{n})\right\rbrace_{n \ge 0} $, define
	$$
    d_t\theta_{c}^{n+1/2} :=\df{\theta_{c}^{n+1}-\theta_{c}^{n}}{\De t_{c}}, \quad \ovl{\theta}_{c}^{n+1/2} :=\df{\theta_{c}^{n+1}+\theta_c^{n}}{2}.
	$$
    
	Second, let $ N_x $ and $N_y$ be the number of grids along the $x$ and $y$ coordinates, respectively. We introduce staggered spatial grids $\{ \Pi_{x}, \Pi_{x}^{*}\}$ and $\{ \Pi_{y}, \Pi_{y}^{*}\}$ along $x$ and $y$ directions, respectively, by
$$
  \begin{aligned}
    \Pi_{x} : x_{i+1/2} = x_{L}+ih^{x}, \ i = 0,\cdots, N_{x}, \quad 
        \Pi_{x}^{*}: x_{i} = \f{ x_{i+1/2} +x_{i-1/2} }{2}, \ i = 1,\cdots,N_{x},\\
    \Pi_{y} : y_{j+1/2} = y_{L}+jh^{x}, \ j = 0,\cdots, N_{y}, \quad 
         \Pi_{y}^{*}: y_{j} = \f{ y_{j+1/2} + y_{j-1/2} }{2}, \ j = 1,\cdots,N_{y},
   \end{aligned}
$$
with spatial mesh sizes $h^{x} = (x_{R}-x_{L})/N_{x}$ and $h^{y} = (y_{R}-y_{L})/N_{y}$. Let $h:= \max\{ h^{x}, h^{y}\}$.
The three sets of staggered grids $\Pi_{x}^{*} \times \Pi_{y}^{*}$, $\Pi_{x}\times \Pi_{y}^{*}$ and $\Pi_{x}^{*} \times \Pi_{y}$ are shown in Figure \ref{fig:grid:space}, where the concentration $c$ and pressure $p$ are both approximated over $\Pi_{x}^{*} \times \Pi_{y}^{*}$, while the velocity component $u^{x}$ and flux $v^{x}$ in the $x$-direction, and $u^{y}$ and $v^{y}$ in the $y$-direction, are solved on grids $\Pi_{x} \times \Pi_{y}^{*}$ and $\Pi_{x}^{*} \times \Pi_{y}$, respectively.  
    \begin{figure}[!htbp]
    \vspace{-5 mm}
    \centering
    	\subfigure[The distributions of $t_{p}^{m}$ and $t_{c}^{n}$ with $Q=5$]
	{ 
		\includegraphics[width=0.45\textwidth]{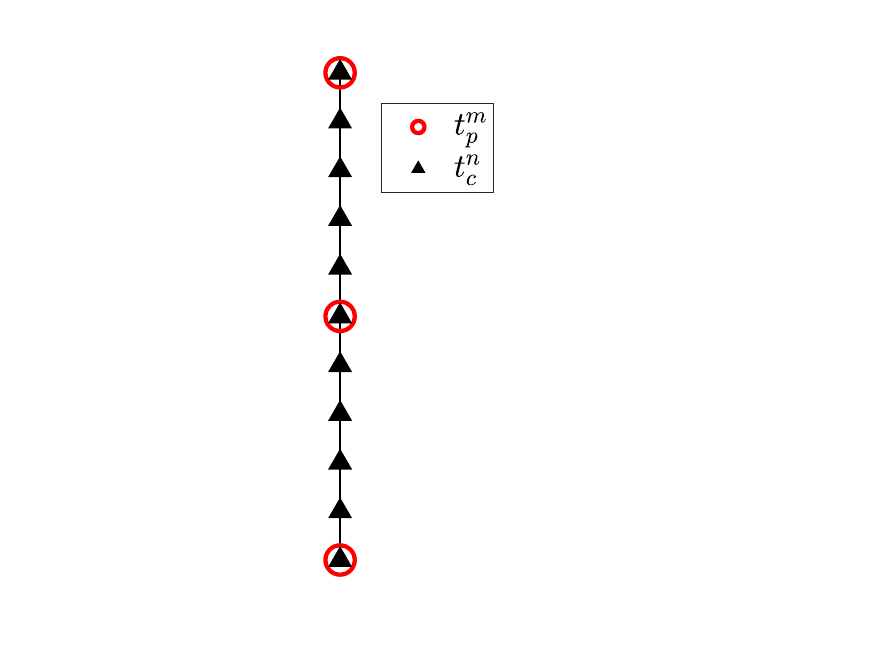}
          \label{fig:grid:time}
		}
        \subfigure[The distributions of staggered grids]
    	{ 
		\includegraphics[width=0.45\textwidth]{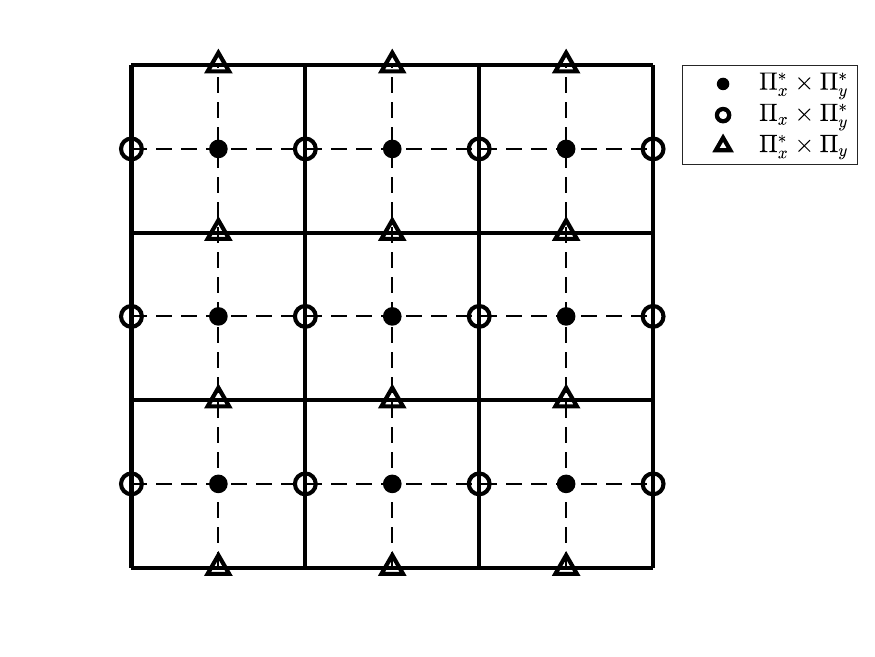}
          \label{fig:grid:space}
		}  
    \caption{The schematic illustration of temporal and spatial grids.}
   \end{figure}

Besides,  define the spaces of grid functions with periodic boundary conditions
	$$
	\begin{aligned}
		&\mal{P} := \left\lbrace g \mid g=\{g_{i,j}\}, (x_{i},y_{j})\in \Pi_{x}^{*}\times \Pi_{y}^{*},~ 
        ~ \text{and}~  g ~ \text{is periodic} \right\rbrace ,\\
		&\mal{U} := \left\lbrace g  \mid g = \{g_{i+1/2,j}\},  (x_{i+1/2},y_{j})\in \Pi_{x}\times \Pi_{y}^{*},~ 
        ~ \text{and}~  g ~ \text{is periodic}\right\rbrace,\\
		&\mal{V} := \left\lbrace g  \mid  g= \{g_{i,j+1/2}\}, (x_{i},y_{j+1/2})\in \Pi_{x}^{*}\times \Pi_{y},~
        ~ \text{and}~  g ~ \text{is periodic} \right\rbrace.
	\end{aligned}
	$$
Furthermore, for $\bm{\xi}=(\xi^x,\xi^y) \in \mal{U}\times \mal{V}$,  $ \bm{\nu}=(\nu^x, \nu^y)\in \mal{U} \times \mal{V}$ and $\omega, \sigma \in \mal{P}$, we introduce the following discrete $L^2$ inner products
	\begin{equation*}
		\begin{aligned}
			&   (\xi^x,\nu^x)_{x}:= \sum_{i=1}^{N_{x}}\sum_{j=1}^{N_{y}} h^{x} h^{y}\, \xi^x_{i+1/2,j}\, \nu^x_{i+1/2,j},
			&& (\xi^y,\nu^y)_{y}:=\sum_{i=1}^{N_{x}}\sum_{j=1}^{N_{y}} h^{x} h^{y}\, \xi^y_{i,j+1/2}\, \nu^y_{i,j+1/2},
			\\
			&     (\bm{\xi},\bm{\nu})_{\rm T}:=(\xi^x,\nu^x)_{x}+(\xi^y,\nu^y)_{y},      
            && (\omega,\sigma)_{\rm M}:=\sum_{i=1}^{N_{x}}\sum_{j=1}^{N_{y}}h^{x} h^{y}\, \omega_{i,j}\, \sigma_{i,j},
		\end{aligned}
	\end{equation*}
	and  corresponding discrete norms $\| \cdot \|_{\ell} = \sqrt{(\cdot,\cdot)_{\ell}}$ for $\ell = x, y, \rm{T}, \rm{\rm{M}}$. 
Moreover, define the following spatial difference operators
	$$
	[\de_x\omega]_{s,j}:=\df{\omega_{s+1/2,j}-\omega_{s-1/2,j}}{h^{x}}, \quad [\de_x^2\omega]_{s,j}:=\df{\omega_{s+1,j}-2\omega_{s,j}+\omega_{s-1,j}}{(h^{x})^2},\quad  s= i,i+1/2.
	$$
Similarly, the difference operators $\de_y$ and $\de_y^{2}$ can be defined. In addition, we define $|\omega|_{1}^{2}:= \|\de_{x} \omega\|_{x}^{2} + \|\de_{y} \omega\|_{y}^{2}$ and $\|\omega\|_{\infty}:=\max_{i,j} |\omega_{i,j}|$. We also introduce the following fourth-order difference operators with three-point stencils
	\begin{equation}\label{op:L}
		\mal{L}_{x} :=  \mal{I} + \f{(h^{x})^{2}}{24} \delta_x^2,\quad
		\mal{L}_{y} :=  \mal{I} + \f{(h^{y})^{2}}{24} \delta_y^2,   
	\end{equation}
	where $\mal{I}$ is the identity operator.
	For simplicity, denote $\mal{L} := \mal{L}_{x} \mal{L}_{y}$. 

    \begin{figure}[!htbp]
    \vspace{-5pt}
    \centering
    	\subfigure[Nodal distributions of $\mal{T}_{x}$ and $\mal{T}_{x}^{*}$ ]
	{ 
		\includegraphics[width=0.45\textwidth]{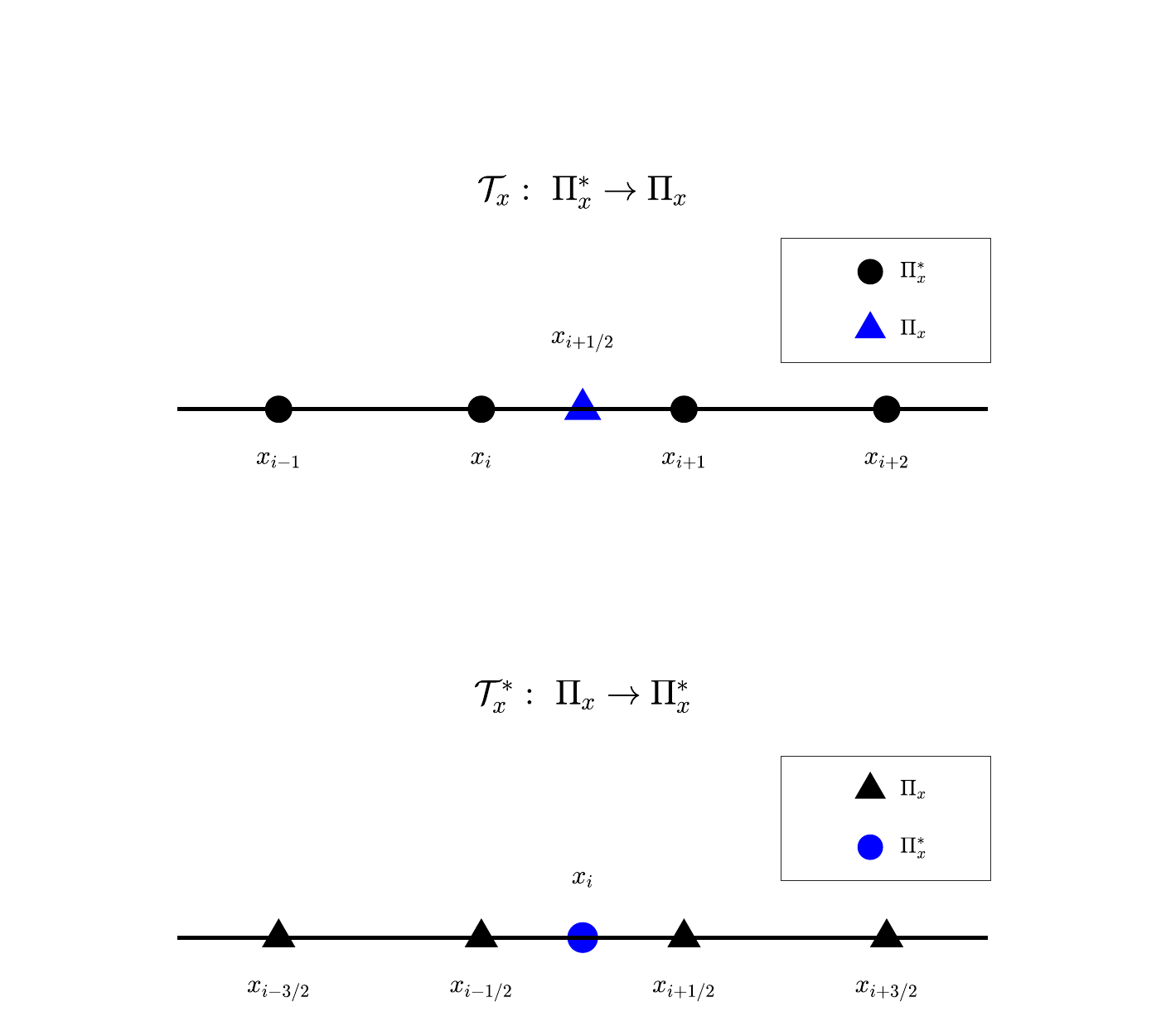}
          \label{fig:Tx}
		}
        \subfigure[Nodal distributions of $\mal{H}_{x}$]
    	{ 
		\includegraphics[width=0.45\textwidth]{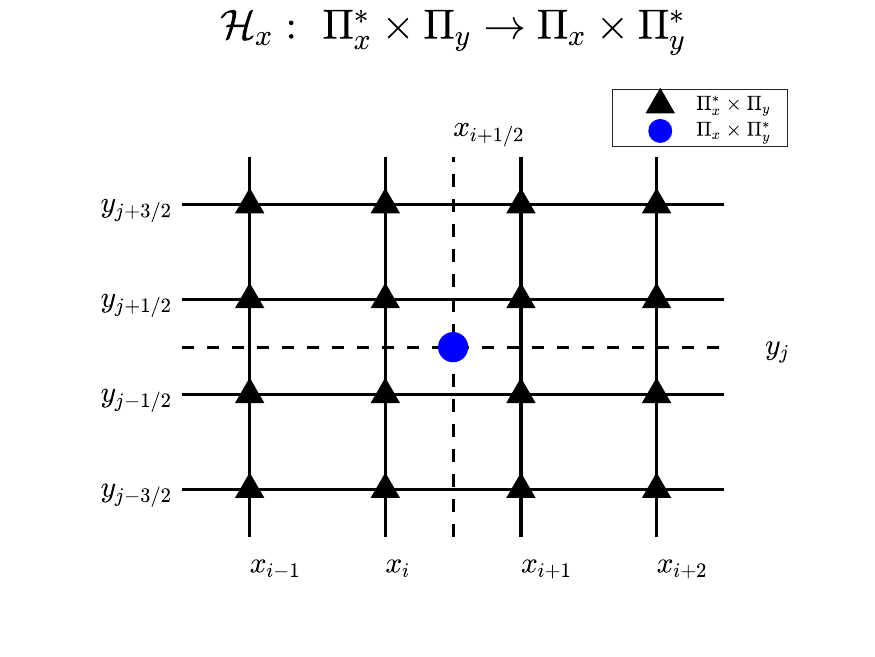}
          \label{fig:Hx}
		}  
    \caption{The schematic illustrations of interpolation operators $\mal{T}_{x}$, $\mal{T}_{x}^{*}$ and $\mal{H}_{x}$.}
   \end{figure}
Finally, we introduce some local interpolation operators. First, define the local cubic Lagrange interpolation operators $\mal{T}_{x}$ from $\Pi_{x}^{*}$ to $\Pi_{x}$ and $\mal{T}_{x}^{*}$ from $\Pi_{x}$ to $\Pi_{x}^{*}$ by
\begin{equation}\label{op:t}
    \begin{aligned}
		[ \mal{T}_{x} \omega ]_{i+1/2} &:= \f{-\omega_{i-1} + 9\omega_{i} + 9\omega_{i+1} - \omega_{i+2} }{16},\\
		[ \mal{T}_{x}^{*} \omega ]_{i} &:= \f{-\omega_{i-3/2} + 9\omega_{i-1/2} + 9\omega_{i+1/2} - \omega_{i+3/2} }{16},
        \end{aligned}
\end{equation}
see Figure \ref{fig:Tx} for the nodal distributions of $\mal{T}_{x}$ and $\mal{T}_{x}^{*}$. Similarly, the local cubic Lagrange interpolation operators $\mal{T}_{y}$ from $\Pi_{y}^{*}$ to $\Pi_{y}$ and $\mal{T}_{y}^{*}$ from $\Pi_{y}$ to $\Pi_{y}^{*}$ can be defined.  Second, with the introduced directional interpolation operators, we can define the local bicubic Lagrange interpolation operators $\mal{H}_{x}$ from $\Pi_{x}^{*} \times \Pi_{y}$  to $\Pi_{x} \times \Pi_{y}^{*}$ 
by
    \begin{equation}\label{op:h}
            \mal{H}_{x} := \mal{T}_{y}^{*}  \mal{T}_{x}  =   \mal{T}_{x} \mal{T}_{y}^{*},
	\end{equation}
see Figure \ref{fig:Hx} for the nodal distributions of $\mal{H}_{x}$. Similarly, one can define $\mal{H}_{y} := \mal{T}_{x}^{*}  \mal{T}_{y}  =   \mal{T}_{y} \mal{T}_{x}^{*}$ from $\Pi_{x} \times \Pi_{y}^{*}$  to $\Pi_{x}^{*} \times \Pi_{y}$.
\subsection{The MC-MTS-CBCFD scheme and its mass conservation}
In this subsection, we construct the linear, mass-conserving, multi-time-step CBCFD (MC-MTS-CBCFD) scheme to approximate the velocity $\bm{u}$/pressure $p$ and the concentration $c$ to a same fourth-order spatial accuracy, in which the Crank-Nicolson method and the multi-time-step method are used in time discretization, and the CBCFD method is considered in space discretization on staggered spatial grids.

Let us denote the difference approximations of $\{ p(t_{p}^{m}), u^{x}(t_{p}^{m}), u^{y}(t_{p}^{m})\}$ at time $t=t_{p}^{m}$ and $\{ c(t_{c}^{n}), v^{x}(t_{c}^{n}),  v^{y}(t_{c}^{n})\}$ at time $t=t_{c}^{n}$ by $\{ P^{m}, U^{x,m}, U^{y,m}\} \in  \mal{P} \times \mal{U} \times\mal{V} $ and $\{ C^{n}, V^{x,n}, V^{y,n}\} \in  \mal{P} \times \mal{U} \times\mal{V} $, respectively. Note that the main difficulty in the construction of the MC-MTS-CBCFD scheme lies in the discretization of the nonlinear convective term $\nabla \cdot ( \bm{u}c)$ and tensor-form nonlinear diffusion coefficient $\bm{D}(\bm{u})$, as concentration $c$ and velocity components $u^{x}$ and $u^{y}$ in two directions are approximated on staggered spatial grids $\Pi_{x}^{*} \times \Pi_{y}^{*}$, $\Pi_{x} \times \Pi_{y}^{*}$ and $\Pi_{x}^{*} \times \Pi_{y}$, respectively.
Thus, we shall use the local cubic Lagrange interpolation operator $\mal{T}_{x}$ (resp. $\mal{T}_{y}$) defined by \eqref{op:t} to make $c$  and $u^{x}$ (resp. $u^{y}$) match on the same grid points $\Pi_{x} \times \Pi_{y}^{*}$ (resp. $\Pi_{x}^{*} \times \Pi_{y}$). Besides, we also utilize the local bicubic Lagrange interpolation operators $\mal{H}_{x}$ (resp. $\mal{H}_{y}$) defined by \eqref{op:h} to make the velocity components $u^{y}$ (resp. $u^{x}$) and $u^{x}$ (resp. $u^{y}$) match on the same grid points $\Pi_{x} \times \Pi_{y}^{*}$ (resp. $\Pi_{x}^{*} \times \Pi_{y}$). Then, the fully discrete fourth-order MC-MTS-CBCFD scheme  for problem \eqref{mod:2D} is proposed as follows:

\textbf{Step 1.} For $m \ge 0$, solve $\{P^{m},U^{x,m},U^{y,m}\} \in  \mal{P} \times \mal{U} \times\mal{V} $  at pressure time $t=t_{p}^{m}$ by  
    \begin{align}
         & \big[\mal{L}_{y} \de_{x} U^{x} + \mal{L}_{x} \de_{y} U^{y} \big]_{i,j}^{m} = \big[ \mal{L} q\big]_{i,j}^{m}, 
               & \text{on}~ \Pi_{x}^{*}\times \Pi_{y}^{*}, \label{scm:pu:1}\\
	&\big[ \mal{L}_{x}[ a(\mal{T}_{x}C_{p}) U^{x}  ] + \de_{x} P  \big]_{i+1/2,j}^{m} = 0, 
               & \text{on}~ \Pi_{x}\times \Pi_{y}^{*}, \label{scm:pu:2}\\
	&\big[ \mal{L}_{y}[ a(\mal{T}_{y}C_{p}) U^{y}  ] + \de_{y} P  \big]_{i,j+1/2}^{m} = 0, 
               & \text{on}~ \Pi_{x}^{*}\times \Pi_{y}, \label{scm:pu:3}
    \end{align}
hereafter, the subscript $p$ in $C_p^m$ denotes concentration at pressure time $t_{p}^{m}$, taking value at concentration time $t_{c}^{n}$, i.e., $C_p^m=C^n$ for $n=Qm$.

\textbf{Step 2.} Solve the intermediate solutions $\{C_{p}^{1,*},V_{p}^{x,1,*},V_{p}^{y,1,*},W_{p}^{x,1,*},W_{p}^{y,1,*} \} \in  \mal{P} \times  \mal{U} \times\mal{V}\times  \mal{U} \times\mal{V} $ at pressure time $t=t_{p}^{1}=t_{c}^{Q}$ via the Crank-Nicolson scheme
      \begin{align}
            &\Big[\mal{L}[\phi D_{t} C_{p}] + \mal{L}_{y} \de_{x} \ovl{W}_{p}^x + \mal{L}_{x} \de_{y} \ovl{W}_{p}^y - \mal{L} [q_{P} \ovl{C}_{p}]\Big]_{i,j}^{1/2,*}  = \mal{L} [c_{I}q_{I}]_{i,j}^{1/2},  & \text{on}~ \Pi_{x}^{*}\times \Pi_{y}^{*}, \label{scm:cu:p:1} \\
            & \Big[\mal{L}_{x} V_{p}^{x} + \de_{x} C_{p} \Big]_{i+1/2,j}^{1,*} = 0,
            & \text{on}~ \Pi_{x} \times \Pi_{y}^{*}, \label{scm:cu:p:2} \\
            & \Big[\mal{L}_{y} V_{p}^{y} + \de_{y} C_{p} \Big]_{i,j+1/2}^{1,*} = 0,
            & \text{on}~ \Pi_{x}^{*}\times \Pi_{y}, \label{scm:cu:p:3} \\
            & \Big[ W_{p}^{x} - U_{\#}^{x} \mal{T}_{x} C_{p} -  \bm{D}^{x}( U_{\#}^{x},\mal{H}_{x} U_{\#}^{y} ) \cdot ( V_{p}^{x}, \mal{H}_{x} V_{p}^{y} )  \Big]_{i+1/2,j}^{1,*} = 0,
            & \text{on}~ \Pi_{x}\times \Pi_{y}^{*}, \label{scm:cu:p:4} \\
            & \Big[ W_{p}^{y} - U_{\#}^{y} \mal{T}_{y} C_{p} -  \bm{D}^{y}( \mal{H}_{y}U_{\#}^{x}, U_{\#}^{y} ) \cdot ( \mal{H}_{y}V_{p}^{x},  V_{p}^{y} )  \Big]_{i,j+1/2}^{1,*} = 0,
            & \text{on}~ \Pi_{x}^{*}\times \Pi_{y}, \label{scm:cu:p:5} 
    \end{align}
    where $\bm{U}_{\#}^{1,*}:=\bm{U}^{0}$, and
	$$
	D_{t} C_{p}^{1/2,*}:= \f{ C_{p}^{1,*} - C^{0} }{\De t_{p}}, \quad 
	\ovl{C}_{p}^{1/2,*}:=\f{ C_{p}^{1,*} + C^{0} }{2},  \quad 
	\ovl{\bm W}_{p}^{1/2,*}:=\f{\bm W_{p}^{1,*} + \bm W^{0} }{2}.
	$$
    And then, solve the intermediate solutions $\{P^{1,*},U^{x,1,*},U^{y,1,*}\} \in  \mal{P} \times \mal{U} \times\mal{V} $ at pressure time $t=t_{p}^{1}$ via  
    \begin{align}
	& \Big[\mal{L}_{y} \de_{x} U^{x} + \mal{L}_{x} \de_{y} U^{y} \Big]_{i,j}^{1,*} = \big[ \mal{L} q \big]_{i,j}^{1}, 
            & \text{on}~ \Pi_{x}^{*}\times \Pi_{y}^{*}, \label{scm:cu:p:6} \\
	&\Big[ \mal{L}_{x}[ a(\mal{T}_{x}C_{p}) U^{x}  ] + \de_{x} P  \Big]_{i+1/2,j}^{1,*} = 0, 
            & \text{on}~ \Pi_{x}\times \Pi_{y}^{*},\label{scm:cu:p:7} \\
	&\Big[ \mal{L}_{y}[ a(\mal{T}_{y}C_{p}) U^{y}  ] + \de_{y} P  \Big]_{i,j+1/2}^{1,*} = 0, 
            & \text{on}~ \Pi_{x}^{*}\times \Pi_{y}. \label{scm:cu:p:8} 
    \end{align}

\textbf{Step 3.} Solve $\{C^{n+1},V^{x,n+1},V^{y,n+1} \} \in  \mal{P} \times \mal{U} \times\mal{V} $  at concentration time $t=t_{c}^{n+1}$ such that $t_{p}^{m} < t_{c}^{n+1} \leq t_{p}^{m+1}$ via the Crank-Nicolson scheme
   \begin{align}
            &\Big[\mal{L}[\phi d_{t} C] + \mal{L}_{y} \de_{x} \ovl{W}^x + \mal{L}_{x} \de_{y} \ovl{W}^y - \mal{L} [q_{P} \ovl{C}]\Big]_{i,j}^{n+1/2}  = \mal{L} [c_{I}q_{I}]_{i,j}^{1/2},  & \text{on}~ \Pi_{x}^{*}\times \Pi_{y}^{*}, \label{scm:cv:1} \\
            & \Big[\mal{L}_{x} V^{x} + \de_{x} C \Big]_{i+1/2,j}^{n+1} = 0,
            & \text{on}~ \Pi_{x} \times \Pi_{y}^{*}, \label{scm:cv:2} \\
            & \Big[\mal{L}_{y} V^{y} + \de_{y} C \Big]_{i,j+1/2}^{n+1} = 0,
            & \text{on}~ \Pi_{x}^{*}\times \Pi_{y}, \label{scm:cv:3} \\
            & \Big[ W^{x} - U_{\#}^{x} \mal{T}_{x} C -  \bm{D}^{x}( U_{\#}^{x},\mal{H}_{x} U_{\#}^{y} ) \cdot ( V^{x}, \mal{H}_{x} V^{y} )  \Big]_{i+1/2,j}^{n+1} = 0,
            & \text{on}~ \Pi_{x}\times \Pi_{y}^{*}, \label{scm:cv:4} \\
            & \Big[ W^{y} - U_{\#}^{y} \mal{T}_{y} C -  \bm{D}^{y}( \mal{H}_{y}U_{\#}^{x}, U_{\#}^{y} ) \cdot ( \mal{H}_{y}V^{x},  V^{y} )  \Big]_{i,j+1/2}^{n+1} = 0,
            & \text{on}~ \Pi_{x}^{*}\times \Pi_{y}, \label{scm:cv:5} 
    \end{align}
    for $Qm \le  n \le Q(m+1)-1$ and some $m \geq 0$, where
	\begin{equation}\label{udl}
		\bm{U}_{\#}^{n+1} := 
		\left\lbrace 
		\begin{aligned}
			&\f{t_{c}^{n+1}-t_{p}^{m-1}}{t_{p}^{m}-t_{p}^{m-1}} \bm{U}^{m} - 
			\f{t_{c}^{n+1}-t_{p}^{m}}{t_{p}^{m}-t_{p}^{m-1}} \bm{U}^{m-1}, & 
			t_{p}^{m} < t_{c}^{n+1} \leq t_{p}^{m+1},\quad m \geq 1,\\
			& \f{t_{c}^{n+1}-t_{p}^{0}}{t_{p}^{1}-t_{p}^{0}} \bm{U}^{1,*} + 
			\f{t_{p}^{1}-t_{c}^{n+1}}{t_{p}^{1}-t_{p}^{0}} \bm{U}^{0},&  t_{p}^{0} \leq t_{c}^{n+1} \leq t_{p}^{1}, \quad m =0. 
		\end{aligned}
		\right.
	\end{equation}
    

Steps 1--3 are enclosed with the following initial conditions
  \begin{equation}\label{ibc}
     \begin{aligned}
         &C_{i,j}^{0} = c^{o}(x_{i},y_{j}), & \quad \text{on}~ \Pi_{x}^{*}\times \Pi_{y}^{*},\\
         &V_{i+1/2,j}^{x,0} = \p_{x} c^{o}(x_{i+1/2},y_{j}), & \quad  \text{on}~ \Pi_{x}\times \Pi_{y}^{*},\\
         &V_{i,j+1/2}^{y,0} = \p_{y} c^{o}(x_{i},y_{j+1/2}), &\quad  \text{on}~ \Pi_{x}^{*}\times \Pi_{y}. 
     \end{aligned}
  \end{equation}
In addition, we enforce
	\begin{equation}\label{comP}
		P^{m}_{1,1} = 0,\quad m \geq 0,
	\end{equation}
	in \eqref{scm:pu:1}--\eqref{scm:pu:3} and \eqref{scm:cu:p:6}--\eqref{scm:cu:p:8} to ensure the uniqueness of $P$.

   \begin{remark} In \eqref{scm:cu:p:1}--\eqref{scm:cu:p:5} of Step 2 and also \eqref{scm:cv:1}--\eqref{scm:cv:5} of Step 3 for $n=0$, the initial auxiliary vector variable $\{ W^{x,0}, W^{y,0}\} \in \mal{U} \times \mal{V}$ are required. It can be calculated via the computed initial velocity $\bm{U}^{0}$ in Step 1 and the initial concentration \eqref{ibc} that
     \begin{align}
           &W_{i+1/2,j}^{x,0} = \big[ U^{x}c +  \bm{D}^{x}( U^{x},\mal{H}_{x}U^{y}) \cdot \nabla c \big]_{i+1/2,j}^{0},  &  \text{on}~ \Pi_{x}\times \Pi_{y}^{*}, \label{scm:w:1} \\
           &W_{i,j+1/2}^{y,0} = \big[ U^{y}c + \bm{D}^{y}( \mal{H}_{y} U^{x},U^{y}) \cdot \nabla c \big]_{i,j+1/2}^{0},  &  \text{on}~ \Pi_{x}^{*}\times \Pi_{y}. \label{scm:w:2} 
    \end{align}
     In practical computations, it can also be calculated by the fourth-order numerical scheme \eqref{scm:cv:4}--\eqref{scm:cv:5}, i.e.,
    \begin{align}
            &\Big[ W^{x} - U_{\#}^{x} \mal{T}_{x} C -  \bm{D}^{x}( U_{\#}^{x},\mal{H}_{x} U_{\#}^{y} ) \cdot ( V^{x}, \mal{H}_{x} V^{y} ) \Big]_{i+1/2,j}^{0} = 0,
            & \text{on}~ \Pi_{x}\times \Pi_{y}^{*}, \label{scm:w:3} \\
			& \Big[ W^{y} - U_{\#}^{y} \mal{T}_{y} C -  \bm{D}^{y}( \mal{H}_{y}U_{\#}^{x}, U_{\#}^{y} ) \cdot ( \mal{H}_{y}V^{x},  V^{y} )  \Big]_{i,j+1/2}^{0} = 0,
            & \text{on}~ \Pi_{x}^{*}\times \Pi_{y}. \label{scm:w:4} 
    \end{align}
   \end{remark} 
\begin{remark} The auxiliary variable $\bm{W}$ in Steps 2--3 only serve as an intermediate quantity, which does not need to be solved. In fact, by substituting \eqref{scm:cv:4}--\eqref{scm:cv:5} into \eqref{scm:cv:1}, Step 3 reduces to linear systems only in terms of $C^{n+1}$ and $\bm{V}^{n+1}$ such that
    \begin{equation}\label{matrix}
         \begin{bmatrix}
       	\mal{A}_{1} &\mal{A}_{2} & \mal{A}_{3}\\
       	\de_{x}  & \mal{L}_{x}  & \bm{0}\\
       	\de_{y}  & \bm{0} &  \mal{L}_{y}
       \end{bmatrix}
       \begin{bmatrix}
       	C^{n+1} \\
       	V^{x,n+1} \\
       	V^{y,n+1} 
       \end{bmatrix}
       = F^{n+1},
    \end{equation}
       where the linear operators
       $$
         \begin{aligned}
             &\mal{A}_{1}\omega:= \f{2}{\De t_{c}}\mal{L} \big[\phi \omega\big]-  \mal{L} \big[ q_{P}^{n+1/2}\omega\big] +   \mal{L}_{y} \de_{x} \big[  U_{\#}^{x,n+1} \mal{T}_{x}\omega\big] +  \mal{L}_{x} \de_{y} \big[  U_{\#}^{y,n+1} \mal{T}_{y}\omega\big],\\
             &\mal{A}_{2}\omega:=   \mal{L}_{y} \de_{x} \big[ D_{11} ( U_{\#}^{x,n+1},\mal{H}_{x} U_{\#}^{y,n+1} )\omega\big] +   \mal{L}_{x} \de_{y} \big[D_{21} ( \mal{H}_{y}U_{\#}^{x,n+1}, U_{\#}^{y,n+1} )\mal{H}_{y}\omega\big], \\
              &\mal{A}_{3}\omega:=   \mal{L}_{y} \de_{x} \big[ D_{12} ( U_{\#}^{x,n+1},\mal{H}_{x} U_{\#}^{y,n+1} ) \mal{H}_{x} \omega \big]+  \mal{L}_{x} \de_{y} \big[ D_{22} ( \mal{H}_{y}U_{\#}^{x,n+1}, U_{\#}^{y,n+1} )\omega\big].
         \end{aligned}
       $$
\end{remark}
\begin{remark}   When only molecular diffusion is considered \cite{Rui20,LA22,LA24,YYR21,GH17,Douglas83-3}, that is, the diffusion tensor is velocity-independent. In this case, it reduces to a linear form $\bm{D}(\bm x)=\alpha_{m}\phi(\bm{x}) \bm{I}=D(\bm x)\bm{I}$. Then, equations \eqref{scm:cu:p:4}--\eqref{scm:cu:p:5} and \eqref{scm:cv:4}--\eqref{scm:cv:5} can now be written into the following simple formulations:
      \begin{align*}
            & \Big[ W_{p}^{x} - U_{\#}^{x} \mal{T}_{x} C_{p} -  D V_{p}^{x}  \Big]_{i+1/2,j}^{1,*} = 0,
            & \text{on}~ \Pi_{x}\times \Pi_{y}^{*},  \\
            & \Big[ W_{p}^{y} - U_{\#}^{y} \mal{T}_{y} C_{p} -  D V_{p}^{y}  \Big]_{i,j+1/2}^{1,*} = 0,
            & \text{on}~ \Pi_{x}^{*}\times \Pi_{y}, 
    \end{align*}
    and
    \begin{align*}
            & \Big[ W^{x} - U_{\#}^{x} \mal{T}_{x} C -  D V^{x}  \Big]_{i+1/2,j}^{n+1} = 0,
            & \text{on}~ \Pi_{x}\times \Pi_{y}^{*},  \\
            & \Big[ W^{y} - U_{\#}^{y} \mal{T}_{y} C -  D V^{y}  \Big]_{i,j+1/2}^{n+1} = 0,
            & \text{on}~ \Pi_{x}^{*}\times \Pi_{y}.
    \end{align*}
\end{remark}
\begin{remark} 	Note that the proposed MC-MTS-CBCFD scheme \eqref{scm:pu:1}--\eqref{comP} is linear and decoupled, where the velocity/pressure and concentration are solved sequentially, and every $Q$ concentration time steps, the velocity/pressure have to be updated only once. This multi-time-step idea  is quite agreement with the fact that the velocity field varies slower in time than the concentration one. Besides, to achieve global second-order temporal accuracy, a predicted velocity $\bm{U}^{1,*}$ in the first pressure time step is solved. In addition, the scheme can solve the velocity and pressure simultaneously in the same fourth-order spatial accuracy, avoiding reduced accuracy in the velocity approximation caused by taking differentiation for numerical pressure and then multiplying by an approximate nonlinear coefficient $a^{-1}(c)$ in the traditional FD methods. The detailed implementation of the MC-MTS-CBCFD scheme is referred to Algorithm \ref{alg:1}.
    \begin{algorithm}
       \caption{The implementation of the MC-MTS-CBCFD scheme }\label{alg:1}
         \begin{algorithmic}[1]       
          \State Solve $\{P^{0},\bm{U}^{0}\}$ by \eqref{scm:pu:1}--\eqref{scm:pu:3} enclosed with \eqref{comP} and input $C^0$
             \State Solve $\bm{W}^{0}$ by \eqref{scm:w:1}--\eqref{scm:w:2} or \eqref{scm:w:3}--\eqref{scm:w:4}  with inputs $\{C^0,\bm{V}^{0}, \bm{U}^{0}\}$
             \State Solve the predicted solution $\bm{U}^{1,*}$ by \eqref{scm:cu:p:1}--\eqref{scm:cu:p:8} enclosed with \eqref{comP} and inputs $\{C^0,\bm{V}^{0}, \bm{U}^{0}, \bm{W}^{0}\}$
             \For{$m \gets 0$ \textbf{to} $N_{p}-1$} 
                 \State $t_{p} \gets m \De t_{p} $    
                       \For{$i \gets 0$ \textbf{to} $Q-1$} 
                          \State $t_{c} \gets t_{p} + (i+1) \De t_{c}  $
                          \State Calculate $\bm{U}_{\#}^{n+1}$ by \eqref{udl}
                          \State Solve $\{C^{n+1},\bm{V}^{n+1}\}$ by \eqref{scm:cv:1}--\eqref{scm:cv:5} enclosed with \eqref{ibc}, see also \eqref{matrix}
                       \EndFor
                 \State Solve $\{P^{m+1},\bm{U}^{m+1}\}$ by \eqref{scm:pu:1}--\eqref{scm:pu:3} enclosed with \eqref{comP}
             \EndFor
         \end{algorithmic}
      \end{algorithm}
      
        The computational sequence clearly demonstrates that except for the prediction step, the velocity-pressure equation are solved only once every $Q$ times the concentration equation is computed. Compared to using the same time stepsize (i.e., $Q=1$) for both concentration and pressure, the multi-time-step method significantly improves computational efficiency, which has been numerically validated in our experiments. Besides, from the process of proving the existence and uniqueness of the solution in Theorem \ref{thm:exit}, it can be observed that the auxiliary vector variable $\bm{w}$ can be explicitly computed.
	\end{remark}
\begin{lemma}[Mass conservation]
		Let $c(\bm{x},t)$ be the solution of \eqref{mod:PDE}. Then there holds
		$$
		\int_{\Omega} \phi c (\bm{x},t) {\rm d} \bm{x} 
		= \int_{\Omega} \phi c^o(\bm{x}) {\rm d} \bm{x}
		+ \int_{0}^{t} \int_{\Omega} (q_{I} c_{I}+ q_{P} c)(\bm{x},\tau) {\rm d} \bm{x} {\rm d} \tau.
		$$
\end{lemma}
\begin{proof} This lemma can be easily proved by integrating the first equation in \eqref{mod:PDE} over $\Omega \times [0,t]$ and utilizing the boundary conditions.
\end{proof}

In the following, we shall prove that the MC-MTS-CBCFD scheme \eqref{scm:pu:1}--\eqref{comP} is mass-conserving in the discrete fashion.
\begin{lemma}\label{lem:dis} 
	Let $\omega \in \mal{P}$, $\nu^{x} \in \mal{U}$ and $\nu^{y}\in \mal{V}$, we have
$$
  \begin{aligned}
			& ( \mal{L}_{y} \de_{x} \nu^{x}, \omega   )_{\rm M} =  - ( \mal{L}_{y} \nu^{x}, \de_{x} \omega  )_{x},  
			&( \mal{L}_{x} \de_{y} \nu^{y}, \omega   )_{\rm M} =  - ( \mal{L}_{x} \nu^{y}, \de_{y} \omega  )_{y},\\
			& (  \mal{L}_{x}^{-1}\de_{x} \nu^{x}, \omega   )_{\rm M} =  - (  \nu^{x}, \mal{L}_{x}^{-1}\de_{x} \omega  )_{x},  
			& ( \mal{L}_{y}^{-1} \de_{y} \nu^{y}, \omega   )_{\rm M} =  - (  \nu^{y}, \mal{L}_{y}^{-1}\de_{y} \omega  )_{y}.
		\end{aligned}
		$$
	\end{lemma}
    \begin{proof}
	 The conclusions can be similarly proved as those in Lemma 2.1 of Ref. \cite{Wang24}, where the summation by parts and the periodic boundary conditions are applied. 
    \end{proof}
\begin{theorem}[Discrete mass conservation] \label{thm:mass}
Let $\{ C^{n}, \bm{V}^{n}\}$ be the solutions to the MC-MTS-CBCFD scheme \eqref{scm:pu:1}--\eqref{comP}.  Then, there holds
\begin{equation*}
     \big(  \mal{L}  [\phi C^{n}], 1 \big)_{\rm M} 
    = \big(  \mal{L}  [\phi C^{0}], 1 \big)_{\rm M} 
       + \De t_{c} \sum_{\ell=0}^{n-1} \big(  \mal{L}  [ c_{I}q_{I} + q_{P} \ovl{C} ]^{\ell+1/2}  , 1 \big)_{\rm M}, 
\end{equation*}
      for $1 \leq n \leq N_{c}$.
\end{theorem}
 \begin{proof}
	Taking the inner product for \eqref{scm:cv:1} with $1$ in the sense of $(\cdot,\cdot)_{\rm M}$ gives us
	\begin{equation*}
    \begin{aligned}
        & \big(  \mal{L}  [\phi d_{t} C^{n+1/2}], 1 \big)_{\rm M}  
           +  \big( \mal{L}_{y} \de_{x} \ovl{W}^{x,n+1/2} 
                    + \mal{L}_{x} \de_{y} \ovl{W}^{y,n+1/2}, 1 \big)_{\rm M} 
           - \big(\mal{L}  [ q_{P}\ovl{C}]^{n+1/2}, 1 \big)_{\rm M}\\
        & \qquad = \big(  \mal{L}  [c_{I}q_{I}]^{n+1/2}, 1 \big)_{\rm M}.
    \end{aligned}
	\end{equation*}
It follows from the Lemma \ref{lem:dis} with $W^{x,n} \in \mal{U}$ and $W^{y,n} \in \mal{V}$ that
	\begin{equation*}
		\big( \mal{L}_{y} \de_{x} \ovl{W}^{x,n+1/2} + \mal{L}_{x} \de_{y} \ovl{W}^{y,n+1/2}, 1 \big)_{\rm M}  
          = - \big( \mal{L}_{y}  \ovl{W}^{x,n+1/2}, \de_{x}1 \big)_{x}  
             -  \big(\mal{L}_{x}  \ovl{W}^{y,n+1/2}, \de_{y}1 \big)_{y} = 0.
	\end{equation*}
Therefore,  we have
\begin{equation*}
   \big(  \mal{L}  [\phi C^{n+1}]  , 1 \big)_{\rm M} 
       - \big(  \mal{L}  [\phi C^{n}]  , 1 \big)_{\rm M} 
       - \De t_{c} \big(  \mal{L}  [ q_{P}\ovl{C}]^{n+1/2}  , 1 \big)_{\rm M} 
    = \De t_{c} \big(  \mal{L}  [c_{I}q_{I}]^{n+1/2}  , 1 \big)_{\rm M},
	\end{equation*}
	which implies the conclusion.
\end{proof}
\section{Numerical analysis}\label{sec:analysis}
In this section, we shall first show the optimal error analysis of the MC-MTS-CBCFD scheme \eqref{scm:pu:1}--\eqref{comP}, and then we prove the unique solvability of the scheme. For simplicity, in this section we constrain our numerical analysis only to the molecular diffusion case, i.e., the diffusion coefficient $D(\bm x)=\alpha_{m}\phi(\bm{x})$.

Several useful lemmas are presented for the subsequent theoretical analysis.
\begin{lemma}[\cite{Shi21}]\label{lem:com} 
 Under periodic boundary conditions, the operators $\delta_x$, $\delta_y$, $\mal{L}_{x}$  and $\mal{L}_{y}$ are commutative in the sense that for $\omega \in \mal{P}$, $\nu_{1}^{x}, \nu_{2}^{x} \in  \mal{U}$ and $\nu_{1}^{{y}}, \nu_{2}^{{y}} \in \mal{V}$ 
		$$
		( \mal{L}_{\kappa}^{-1} \delta_{\kappa}\omega , \nu_{1}^{\kappa}   )_{\kappa} =  (  \delta_{\kappa} \mal{L}_{\kappa}^{-1} \omega, \nu_{1}^{\kappa}  )_{\kappa}, \quad  ( \mal{L}_{\kappa}\nu_1^{\kappa}, \nu_2^{\kappa}   )_{\kappa} =  (   \nu_1^{\kappa}, \mal{L}_{\kappa} \nu_2^{\kappa}  )_{\kappa}, ~~ \kappa=x,y.
		$$
	\end{lemma}
\begin{lemma}[\cite{Shi21}]\label{lem:seminorm}
  Assume that $\omega \in \mal{P}$ and $\omega_{1,1}=0$. Then, we have 
		\begin{equation*}
			\| \omega \|_{\rm M}^{2} \leq 2(x_{R}-x_{L})^{2} \|\de_{x}\omega \|_{x}^{2} + 2(y_{R}-y_{L})^{2} \|\de_{y}\omega \|_{y}^{2} \leq K_{3} | \omega |_{1}^{2},
		\end{equation*}
		where $K_{3}:= 2(x_{R}-x_{L})^{2} + 2(y_{R}-y_{L})^{2}$. 
\end{lemma}
\begin{lemma}[\cite{Shi21,Wang24}]\label{lem:op:L}
  Let $\omega \in \mal{P}$, $\nu^{x} \in \mal{U}$ and $\nu^{y} \in \mal{V}$. Then, we have
    \begin{equation*}
	\begin{aligned}
		& \| \mal{L}_{\kappa} \omega\|_{\rm M} \leq  \|  \omega\|_{\rm M}, \quad  
          \| \mal{L} \omega\|_{\rm M} \leq  \|  \omega\|_{\rm M} \leq  \|\mal{L}^{-1} \omega\|_{\rm M} \leq \left(\f{16}{11}\right)^{2}\|  \omega \|_{\rm M}, \\
        & \| \mal{L}_{\kappa} \nu^{\kappa} \|_{\kappa} \leq  \|  \nu^{\kappa} \|_{\kappa} \le \|\mal{L}_{\kappa}^{-1} \nu^{\kappa}\|_{\kappa} \leq \f{16}{11}\|  \nu^{\kappa} \|_{\kappa},\quad \kappa=x,y.
      \end{aligned}
     \end{equation*}
\end{lemma}
\begin{lemma}\label{lem:op:t}	Let $\omega \in \mal{P}$. Then, we have 
    \begin{equation*}
		\| \mal{T}_{\kappa} \omega \|_{\kappa} \leq \f{5}{4}\|  \omega \|_{\rm M} \leq \f{5}{4}\|  \omega \|_{\infty},\quad \kappa=x,y.
    \end{equation*}
\end{lemma}
\begin{proof} We only consider the case $\kappa=x$. By definition, we have
	  \begin{equation*}
		  \begin{aligned}
			\| \mal{T}_{x} \omega \|_{x}^{2} 
			&= \f{h^{x}h^{y}}{16^{2}} \sum_{i=1}^{N_{x}} \sum_{j=1}^{N_{y}} \left( - \omega_{i-1,j} + 9\omega_{i,j} + 9\omega_{i+1,j} - \omega_{i+2,j} \right)^{2} \\
			&\leq \f{h^{x}h^{y}}{16^{2}} \sum_{i=1}^{N_{x}} \sum_{j=1}^{N_{y}} \left( 20 \omega_{i-1,j}^{2} + 180 \omega_{i,j}^{2} + 180 \omega_{i+1,j}^{2} + 20 \omega_{i+2,j}^{2} \right) \\
			& \leq \f{400h^{x}h^{y}}{16^{2}}\sum_{i=1}^{N_{x}} \sum_{j=1}^{N_{y}} \omega_{i,j}^{2} 
            = \f{25}{16} \|\omega\|_{\rm M}^{2} \le \f{25}{16} \|\omega\|_{\infty}^{2},
		 \end{aligned}
	  \end{equation*}
which implies the conclusion.
\end{proof}
The following lemma follows from the results of classical Lagrange interpolations.
\begin{lemma}[\cite{TNA09}]\label{lem:t:err}
  Assume that $g \in H^{4}(\Omega)$, there exist positive constants $K$, independent of $h$, such that
		\begin{equation*}
			\| g - \mal{T}_{x} g \|_{x} \leq Kh^{4},\quad
			\| g - \mal{T}_{y} g \|_{y} \leq Kh^{4}.
		\end{equation*}
 \end{lemma}
By the well-known Bramble-Hilbert Lemma (cf. \cite{BH}), we can easily prove the following conclusion.
\begin{lemma}[\cite{Wang24}]\label{lem:BH}
   If $g \in H^{5}(\Omega)$. Then, we have
	\begin{equation*}
		\| \de_{\kappa}g - \mal{L}_{\kappa} \p_{\kappa}g \|_{\rm M}  + \| \de_{\kappa}g - \mal{L}_{\kappa} \p_{\kappa}g \|_{\kappa}  \leq Kh^{4},\quad \kappa= x,y.
	\end{equation*}
\end{lemma}
	
  \subsection{Truncation errors and error analysis}
 Define the errors $e_{c}:=c-C$, $\bm{e}_{\bm{v}}:=(e_{v^{x}},e_{v^{y}})=\bm{v}-\bm{V}$, $\bm{e}_{\bm{w}}:=(e_{w^{x}},e_{w^{y}})=\bm{w}-\bm{W}$, $e_{p}:=p-P$, and $\bm{e}_{\bm{u}}:=(e_{u^{x}},e_{u^{y}})=\bm{u}-\bm{U}$. From \eqref{mod:2D} and \eqref{scm:cv:1}--\eqref{scm:cv:5}, we obtain the error equations of the concentration equation at time $t=t_{c}^{n+1}$ as 
    \begin{align}
            &\big[ \mal{L}[\phi d_{t} e_{c}] +  \mal{L}_{y}\de_{x}\ovl{e}_{w^{x}} +  \mal{L}_{x}\de_{y}\ovl{e}_{w^{y}} - \mal{L}[ q_{P} \ovl{e}_{c} ]  \big]_{i,j}^{n+1/2}  = \big[ E_{1} \big]_{i,j}^{n+1/2} , \label{err:cv:1} \\
            &\big[ \mal{L}_{x} e_{v^{x}} + \de_{x}e_{c} \big]_{i+1/2,j}^{n+1} = \big[  E_{2}^{x}  \big]_{i+1/2,j}^{n+1}, \label{err:cv:2} \\
            &\big[ \mal{L}_{y} e_{v^{y}} + \de_{y}e_{c} \big]_{i,j+1/2}^{n+1} = \big[  E_{2}^{y}  \big]_{i,j+1/2}^{n+1}, \label{err:cv:3} \\
			&\big[ e_{w^{x}} + U_{\#}^{x}\mal{T}_{x}C - u_{\#}^{x}c - D e_{v^{x}} \big]_{i+1/2,j}^{n+1} = \big[  E_{3}^{x}  \big]_{i+1/2,j}^{n+1}, \label{err:cv:4} \\
			&\big[ e_{w^{y}} + U_{\#}^{y}\mal{T}_{y}C - u_{\#}^{y}c - D e_{v^{y}} \big]_{i,j+1/2}^{n+1} = \big[E_{3}^{y}\big]_{i,j+1/2}^{n+1}, \label{err:cv:5} 
    \end{align}
and similarly, from \eqref{mod:2D} and \eqref{scm:pu:1}--\eqref{scm:pu:3}, we get the error equations of the pressure equation at time $t=t_{p}^{m}$ as
    \begin{align}
            & \big[ \mal{L}_{y} \de_{x} e_{u^{x}} + \mal{L}_{x} \de_{y} e_{u^{y}} \big]_{i,j}^{m} = \big[ E_{4} \big]_{i,j}^{m}, \label{err:pu:1}\\
			& \big[ \mal{L}_{x}\big[ [a(\mal{T}_{x}c_{p}) - a(\mal{T}_{x}C_{p})] u^{x} + a(\mal{T}_{x}C_{p})e_{u^{x}} \big] + \de_{x}e_{p}  \big]_{i+1/2,j}^{m} = \big[ E_{5}^{x} \big]_{i+1/2,j}^{m}, \label{err:pu:2}\\
			& \big[ \mal{L}_{y}\big[ [a(\mal{T}_{y}c_{p}) -a(\mal{T}_{y}C_{p})] u^{y} +  a(\mal{T}_{y}C_{p})e_{u^{y}} \big] + \de_{y}e_{p}  \big]_{i,j+1/2}^{m} = \big[ E_{5}^{y} \big]_{i,j+1/2}^{m},\label{err:pu:3}
    \end{align}
for $n \geq 0 $ and $m \geq 0 $, where the local truncation errors $E_{1}$, $\bm{E}_{\bm{2}}:=(E_{2}^{x},E_{2}^{y})$, $\bm{E}_{\bm{3}}:=(E_{3}^{x},E_{3}^{y})$, $E_{4}$ and $\bm{E}_{\bm{5}}:=(E_{5}^{x},E_{5}^{y})$  are given by
	\begin{equation*}
		\begin{aligned}
			\big[ E_{1} \big]_{i,j}^{n+1/2}&:= \big[ \mal{L}[ \phi[ d_{t}c - \p_{t}c] + \p_{x}[\ovl{w}^{x} - w^{x}] + \p_{y}[\ovl{w}^{y} - w^{y}] ] + q_{P}[\ovl{c}-c] \\
			&\qquad +  \mal{L}_{y}[\de_{x}\ovl{w}^{x} - \mal{L}_{x}\p_{x}\ovl{w}^{x}] +  \mal{L}_{x}[\de_{y}\ovl{w}^{y} - \mal{L}_{y}\p_{y}\ovl{w}^{y}]  \big]^{n+1/2}_{i,j},\\
			\big[ E_{2}^{x} \big]_{i+1/2,j}^{n+1}&:= \big[ \de_{x}c -  \mal{L}_{x}\p_{x}c \big]_{i+1/2,j}^{n+1}, \quad \big[ E_{2}^{y} \big]_{i,j+1/2}^{n+1}:= \big[ \de_{y}c -  \mal{L}_{y}\p_{y}c  \big]_{i,j+1/2}^{n+1}, \\
			\big[ E_{3}^{x} \big]_{i+1/2,j}^{n+1} &:= \big[ [u^{x} -  u_{\#}^{x} ]c  \big]_{i+1/2,j}^{n+1}, \quad \big[ E_{3}^{y} \big]_{i,j+1/2}^{n+1}:= \big[ [ u^{y} - u_{\#}^{y}] c \big]_{i,j+1/2}^{n+1},\\
			\big[ E_{4} \big]_{i,j}^{m}&:= \big[ \mal{L}_{y} [\de_{x} u^{x} - \mal{L}_{x}\p_{x}u^{x}]  
                              + \mal{L}_{x} [\de_{y} u^{y} - \mal{L}_{y}\p_{y}u^{y}] \big]_{i,j}^{m},\\
			\big[ E_{5}^{x} \big]_{i+1/2,j}^{m}&:= \big[ \mal{L}_{x}[ a(\mal{T}_{x}c_{p})u^{x} - a(c_{p})u^{x} ] + [\de_{x}p - \mal{L}_{x}\p_{x}p] \big]_{i+1/2,j}^{m},\\
			\big[ E_{5}^{y} \big]_{i,j+1/2}&:= \big[ \mal{L}_{y}[ a(\mal{T}_{y}c_{p})u^{y} - a(c_{p})u^{y} ] + [\de_{y}p - \mal{L}_{y}\p_{y}p ] \big]_{i,j+1/2}^{m}.
		\end{aligned}
	\end{equation*}
	Moreover, under the regularity condition (A1), by using the Taylor expansion, Lemmas 2.1--2.2 of Ref. \cite{Xu25} and Lemmas \ref{lem:t:err}--\ref{lem:BH}, it is easy to see
	\begin{equation}\label{err:truc}
		\begin{aligned}
			& \De t_{c} \sum_{\ell=0}^{n}\| E_{1}^{\ell+1/2} \|_{\rm M}^{2} \leq K \left( (\De t_{c})^{4} + h^{8} \right),~
			\De t_{c} \sum_{\ell=0}^{n} \| \bm{E}_{\bm{3}}^{\ell+1} \|_{\rm T}^{2} \leq K (\De t_{p})^{4},\\
			& \|\bm{E}_{\bm{2}}^{n} \|_{\rm T} + \|\bm{E}_{\bm{3}}^{0} \|_{\rm T} + \| E_{4}^{m} \|_{\rm M} + \| \bm{E}_{\bm{5}}^{m} \|_{\rm T}  \leq Kh^{4},\quad  n, m \geq 0.
		\end{aligned}
	\end{equation}
	
In addition, define the errors $e_{c,p}^{1,*}:=c(t_{p}^{1})-C_{p}^{1,*}$, $\bm{e}_{\bm{v},p}^{1,*}:=\bm{v}(t_{p}^{1})-\bm{V}_{p}^{1,*}$, $\bm{e}_{\bm{w},p}^{1,*}:=\bm{w}(t_{p}^{1})-\bm{W}_{p}^{1,*}$, $e_{p}^{1,*}:=p(t_{p}^{1})-P^{1,*}$ and $\bm{e}_{\bm{u}}^{1,*}:=\bm{u}(t_{p}^{1})-\bm{U}^{1,*}$. Then, the following error equations at pressure time $t=t_{p}^{1}$ hold:
    \begin{align}
            &\big[ \mal{L}[\phi D_{t} e_{c,p}] +  \mal{L}_{y}\de_{x}\ovl{e}_{w^{x},p} +  \mal{L}_{x}\de_{y}\ovl{e}_{w^{y},p} - \mal{L}[ q_{P} \ovl{e}_{c,p} ] \big]_{i,j}^{1/2,*}  = \big[ E_{1,p} \big]_{i,j}^{1/2,*} , \label{err:cu:p:1} \\
            &\big[ \mal{L}_{x} e_{v^{x},p} + \de_{x}e_{c,p} \big]_{i+1/2,j}^{n+1} = \big[  E_{2,p}^{x}  \big]_{i+1/2,j}^{n+1}, \label{err:cu:p:2} \\
            &\big[ \mal{L}_{y} e_{v^{y},p} + \de_{y}e_{c,p} \big]_{i,j+1/2}^{n+1} = \big[  E_{2,p}^{y}  \big]_{i,j+1/2}^{n+1}, \label{err:cu:p:3} \\
			&\big[ e_{w^{x},p} + U_{\#}^{x}\mal{T}_{x}C_{p} - u_{\#}^{x}c_{p} - D e_{v^{x},p} \big]_{i+1/2,j}^{1,*} = \big[  E_{3,p}^{x} \big]_{i+1/2,j}^{1,*}, \label{err:cu:p:4}\\
			&\big[ e_{w^{y},p} + U_{\#}^{y}\mal{T}_{y}C_{p} - u_{\#}^{y}c_{p} - D e_{v^{y},p} \big]_{i,j+1/2}^{1,*} = \big[E_{3,p}^{y}\big]_{i,j+1/2}^{1,*}, \label{err:cu:p:5}\\
			& \big[ \mal{L}_{y} \de_{x} e_{u^{x}} + \mal{L}_{x} \de_{y} e_{u^{y}} \big]_{i,j}^{1,*} = \big[ E_{4} \big]_{i,j}^{1,*}, \label{err:cu:p:6}\\
			& \big[ \mal{L}_{x}\big[ [a(\mal{T}_{x}c_{p}) - a(\mal{T}_{x}C_{p})] u^{x} + a(\mal{T}_{x}C_{p})e_{u^{x}} \big] + \de_{x}e_{p}  \big]_{i+1/2,j}^{1,*} = \big[ E_{5}^{x} \big]_{i+1/2,j}^{1,*}, \label{err:cu:p:7}\\
			& \big[ \mal{L}_{y}\big[ [a(\mal{T}_{y}c_{p}) - a(\mal{T}_{y}C_{p})] u^{y} + a(\mal{T}_{y}C_{p})e_{u^{y}} \big] + \de_{y}e_{p}  \big]_{i,j+1/2}^{1,*} = \big[ E_{5}^{y} \big]_{i,j+1/2}^{1,*}, \label{err:cu:p:8}
    \end{align}
where the truncation errors satisfy
\begin{equation}\label{err:truc:p}
	\begin{aligned}
		&\| E_{1,p}^{1/2,*} \|_{\rm M} \leq K \left( (\De t_{p})^{2} + h^{4} \right),
			&\| \bm{E}_{\bm{2},p}^{1,*} \|_{\rm T}+\|D_{t} \bm{E}_{\bm{2},p}^{1/2,*} \|_{\rm T} \leq K h^{4},\\
		&\| \bm{E}_{\bm{3},p}^{1,*} \|_{\rm T}+\|D_{t} \bm{E}_{\bm{3},p}^{1/2,*} \|_{\rm T} \leq K \De t_{p}, &\| E_{4}^{1,*} \|_{\rm M} + \| \bm{E}_{\bm{5}}^{1,*} \|_{\rm T}  \leq Kh^{4}.
	\end{aligned}
\end{equation}

\begin{lemma}\label{lem:err:pu} 
   Let $\{P^{m},\bm{U}^{m}\}$ be the solutions of system \eqref{scm:pu:1}--\eqref{scm:pu:3} at time $t=t_{p}^{m}$. Under assumptions (A1)--(A2), there exists a positive constant $K$,  such that the following estimate holds 
		\begin{equation*}
			\| e_{p}^{m} \|_{\rm M}^{2} + | e_{p}^{m} |_{1}^{2} + \| \bm{e}_{\bm{u}}^{m} \|_{\rm T}^{2} 
			\leq  K\| e_{c,p}^{m} \|_{\rm M}^{2} + Kh^{8}, \quad m \geq 0,
		\end{equation*}
    where $e_{c,p}^{m}=e_{c}^{mQ}$ represents the concentration error at pressure time $t=t_{p}^{m}$.
\end{lemma}
\begin{proof}
  Applying the operator $\mal{L}^{-1}$ to both sides of \eqref{err:pu:1} and taking the inner product of the resulting equation with $e_{p}^{m}$ in the sense of $(\cdot,\cdot)_{\rm M}$ yields
  \begin{equation}\label{erru:1}
	  \left( \mal{L}_{x}^{-1} \de_{x} e_{u^{x}}^{m}, e_{p}^{m} \right)_{\rm M} 
      + \left( \mal{L}_{y}^{-1} \de_{y} e_{u^{y}}^{m}, e_{p}^{m} \right)_{\rm M} 
      = \left( \mal{L}^{-1}E_{4}^{m}, e_{p}^{m} \right)_{\rm M} .
  \end{equation}
	
First, by Lemma \ref{lem:dis} and utilizing error equations \eqref{err:pu:2}--\eqref{err:pu:3} and assumption (A2), the left-hand side of \eqref{erru:1} can be bounded below by
   \begin{equation}\label{erru:2}
	\begin{aligned}
			& \left( \mal{L}_{x}^{-1} \de_{x} e_{u^{x}}^{m}, e_{p}^{m} \right)_{\rm M} 
            + \left( \mal{L}_{y}^{-1} \de_{y} e_{u^{y}}^{m}, e_{p}^{m} \right)_{\rm M} \\
			& = - \left( e_{u^{x}}^{m}, \mal{L}_{x}^{-1}\de_{x}e_{p}^{m} \right)_{x} 
                - \left( e_{u^{y}}^{m}, \mal{L}_{y}^{-1}\de_{y}e_{p}^{m} \right)_{y} \\
            & = \left( e_{u^{x}}^{m}, a(\mal{T}_{x}C_{p}^{m}) e_{u^{x}}^{m} \right)_{x}
              +  \left( e_{u^{y}}^{m}, a(\mal{T}_{y}C_{p}^{m}) e_{u^{y}}^{m} \right)_{y} + I_{1} 
			\geq a_{*}\| \bm{e}_{\bm{u}}^{m} \|_{\rm T}^{2}  + I_{1},
		\end{aligned}
    \end{equation}
    where $a_{*}:=\mu_{*}/k^{*}$ and 
    \begin{equation*}
           I_{1} := \sum_{\kappa=x,y} \left(  e_{u^{\kappa}}^{m} , [a(\mal{T}_{\kappa}c_{p}^{m}) - a(\mal{T}_{\kappa}C_{p}^{m})]u^{\kappa,m} - \mal{L}_{\kappa}^{-1}E_{5}^{\kappa,m} \right)_{\kappa}.
    \end{equation*}
 Furthermore, Lemma \ref{lem:op:t} and assumption (A2) implies that
 \begin{equation}\label{erru:3}
    \|  a(\mal{T}_{x}c_{p}^{m}) - a(\mal{T}_{x}C_{p}^{m}) \|_{x} 
            + \|  a(\mal{T}_{y}c_{p}^{m}) - a(\mal{T}_{y}C_{p}^{m}) \|_{y} 
		\leq  K\| e_{c,p}^{m} \|_{\rm M}.
\end{equation}
 Then, by Young's inequality, estimate \eqref{erru:3} and Lemma \ref{lem:op:L}, the $I_{1}$ term is bounded by
	\begin{equation}\label{erru:4}
		I_{1}
		\leq \f{a_{*}}{4} \| \bm{e}_{\bm{u}}^{m} \|_{\rm T}^{2} + K\| e_{c,p}^{m} \|_{\rm M}^{2} + Kh^{8}.
	\end{equation}
	
Next, it follows from \eqref{err:pu:2}--\eqref{err:pu:3} and \eqref{erru:3}, we get
	\begin{equation}\label{erru:6}
		\begin{aligned}
			|e_{p}^{m}|_{1}^{2} 
            & = \| \de_{x}e_{p}^{m} \|_{x}^{2} + \| \de_{y}e_{p}^{m} \|_{y}^{2}\\  
            & \le \sum_{\kappa=x,y} \| \mal{L}_{\kappa}[ a(\mal{T}_{\kappa}c_{p}^{m}) - a(\mal{T}_{\kappa}C_{p}^{m}) ]  u^{\kappa,m} 
                + a(\mal{T}_{\kappa}C_{p}^{m})e_{u^{\kappa}}^{m}- E_{5}^{\kappa,m} \|_{\kappa}^{2}\\
			& \leq 3(a^{*})^{2}\| \bm{e}_{\bm{u}}^{m} \|_{\rm T}^{2}  + K\| e_{c,p}^{m} \|_{\rm M}^{2} + Kh^{8},
		\end{aligned}
	\end{equation}
    where $a^{*}:= \mu^{*}/k_{*}$.	Then, utilizing Cauchy-Schwartz inequality, Lemmas \ref{lem:seminorm}--\ref{lem:op:L} and the estimate \eqref{err:truc}, the left-hand side of \eqref{erru:1} can be estimated as
  \begin{equation}\label{erru:7}
	\begin{aligned}
		\left( \mal{L}^{-1}E_{4}^{m} ,e_{p}^{m}\right)_{\rm M} &\leq   K\|e_{p}^{m} \|_{\rm M}h^{4} \leq K |e_{p}^{m} |_{1}h^{4}\\
			&\leq \f{a_{*}}{4}\| \bm{e}_{\bm{u}}^{m} \|_{\rm T}^{2} + K\| e_{c,p}^{m} \|_{\rm M}^{2} + Kh^{8}.
	\end{aligned}
  \end{equation}
	
Finally, combining the above estimates \eqref{erru:1}, \eqref{erru:2}, \eqref{erru:4} and \eqref{erru:7},  we have
	\begin{equation}\label{err:u}
		\| \bm{e}_{\bm{u}}^{m} \|_{\rm T}^{2} \leq  K\| e_{c,p}^{m} \|_{\rm M}^{2} + Kh^{8}.
	\end{equation}
Moreover, it follows from \eqref{erru:6}, \eqref{err:u} and Lemma \ref{lem:seminorm} that
	\begin{equation}\label{err:p}
		\|e_{p}^{m}\|_{\rm M}^{2} + |e_{p}^{m}|_{1}^{2} 
		\leq K\| \bm{e}_{\bm{u}}^{m} \|_{\rm T}^{2} + K\| e_{c,p}^{m} \|_{\rm M}^{2} + Kh^{8}
		\leq  K\| e_{c,p}^{m} \|_{\rm M}^{2} + Kh^{8}.
	\end{equation}
The proof is thereby completed by the estimates derived in \eqref{err:u}--\eqref{err:p}. 
\end{proof}
	
Based on the fact that $e_{c,p}^{0} = e_{c}^{0}=0$, along with Lemma \ref{lem:err:pu}, the following conclusion is established.
\begin{corollary}\label{col:erru0}
   Let $\{P^{0},\bm{U}^{0}\}$ be the solutions of system \eqref{scm:pu:1}--\eqref{scm:pu:3} at time $t=t_{p}^{0}$. Then, there exists a positive constant $K$ such that 
   \begin{equation*}
	\|p^{0}- P^{0}\|_{\rm M} + |p^{0}- P^{0}|_{1} + \| \bm{u}^{0} -\bm{U}^{0} \|_{\rm T} \leq Kh^{4}.
   \end{equation*}
\end{corollary}

The following lemma provides error estimates for the prediction step; see \ref{App:A} for a detailed proof.
\begin{lemma}\label{lem:err:pre}
    Let $\{C_{p}^{1,*},\bm{V}_{p}^{1,*},P^{1,*},\bm{U}^{1,*}\}$ be the solutions of system \eqref{scm:cu:p:1}--\eqref{scm:cu:p:8} with \eqref{ibc}--\eqref{comP} at time $t=t_{p}^{1}$. Under assumptions (A1)--(A2), there exists a positive constant $K$,  such that the following estimate holds
	\begin{equation*}
             \|e_{c,p}^{1,*}\|_{\rm M}^{2} 
           + \|e_{p}^{1,*}\|_{\rm M}^{2} + |e_{p}^{1,*}|_{1}^{2} + \| \bm{e}_{\bm{u}}^{1,*} \|_{\rm T}^{2} 
		  \leq K\left( (\De t_{p})^{4} + h^{8}\right).
	\end{equation*}
\end{lemma}

The main error estimate conclusion of this paper is list as follows.
\begin{theorem}\label{thm:err}
    Let $\{C^{n+1},\bm{V}^{n+1}\}$ be the solutions of system \eqref{scm:cv:1}--\eqref{scm:cv:5} at time $t=t_{c}^{n+1}$, $t_{p}^{m} < t_{c}^{n+1}\leq t_{p}^{m+1}$. Moreover, let $\{P^{m+1},\bm{U}^{m+1}\}$ be the solutions of system \eqref{scm:pu:1}--\eqref{scm:pu:3} at time $t=t_{p}^{m+1}$. Under assumptions (A1)--(A2), there exists a positive constant $K$,  such that the following estimates hold
    \begin{equation*}
		\begin{aligned}
            &\|p^{m+1}-P^{m+1}\|_{\rm M}^{2} + |p^{m+1}-P^{m+1}|_{1}^{2} + \| 
			\bm{u}^{m+1} - \bm{U}^{m+1} \|_{\rm T}^{2} + \|c^{n+1}-C^{n+1}\|_{\rm M}^{2} \\
			&\quad   + \De t_{c} \sum_{\ell=0}^{n} \left( \| \ovl{\bm{v}}^{\ell+1/2} - \ovl{\bm{V}}^{\ell+1/2} \|_{\rm T}^{2} + \| \ovl{\bm{w}}^{\ell+1/2} - \ovl{\bm{W}}^{\ell+1/2} \|_{\rm T}^{2} \right) \leq K\left( (\De t_{p})^{4} + (\De t_{c})^{4}  +  h^{8}\right),
		\end{aligned}
    \end{equation*}
    for $n\geq 0$ and $m\geq 0$.
\end{theorem}
\begin{proof} We shall demonstrate the error results by the mathematical induction method. First, assume
\begin{equation}\label{assum}
	\| \bm{U}^{\ell} \|_{\infty} \leq K_{1} + 1, \quad \ell \leq m,
\end{equation}
From \eqref{bound:u0}, it is clear that \eqref{assum} holds for $\ell=0$. Besides, utilizing Lemma \ref{lem:err:pre} and the inverse inequality, we conclude that 
\begin{equation}\label{bound:u1}
	\| \bm{U}^{1,*}\|_{\infty} \leq \| \bm{u}^{1}\|_{\infty} + \| \bm{e}_{\bm{u}}^{1,*}\|_{\infty} \leq K_{1} + Kh^{-1}(\De t_{p})^{2}  + Kh^{3}  \leq K_{1} + 1,
\end{equation}
under the condition $\De t_{p} = o(h^{1/2})$.
Then, assumption \eqref{assum} combined with the definition of $\bm{U}_{\#}^{n}$ in \eqref{udl} immediately yields
\begin{equation}\label{bound:udlU}
    \| \bm{U}_{\#}^{n} \|_{\infty} \leq \left\{
    \begin{aligned}
         & \| \bm{U}^{1,*} \|_{\infty} + \| \bm{U}^{0} \|_{\infty} \leq 2(K_{1}+1),
         & t_{p}^{0} \leq t_{c}^{n} \leq t_{p}^{1}, ~~& m=0,\\
         & 2\| \bm{U}^{m} \|_{\infty} + \| \bm{U}^{m-1} \|_{\infty} \leq 3(K_{1}+1),
         & t_{p}^{m} < t_{c}^{n} \leq t_{p}^{m+1}, ~~& m \geq 1.
    \end{aligned}\right.
\end{equation}

Now, applying the operator $\mal{L}^{-1}$ to both sides of the error equation \eqref{err:cv:1}, and taking the inner product for the resulting equation with $\ovl{e}_{c}^{n+1/2}$ in the sense of $(\cdot,\cdot)_{\rm M}$ gives us
\begin{equation}\label{errc:1}
	\begin{aligned}
		& \left( \phi d_{t} e_{c}^{n+1/2}, \ovl{e}_{c}^{n+1/2} \right)_{\rm M} 
        + \left( [\mal{L}_{x}^{-1} \de_{x} \ovl{e}_{w^{x}} + \mal{L}_{y}^{-1} \de_{y} \ovl{e}_{w^{y}}]^{n+1/2}, \ovl{e}_{c}^{n+1/2}\right)_{\rm M} \\
		& = \left(  [q_{P}\ovl{e}_{c}]^{n+1/2}, \ovl{e}_{c}^{n+1/2} \right)_{\rm M} 
        +   \left( \mal{L}^{-1}E_{1}^{n+1/2}, \ovl{e}_{c}^{n+1/2} \right)_{\rm M}\\
        & \leq K \| \ovl{e}_{c}^{n+1/2} \|_{\rm M}^{2} + K\| E_{1}^{n+1/2} \|_{\rm M}^{2},
	\end{aligned}
\end{equation}
where Lemma \ref{lem:op:L} has been applied in the last step.

Similar to \eqref{erru:2}, using Lemma \ref{lem:dis} and error equations \eqref{err:cv:2}--\eqref{err:cv:5}, the second term of the left-hand side of \eqref{errc:1} can be written as
\begin{equation}\label{errc:2}
	\begin{aligned}
		& \left(  \mal{L}_{x}^{-1} \de_{x} \ovl{e}_{w^{x}}^{n+1/2}, \ovl{e}_{c}^{n+1/2} \right)_{\rm M} 
        + \left(  \mal{L}_{y}^{-1} \de_{y} \ovl{e}_{w^{y}}^{n+1/2}, \ovl{e}_{c}^{n+1/2} \right)_{\rm M}\\
		& = - \left( \ovl{e}_{w^{x}}^{n+1/2},  \mal{L}_{x}^{-1} \de_{x} \ovl{e}_{c}^{n+1/2} \right)_{x} 
            - \left( \ovl{e}_{w^{y}}^{n+1/2},  \mal{L}_{y}^{-1} \de_{y} \ovl{e}_{c}^{n+1/2} \right)_{y}\\
        & = \left(  D \ovl{\bm{e}}_{\bm{v}}^{n+1/2} ,  \ovl{\bm{e}}_{\bm{v}}^{n+1/2} \right)_{\rm T} + I_{4}
		=:   \|\ovl{\bm{e}}_{\bm{v}}^{n+1/2}\|_{{\rm T},D}^{2} + I_{4},
	\end{aligned}
\end{equation}
where 
\begin{equation*}
  \begin{aligned}
    I_{4}
    &:= - \sum_{\kappa=x,y} \left( \ovl{\eta}^{\kappa,n+1/2},  [ \ovl{e}_{v^{\kappa}} - \mal{L}_{\kappa}^{-1} \ovl{E}_{2}^{\kappa} ]^{n+1/2}\right)_{\kappa} 
                 - \sum_{\kappa=x,y} \left(   D\ovl{e}_{v^{\kappa}}^{n+1/2} ,   \mal{L}_{\kappa}^{-1} \ovl{E}_{2}^{\kappa,n+1/2}\right)_{\kappa},\\
     \bm{\eta}
   & := \big( \eta^{x}, \eta^{y} \big) 
     = \big( U_{\#}^{x} \mal{T}_{x} C - u_{\#}^{x} c - E_{3}^{x}, ~U_{\#}^{y} \mal{T}_{y} C - u_{\#}^{y} c - E_{3}^{y} \big).
  \end{aligned}
\end{equation*}
By the boundedness result in \eqref{bound:udlU} and Lemmas \ref{lem:op:t}--\ref{lem:t:err}, it is easy to verify that
\begin{equation}\label{errc:3}
       \begin{aligned}
            \|\bm{\eta}^{n+1}\|_{\rm T} 
           & = \sum_{\kappa=x,y} \| [ (u_{\#}^{\kappa} - U_{\#}^{\kappa}) c + U_{\#}^{\kappa} (c - \mal{T}_{\kappa}c) + U_{\#}^{\kappa} \mal{T}_{\kappa} e_{c} + E_{3}^{\kappa}  ]^{n+1}\|_{\kappa}\\
           &\leq  K\| e_{c}^{n+1}\|_{\rm M} + K\| \bm{e}_{\bm{u},\#}^{n+1}\|_{\rm T} + K \|\bm{E}_{\bm{3}}^{n+1}\|_{\rm T} + K h^{4},
       \end{aligned}
\end{equation}
where $\bm{e}_{\bm{u},\#}:= \bm{u}_{\#} - \bm{U}_{\#}$.
Then, we use the Cauchy-Schwarz inequality, Lemma \ref{lem:op:L} and \eqref{errc:3} to bound the $I_{4}$ term as follows:
\begin{equation}\label{errc:4}
   \begin{aligned}
	I_{4} 
        &\leq \f{1}{2} \| \ovl{\bm{e}}_{\bm{v}}^{n+1/2} \|_{{\rm T},D}^{2} 
         + K \sum_{\ell=0}^{1} \left(\| e_{c}^{n+\ell}\|_{\rm M}^{2}  
          +\| \bm{e}_{\bm{u},\#}^{n+\ell}\|_{\rm T}^{2}  + \| \bm{E}_{\bm{3}}^{n+\ell} \|_{\rm T}^{2}\right) + Kh^{4}.
    \end{aligned}
\end{equation}

Thus, substituting the estimates \eqref{errc:2} and \eqref{errc:4}  into \eqref{errc:1}, we get
\begin{equation}\label{errc:5}
	\begin{aligned}
		& \f{1}{2\De t_{c}} \left[ \left( \phi e_{c}^{n+1}  ,e_{c}^{n+1}\right)_{\rm M} - \left(  \phi e_{c}^{n}  ,e_{c}^{n}\right)_{\rm M} \right] 
        + \f{1}{2} \| \ovl{\bm{e}}_{\bm{v}}^{n+1/2} \|_{{\rm T},D}^{2}\\
        &\leq  K \sum_{\ell=0}^{1} \left(\| e_{c}^{n+\ell}\|_{\rm M}^{2}  
          +\| \bm{e}_{\bm{u},\#}^{n+\ell}\|_{\rm T}^{2}  + \| \bm{E}_{\bm{3}}^{n+\ell} \|_{\rm T}^{2}\right) + K\| E_{1}^{n+1/2} \|_{\rm M}^{2} + Kh^{4}   .
	\end{aligned}
\end{equation}
Now, multiplying both sides of \eqref{errc:5} by $2\De t_{c}$, with $n$ replaced by $l$, and summing from $0$ to $n$, choosing $\De t_{c} \leq \tau_{1}$ sufficiently small, the application of Gronwall’s inequality along with the truncation errors \eqref{err:truc} and Lemma \ref{lem:err:pu} yields
\begin{equation}\label{err:cn}
	\phi_{*} \|e_{c}^{n+1}\|_{\rm M}^{2} + \De t_{c} \sum_{l=0}^{n} \| \ovl{\bm{e}}_{\bm{v}}^{l+1/2} \|_{\rm T}^{2} 
    \leq K\left( (\De t_{p})^{4}  + (\De t_{c})^{4}  +  h^{8}\right).
\end{equation}
Besides, from \eqref{err:cv:4}--\eqref{err:cv:5}, \eqref{errc:3}, \eqref{err:cn} and Lemma \ref{lem:err:pu}, we can obtain
\begin{equation*}
    \De t_{c} \sum_{l=0}^{n} \| \ovl{\bm{e}}_{\bm{w}}^{l+1/2} \|_{\rm T}^{2} \leq K \De t_{c} \sum_{l=0}^{n} \left( \| \ovl{\bm{e}}_{\bm{v}}^{l+1/2} \|_{\rm T}^{2} + \| \ovl{\bm{\eta}}^{n+1/2} \|_{\rm T}^{2} \right) \leq K\left( (\De t_{p})^{4}  + (\De t_{c})^{4}  +  h^{8}\right).
\end{equation*}
Furthermore, by combining Lemma \ref{lem:err:pu} and estimate \eqref{err:cn}, we derive
\begin{equation}\label{err:un}
	\|e_{p}^{m+1}\|_{\rm M}^{2} + |e_{p}^{m+1}|_{1}^{2} + \| \bm{e}_{\bm{u}}^{m+1} \|_{\rm T}^{2} 
	\leq  K\| e_{c,p}^{m+1} \|_{\rm M}^{2} + Kh^{8} 
    \leq K\left( (\De t_{p})^{4}  + (\De t_{c})^{4} + h^{8} \right).
\end{equation}

Finally, we end the proof of this theorem by completing the assumption \eqref{assum} for $\ell=m+1$. Similar to \eqref{bound:u1}, utilizing \eqref{err:un} and the inverse inequality, we conclude that \eqref{assum} indeed holds for $\ell=m+1$ under the condition $\De t_{p} = Q\De t_{c} = o(h^{1/2})$. 
\end{proof}

\subsection{Existence and uniqueness of the solution}
It remains to prove the existence and uniqueness of the solutions $\{C^{n+1},\bm{W}^{n+1}\}$ and $\{P^{m},\bm{U}^{m}\}$ to the MC-MTS-CBCFD scheme \eqref{scm:pu:1}--\eqref{comP}. Although the scheme is linear, the corresponding coefficient matrix of the resulting linear algebraic system for concentration $C^{n+1}$ is directly related to the velocity $\bm{U}^{m}$, and vice versa. Therefore, the existence proof shall be carried out via the time-marching method using Browder's fixed point theorem \cite{Browder}.
\begin{theorem}\label{thm:exit}
  Assume the conditions in Theorem \ref{thm:err} hold. If the time step $\De t_{c} \leq \tau_{*}$, then the solutions of the MC-MTS-CBCFD scheme \eqref{scm:pu:1}--\eqref{comP} exist and are unique.
\end{theorem}
\begin{proof}
 The proof is divided into two parts.
 
\textbf{Part I: Uniqueness of the approximate solutions $\{P^{m}, \bm{U}^{m}\}$.}
First, we rewrite \eqref{scm:pu:2}--\eqref{scm:pu:3} as follows:
\begin{equation}\label{exit:1}
	U_{i+1/2,j}^{x,m} = - [a^{-1}(\mal{T}_{x}C_{p})\mal{L}^{-1}_{x}\de_{x}P]_{i+1/2,j}^{m},~
	U_{i,j+1/2}^{y,m} = - [a^{-1}(\mal{T}_{y}C_{p})\mal{L}^{-1}_{y}\de_{y}P]_{i,j+1/2}^{m}.
\end{equation}
Then, substituting \eqref{exit:1} into \eqref{scm:pu:1} and applying the operator $\mal{L}^{-1}$ to both sides of the resulting equation yields
\begin{equation}\label{exit:2}
    - \big[ \mal{L}_{x}^{-1}\de_{x}[a^{-1}(\mal{T}_{x}C_{p})\mal{L}^{-1}_{x}\de_{x}P] -  \mal{L}_{y}^{-1}\de_{y}[a^{-1}(\mal{T}_{y}C_{p})\mal{L}^{-1}_{y}\de_{y}P]\big]_{i,j}^{m} =  q_{i,j}^{m}.
\end{equation}

Now, we are ready to prove the existence of $P^{m}$ to \eqref{exit:2} using the well-known Browder's fixed point theorem \cite{Browder}. To this end, for given $ C_{p}^{m}$, let $F: \mal{P} \rightarrow \mal{P} $ be a continuous operator such that
\begin{equation*}
   [F(\omega)]_{i,j} := - \big[  \mal{L}_{x}^{-1}\de_{x}[a^{-1}(\mal{T}_{x}C_{p}^{m})\mal{L}^{-1}_{x}\de_{x}\omega] +    \mal{L}_{y}^{-1}\de_{y}[a^{-1}(\mal{T}_{y}C_{p}^{m})\mal{L}^{-1}_{y}\de_{y}\omega ] + q^{m} \big]_{i,j}.
\end{equation*}
Then, we aim to conclude that $\left( F(\omega),\omega\right)_{\rm M} \geq 0$  for $\omega \in \mal{P}$ with $\| \omega \|_{\rm M} = z_{1} $. It follows from Lemmas \ref{lem:dis}, \ref{lem:seminorm}--\ref{lem:op:L}, and the Cauchy-Schwarz inequality that
\begin{equation}\label{exit:3}
    \begin{aligned}
		\left( F(\omega),\omega\right)_{\rm M}
		& = \sum_{\kappa=x,y} \left( a^{-1}(\mal{T}_{\kappa}C_{p}^{m})\mal{L}^{-1}_{\kappa}\de_{\kappa}\omega, \mal{L}_{\kappa}^{-1}\de_{\kappa}\omega \right)_{\kappa}  - \left( q^{m}, \omega \right)_{\rm M}\\
		& \geq  \f{1}{a^{*}} \big[ \|\de_{x}\omega\|_{x}^{2} + \|\de_{y}\omega\|_{y}^{2}\big]  - \left( q^{m}, \omega\right)_{\rm M}\\
		& \geq \f{1}{a^{*}K_{3}} \left( \|\omega\|_{\rm M} - a^{*}K_{3} \| q^{m}\|_{\rm M} \right)\|\omega\|_{\rm M} \geq 0,
    \end{aligned}
\end{equation}
for $z_{1} = a^{*}K_{3}\| q^{m}\|_{\rm M} $. Thus, as a consequence of the Browder's fixed point
theorem, there exists at least one solution $\omega^{*} \in \mal{P}$ such that $F(\omega^{*})=0$, that is, the equation \eqref{exit:2} is solvable. Therefore, the existence of solutions $\{P^{m},\bm{U}^{m}\}$ to \eqref{scm:pu:1}--\eqref{scm:pu:3} is implied for $m\geq 0$.

Next, we show the uniqueness of the approximate solutions $\{P^{m}, \bm{U}^{m}\}$. Assume that there are another solutions $\{\hat{P}^{m}, \hat{\bm{U}}^{m}\}$ to \eqref{scm:pu:1}--\eqref{scm:pu:3} with \eqref{comP}. Obviously, they also satisfy equations \eqref{exit:1}--\eqref{exit:2}. Then, the errors $\gamma_{p}^{m}:=P^{m}-\hat{P}^{m}$ and $\bm{\gamma}_{u}^{m}:=(\gamma_{u^{x}}^{m},\gamma_{u^{y}}^{m})=\bm{U}^{m}-\hat{\bm{U}}^{m}$ naturally satisfy the following equations: 
\begin{align}
    & - \big[ \mal{L}_{x}^{-1}\de_{x}[a^{-1}(\mal{T}_{x}C_{p})\mal{L}^{-1}_{x}\de_{x}\gamma_{p}] -  \mal{L}_{y}^{-1}\de_{y}[a^{-1}(\mal{T}_{y}C_{p})\mal{L}^{-1}_{y}\de_{y}\gamma_{p}]\big]_{i,j}^{m} =  0, \label{gam:1}\\
    & [ \gamma_{u^{x}} + a^{-1}(\mal{T}_{x}C_{p})\mal{L}^{-1}_{x}\de_{x}\gamma_{p}]_{i+1/2,j}^{m}=0,~
	 [ \gamma_{u^{y}} + a^{-1}(\mal{T}_{y}C_{p})\mal{L}^{-1}_{y}\de_{y}\gamma_{p}]_{i,j+1/2}^{m}=0. \label{gam:2}
\end{align}

Analogous to \eqref{exit:3}, taking the inner product for equation \eqref{gam:1} with $\gamma_{p}^{m}$ yields
\begin{equation*}
     \f{1}{a^{*}K_{3}}\|\gamma_{p}^{m}\|_{\rm M}^{2} 
     \leq  \f{1}{a^{*}}|\gamma_{p}^{m}|_{1}^{2} 
     \leq 
    \sum_{\kappa=x,y} \left( a^{-1}(\mal{T}_{\kappa}C_{p})\mal{L}^{-1}_{\kappa}\de_{\kappa}\gamma_{p}^{m}, \mal{L}_{\kappa}^{-1}\de_{\kappa}\gamma_{p}^{m} \right)_{\kappa} = 0,
\end{equation*}
 which means $\|\gamma_{p}^{m}\|_{\rm M}=0$, i.e., $P^{m}=\hat{P}^{m}$. It then follows from  \eqref{gam:2} that $\bm{\gamma}_{\bm{u}}^{m}=\bm{0}$. Thus, the approximate solutions $\{P^{m}, \bm{U}^{m}\}$ are unique.

\textbf{Part II: Uniqueness of the approximate solutions $\{C^{n+1}, \bm{V}^{n+1}\}$.} 
Assume that  for some $m \geq 0$, $t_{p}^{m} < t_{c}^{n+1} \leq t_{p}^{m+1}$ and $\{P^{m}, \bm{U}^{m}\}$ are uniquely solved.
It follows from \eqref{scm:cv:2}--\eqref{scm:cv:5} that
\begin{equation}\label{exit:4}
    W^{\kappa,n+1} = \big[U_{\#}^{\kappa} \mal{T}_{\kappa} C - D \mal{L}_{\kappa}^{-1}  \de_{\kappa} C \big]^{n+1}, \quad \kappa = x,y. 
\end{equation}
Substituting \eqref{exit:4} into \eqref{scm:cv:1}, and applying the operator $\mal{L}^{-1}$ to both sides of the resulting equation, we get
\begin{equation}\label{exit:5}
    \f{2}{\De t_{c}}  \big[\phi( \ovl{C}^{n+1/2} - C^{n})\big] + \sum_{\kappa=x,y} \mal{L}_{\kappa}^{-1} \de_{\kappa} \beta^{\kappa}(\ovl{C}^{n+1/2})  - q_{P}^{n+1/2} \ovl{C}^{n+1/2} = \big[ c_{I}q_{I} \big]^{n+1/2},
\end{equation}
where 
\begin{equation*}
    \begin{aligned}
        \beta^{\kappa}(\zeta):=  U_{\#}^{\kappa,n+1} \mal{T}_{\kappa} \zeta - D \mal{L}_{\kappa}^{-1} \de_{\kappa} \zeta 
         -\f{1}{2} (U_{\#}^{\kappa,n+1}-U_{\#}^{\kappa,n}) \mal{T}_{\kappa}C^{n},~~\kappa=x,y.
    \end{aligned}
\end{equation*}

Similar to the proof in Part I, in order to show the existence of solution $\ovl{C}^{n+1/2}$ to \eqref{exit:5},  for given $\bm{U}_{\#}^{n+1}$, $\bm{U}_{\#}^{n}$, and $C^{n}$, we introduce a continuous operator  $G(\xi): \mal{P} \rightarrow \mal{P} $ such that
\begin{equation*}
	\begin{aligned}
		\big[G(\xi)\big]_{i,j}
		:= \f{2}{\De t_{c}} [ \phi( \xi - C^{n} )]_{i,j} + \sum_{\kappa=x,y} \mal{L}_{\kappa}^{-1} \de_{\kappa} \beta^{\kappa}(\xi)  - \big[  q_{P}^{n+1/2} \xi \big]_{i,j}- \big[ c_{I}q_{I} \big]_{i,j}^{n+1/2}.
	\end{aligned}
\end{equation*}
Then, taking the inner product of $G(\xi)$ with $\xi$ in the sense of $(\cdot,\cdot)_{\rm M}$ gives us
\begin{equation}\label{exit:6}
		\left( G(\xi), \xi \right)_{\rm M} 
		= \f{2}{\De t_{c}}  \left(  \phi \big[\xi - C^{n}\big], \xi \right)_{\rm M} 
	  - \left( q_{P}^{n+1/2} \xi, \xi \right)_{\rm M} - \left( [ c_{I} q_{I} ]^{n+1/2}, \xi  \right)_{\rm M} + I_{5},
\end{equation}
where $I_{5}:=  \left( \mal{L}_{x}^{-1} \de_{x} \beta^{x}(\xi) + \mal{L}_{y}^{-1} \de_{y}\beta^{y}(\xi), \xi \right)_{\rm M}$. By Lemmas \ref{lem:dis} and \ref{lem:op:L}--\ref{lem:op:t}, the Cauchy-Schwarz inequality and the boundedness result in \eqref{bound:udlU}, it can be seen that $I_{5}$ can be bounded as
\begin{equation}\label{exit:7}
	\begin{aligned}
		I_{5}
        &=  \sum_{\kappa=x,y}  \left(   D \mal{L}_{\kappa}^{-1} \de_{\kappa} \xi -U_{\#}^{\kappa,n+1} \mal{T}_{\kappa} \xi  
         +\f{1}{2} (U_{\#}^{\kappa,n+1}-U_{\#}^{\kappa,n}) \mal{T}_{\kappa}C^{n} , \mal{L}_{\kappa}^{-1}\de_{\kappa} \xi \right)_{\kappa} \\
        & \geq \sum_{\kappa=x,y} \| \mal{L}_{\kappa}^{-1} \de_{\kappa} \xi\|_{\kappa} 
             \left( \alpha_{m} \phi_{*} \| \mal{L}_{\kappa}^{-1} \de_{\kappa} \xi\|_{\kappa} -   \| U_{\#}^{\kappa,n+1} \mal{T}_{\kappa} \xi \|_{\kappa} +  \f{1}{2} \| (U_{\#}^{\kappa,n+1}-U_{\#}^{\kappa,n}) \mal{T}_{\kappa}C^{n} \|_{\kappa} \right)\\
		& \geq \f{\alpha_{m} \phi_{*}}{2} | \xi |_{1}^{2} - K_{4} \left(\| \xi \|_{\rm M}^{2} + \| C^{n}\|_{\rm M}^{2}\right),
	\end{aligned}
\end{equation}
where the constant $K_{4}:=\f{225 (K_{1}+1)^{2} }{8\alpha_{m} \phi_{*}}$.

Now, substituting \eqref{exit:7} into \eqref{exit:6}, and using Lemma \ref{lem:op:L} and the Cauchy-Schwarz inequality, we obtain
\begin{equation*}
    \begin{aligned}
		&\left( G(\xi), \xi \right)_{\rm M} \\
		& \geq \f{1}{\De t_{c}} \left( \phi_{*}\| \xi \|_{\rm M}^{2} - \phi^{*} \| C^{n} \|_{\rm M}^{2} \right) - \left( K_{2} + K_{4} + 1/4 \right)\| \xi\|_{\rm M}^{2}  - K_{4} \| C^{n}\|_{\rm M}^{2}  - \| [ c_{I} q_{I} ]^{n+1/2} \|_{\rm M}^{2}  \\
        & = \f{1}{\De t_{c}} \left( \big(\phi_{*} - \De t_{c}[K_{2} + K_{4} + 1] \big)\| \xi \|_{\rm M}^{2}  - (\phi^{*} +\De t_{c} K_{4} ) \| C^{n} \|_{\rm M}^{2} - \De t_{c} \| [ c_{I} q_{I} ]^{n+1/2} \|_{\rm M}^{2} \right)    \\
		& \geq \f{1}{2\De t_{c}} \left( \phi_{*} \| \xi \|_{\rm M}^{2}  - \left( 2\phi^{*} + \phi_{*} \right) \|  C^{n}\|_{\rm M}^{2} - 4\phi_{*}\| [ c_{I} q_{I} ]^{n+1/2} \|_{\rm M}^{2}\right) \ge 0,
    \end{aligned}
\end{equation*}
for sufficiently small $\De t_{c} \leq \tau_{*} := \phi_{*}/(2(K_{2}+K_{4}+1/4))$ and $\| \xi \|_{\rm M}^{2}=\left( 2 \phi^{*}/\phi_{*} + 1\right)\|C^{n}\|_{\rm M}^{2} +  4\| [ c_{I} q_{I} ]^{n+1/2} \|_{\rm M}^{2}$. Thus, the Browder's fixed point theorem shows that there exists at least one solution $\xi^{*} \in \mal{P}$ such that $G(\xi^{*})=0$, which means that the equation \eqref{exit:5} is solvable, and thus, $C^{n+1}=2\ovl{C}^{n+1/2}-C^{n}$ exists for small $\De t_{c}$. Moreover, the explicit formulas \eqref{scm:cv:2}--\eqref{scm:cv:3} imply $\bm{V}^{n+1}$ also exists. Therefore, it follows that the solutions $\{C^{n+1},\bm{V}^{n+1}\}$ to \eqref{scm:cv:1}--\eqref{scm:cv:5} also exist.

Finally, analogous to the uniqueness proof of $\{P^{m},\bm{U}^{m}\}$, the uniqueness of solutions $\{C^{n+1},\bm{V}^{n+1}\}$ can be derived in a similar way. 
\end{proof}

\section{Numerical experiments}\label{sec:nx}
In this section, we present several numerical examples to demonstrate the accuracy, efficiency, and reliability of the proposed scheme. Moreover, as pointed out in Theorem \ref{thm:mass} that the developed MC-MTS-CBCFD scheme \eqref{scm:pu:1}--\eqref{comP} is mass-conserving in the discrete setting. Therefore,  we also test the mass error measured by
\begin{equation*}
	\mal{E}^{n}  := \sum_{i=1}^{N_{x}} \sum_{j=1}^{N_{y}} h_1 h_2 \mal{L} \,\Big(  \phi C_{i,j}^{n} -  \phi C_{i,j}^{0} - \De t_{c} \sum_{\ell=0}^{n-1} \big[q_{P} \ovl{C} + c_{I}q_{I} \big]_{i,j}^{\ell+1/2}  \Big).
\end{equation*}
\begin{table}[!htbp]\small
   \setlength{\abovecaptionskip}{0.05cm} 
\centering
\caption{Errors and convergence orders with  $N_{c} = N_{x}^{2}= N_{y}^{2}$ and $Q=1$ at $T=1$ for Example \ref{exm:e1}.}
\label{tab:e1:1}
\setlength{\tabcolsep}{2.5mm}
\begin{tabular}{ccccccccc}
\toprule
{$N_{x}$} & $\|e_{c}\|_{\rm M}$ & Order & $\|e_{p}\|_{\rm M}$ & Order & $\|\bm{e}_{\bm{u}}\|_{\rm T}$ & Order & $|e_{p}|_{1}$ & Order  \\
\midrule
20  & 3.86e-05 & ---  & 1.02e-05 & ---  & 1.62e-05 & ---  & 9.01e-05 & ---  \\
30  & 7.62e-06 & 4.00 & 2.01e-06 & 4.01 & 3.28e-06 & 4.00 & 1.78e-05 & 4.00 \\
40  & 2.41e-06 & 4.00 & 6.36e-07 & 4.00 & 1.01e-06 & 4.01 & 5.65e-06 & 3.99 \\
50  & 9.88e-07 & 4.00 & 2.60e-07 & 4.01 & 4.14e-07 & 4.00 & 2.31e-06 & 4.01 \\
60  & 4.76e-07 & 4.01 & 1.26e-07 & 3.97 & 2.08e-07 & 3.99 & 1.12e-06 & 3.97 \\
\bottomrule
\end{tabular}
\end{table}
\begin{table}[!htbp]\small
   \setlength{\abovecaptionskip}{0.05cm} 
\centering
\caption{Errors and convergence orders with $N_{c} = N_{x}^{2}= N_{y}^{2}$ and $Q=10$ at $T=1$ for Example \ref{exm:e1}.}
\label{tab:e1:10}
\setlength{\tabcolsep}{2.5mm}
\begin{tabular}{ccccccccc}
\toprule
{$N_{x}$} & $\|e_{c}\|_{\rm M}$ & Order & $\|e_{p}\|_{\rm M}$ & Order & $\|\bm{e}_{\bm{u}}\|_{\rm T}$ & Order & $|e_{p}|_{1}$ & Order \\
\midrule
20  & 3.90e-05 & ---  & 1.02e-05 & ---  & 1.62e-05 & ---  & 9.01e-05 & ---  \\
30  & 7.71e-06 & 4.00 & 2.01e-06 & 4.01 & 3.29e-06 & 4.00 & 1.78e-05 & 4.00 \\
40  & 2.43e-06 & 4.01 & 6.36e-07 & 4.00 & 1.01e-06 & 4.01 & 5.65e-06 & 3.99 \\
50  & 9.94e-07 & 4.01 & 2.61e-07 & 3.99 & 4.14e-07 & 4.00 & 2.31e-06 & 4.01 \\
60  & 4.78e-07 & 4.02 & 1.26e-07 & 3.99 & 2.09e-07 & 3.99 & 1.12e-06 & 3.97 \\
\bottomrule
\end{tabular}
\end{table}
\begin{table}[!htbp]\small
   \setlength{\abovecaptionskip}{0.05cm} 
\centering
\caption{Errors and convergence orders with  $N_{c} = N_{x}^{2}= N_{y}^{2}$ and $Q=20$ at $T=1$ for Example \ref{exm:e1}.}
\label{tab:e1:20}
\setlength{\tabcolsep}{2.5mm}
\begin{tabular}{ccccccccc}
\toprule
{$N_{x}$} & $\|e_{c}\|_{\rm M}$ & Order & $\|e_{p}\|_{\rm M}$ & Order & $\|\bm{e}_{\bm{u}}\|_{\rm T}$ & Order & $|e_{p}|_{1}$ & Order \\
\midrule
20  & 5.25e-05 & ---  & 1.02e-05 & ---  & 1.63e-05 & ---  & 9.01e-05 & ---  \\
30  & 9.99e-06 & 4.09 & 2.01e-06 & 4.01 & 3.21e-06 & 4.01 & 1.78e-05 & 4.00 \\
40  & 3.07e-06 & 4.10 & 6.36e-07 & 4.00 & 1.01e-06 & 4.01 & 5.65e-06 & 3.99 \\
50  & 1.22e-06 & 4.14 & 2.61e-07 & 3.99 & 4.15e-07 & 3.99 & 2.31e-06 & 4.01 \\
60  & 5.72e-07 & 4.15 & 1.26e-07 & 3.99 & 2.09e-07 & 4.00 & 1.12e-06 & 3.97 \\
\bottomrule
\end{tabular}
\end{table}
\subsection{Accuracy and efficiency tests}
\begin{example}\label{exm:e1}  For the first example, the spatial domain is selected as $\Omega=(0,1)^2$. we consider model \eqref{mod:PDE} with the following manufactured exact solutions
 \[
  c(\bm{x},t) = \sin(5\pi t/2+\pi/4)\cos(2\pi x)\cos(2 \pi y), ~~
  p(\bm{x},t) = \sin(\pi t/2+\pi/4) \sin(2\pi x) \sin(2\pi y),
  \]
  \[
  \bm{u}(\bm{x},t) = \sin(\pi t/2+\pi/4)\,( \sin(2\pi x) \cos(2\pi y),~  \cos(2\pi x) \sin(2\pi y))^\top,
 \]
 where the parameters in model \eqref{mod:PDE} are chosen as follows:
 \[
    \phi(\bm{x}) = (\cos(2\pi(x+y)) + 2 )/4,~~ \mu(c) = 1 + c^{2},~~ k(\bm{x}) = ( \sin( 2 \pi (x+y) ) + 2 )^2,
 \]
 \[
     q_{P}(\bm{x}) =   \sin(2\pi (x + y + t ) ) - 2,~~ \alpha_{m}(\bm{x}) = \sin(2\pi(x+y)) + 2,~~ D(\bm{x}) = \phi(\bm{x}) \alpha_{m}(\bm{x}).
 \]
\end{example}
\begin{table}[!htbp]\small
   \setlength{\abovecaptionskip}{0.05cm} 
\centering
\caption{Comparisons of CPU time  with  $N_{c} = N_{x}^{2}= N_{y}^{2}$ and $N_{p}=N_{c}/Q$ at $T=1$ for Example \ref{exm:e1}.}
\label{tab:e1:cpu1}
\setlength{\tabcolsep}{6.5mm}
\begin{tabular}{cccccc}
\toprule
$N_{x}$ & $N_{c}$ &$Q=1$ & $Q=10$ & $Q=20$   \\
\midrule
40  & 1,600 & 79~s~(135~s)    & 8~s~(133~s)     & 4~s~(134~s)       \\
50  & 2,500 & 201~s~(358~s)  & 20~s~(353~s)    & 10~s~(353~s)    \\
60  & 3,600 & 461~s~(878~s)  & 46~s~(877~s)    & 23~s~(876~s)    \\
\bottomrule
\end{tabular}
\end{table}
\begin{table}[htbp]\small
   \setlength{\abovecaptionskip}{0.05cm} 
\centering
\caption{Comparisons of CPU time with  $\De t_{c} =0.005$,  $N_{x}= N_{y}=50$ for Example \ref{exm:e1}.}
\label{tab:e1:cpu2}
\setlength{\tabcolsep}{4.5mm}
\begin{tabular}{cccccc}
\toprule
$T$ & $Q$ & $\|e_{c}\|_{\rm M}$ & $\|e_{p}\|_{\rm M}$ & $\|\bm{e}_{\bm{u}}\|_{\rm T}$  & CPU time   \\
\midrule
    &1    & 6.81e-05 & 7.36e-07 & 7.50e-06 & 159~s~(361~s) \\
10  &10   & 7.44e-05 & 9.24e-07 & 7.39e-06 & 16~s~(360~s) \\
    &100  & 6.54e-05 & 7.36e-07 & 7.07e-06 & 2~s~(360~s) \\
\midrule
    &1    & 6.90e-05 & 7.35e-07 & 7.60e-06 & 1,691~s~(3,634~s) \\
100 &10   & 7.73e-05 & 5.69e-07 & 7.99e-06 & 156~s~(3,484~s) \\
    &100  & 7.39e-05 & 7.30e-07 & 7.71e-06 & 15~s~(3,405~s) \\
\midrule
     &1    & 6.90e-05 & 7.35e-07 & 7.60e-06 & 16,393~s~(36,357~s) \\
1000 &10   & 7.73e-05 & 5.69e-07 & 7.99e-06 & 1,608~s~(36,093~s) \\
     &100  & 7.39e-05 & 7.30e-07 & 7.71e-06 & 161~s~(35,870~s) \\
\bottomrule
\end{tabular}
\end{table}

The goal of this example is three-fold: accuracy test, efficiency comparisons of multi-time-step strategy, and mass conservation verification. First, we test the numerical spatial and temporal accuracy by setting $N_{c}=N_{x}^{2}=N_{y}^{2}$, in which the discrete $ L^{2}$ and $H^{1}$ errors are measured at $T=1$. Tables \ref{tab:e1:1}--\ref{tab:e1:20} present the errors and convergence orders for the concentration $c$, the pressure $p$ and velocity $\bm{u}$ for different $Q$, utilizing the MC-MTS-CBCFD scheme. Next, we show the CPU times cost by the velocity/pressure equation as well as the concentration equation with respect to $Q$ in Table \ref{tab:e1:cpu1}, where the parenthesized values indicate the running CPU time of the concentration equation. In addition, to further show the advantages of the proposed algorithm in long-term simulation, we set $N_{x}=N_{y}=50$ and $\De t_{c} = 0.05$, and list the numerical errors and consumed CPU time at $T=10$, $100$ and $1000$ for different $Q$ in Table \ref{tab:e1:cpu2}. Finally, the evolutions of mass errors with $N_{c}=N_{x}^{2}=N_{y}^{2}=60^{2}$ and $N_{x}=N_{y}=50$, $\De t_{c} = 0.005$ for different $Q$ are depicted Figures \ref{fig:e1:mass} and \ref{fig:e1:mass:L}, respectively. Basically, the key observations are summarized as follows:
\begin{figure}[!htbp]
	\centering
	\subfigure[$Q=1$]{
		\includegraphics[width=0.3\textwidth]{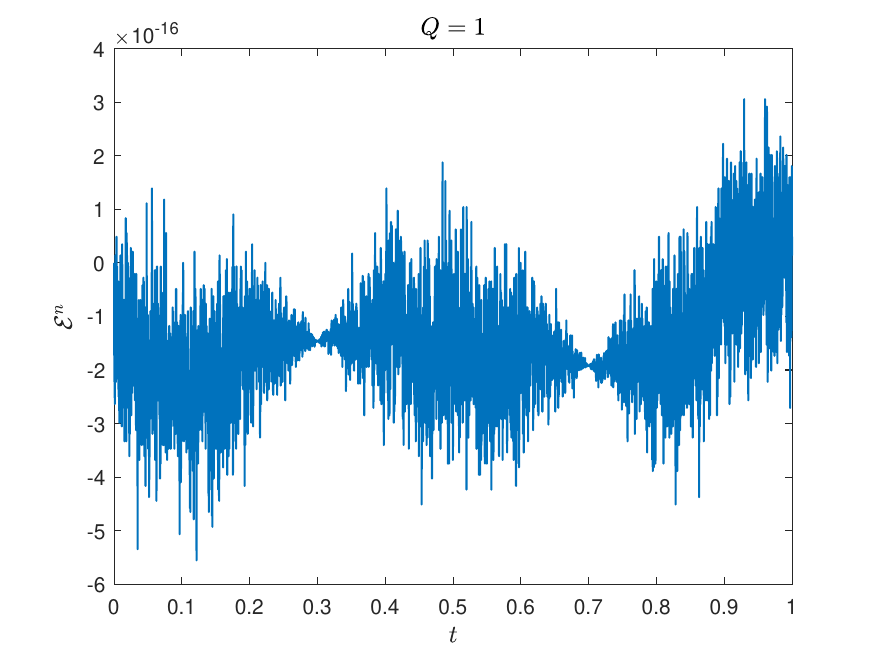}
		\label{fig:e1:mass:1}
	}
	\subfigure[$Q=10$]{
		\includegraphics[width=0.3\textwidth]{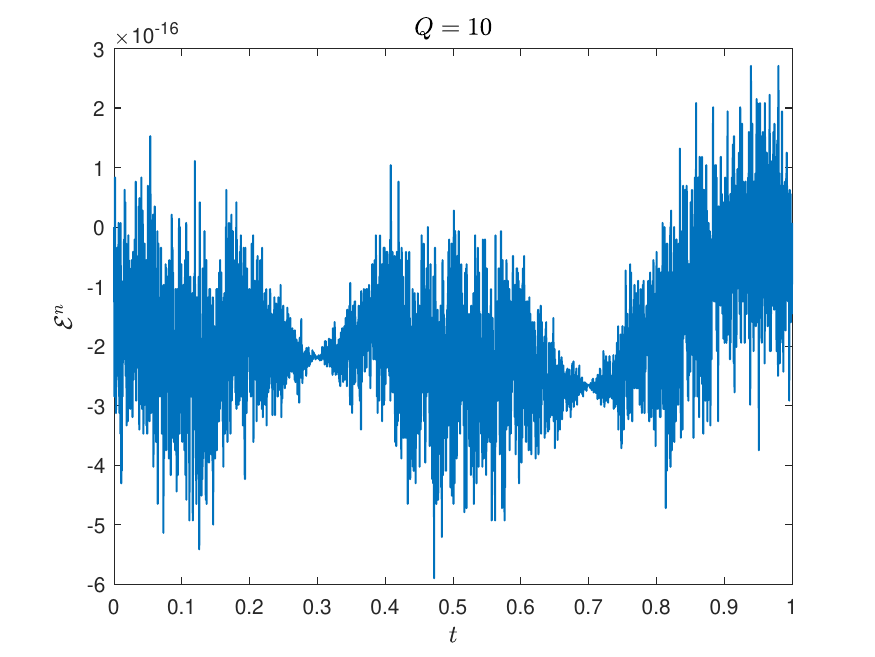}
		\label{fig:e1:mass:10}
	}
	\subfigure[$Q=20$]{
		\includegraphics[width=0.3\textwidth]{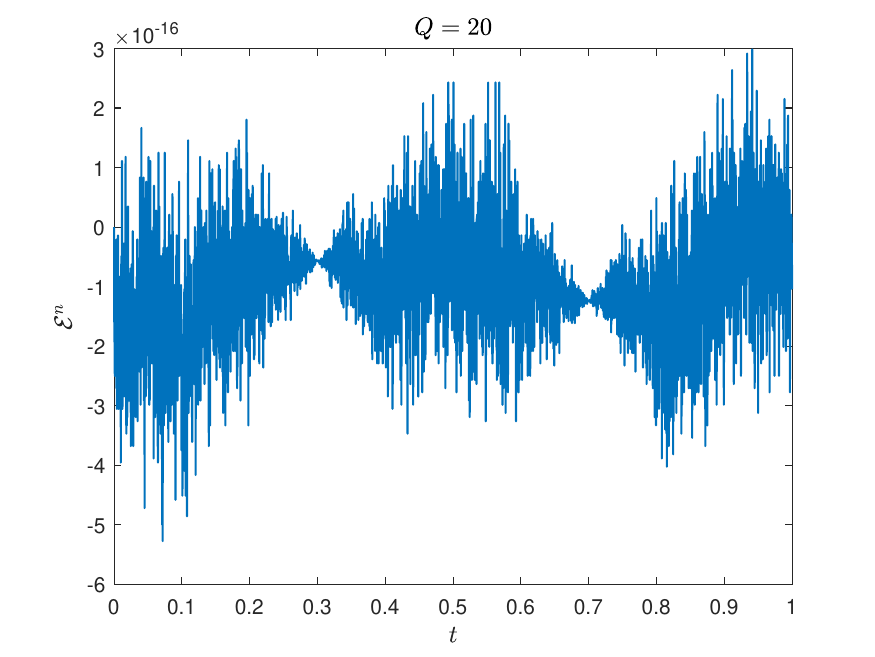}
		\label{fig:e1:mass:20}
	}
	\caption{The mass errors with $N_{c}=N_{x}^{2}=N_{y}^{2}=60^{2}$ and different $Q$ for Example \ref{exm:e1}.}
	\label{fig:e1:mass}
   \setlength{\belowcaptionskip}{0.05cm}  
\end{figure}
\begin{figure}[!htbp]
	\centering
	\subfigure[$Q=1$]{
		\includegraphics[width=0.3\textwidth]{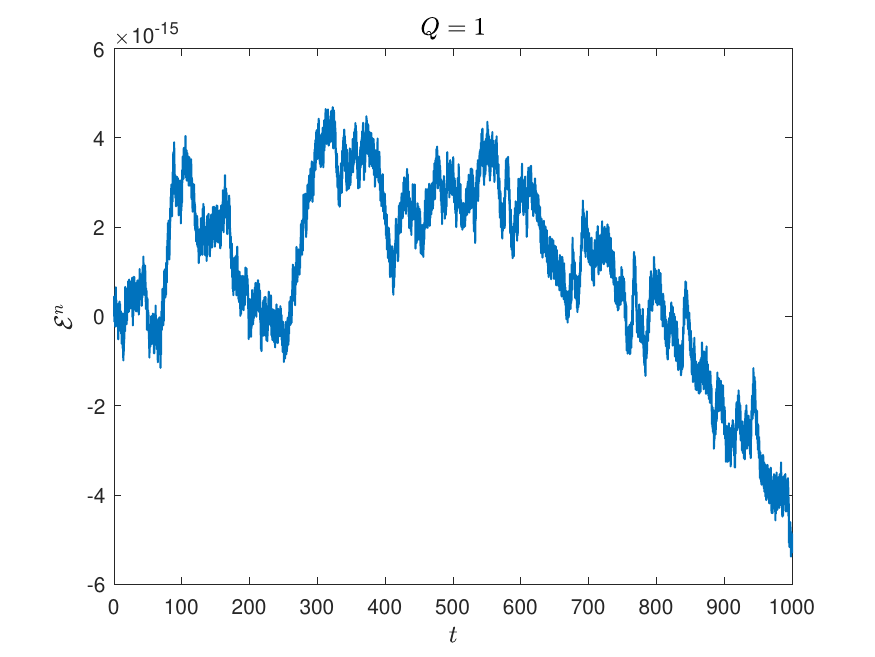}
		\label{fig:e1:mass:L:1}
	}
	\subfigure[$Q=10$]{
		\includegraphics[width=0.3\textwidth]{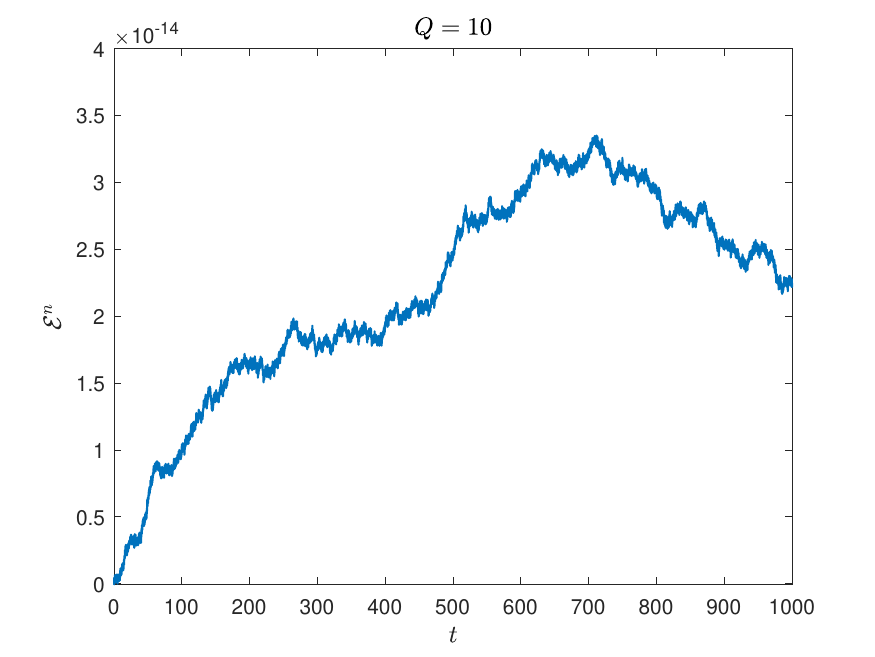}
		\label{fig:e1:mass:L:10}
	}
	\subfigure[$Q=100$]{
		\includegraphics[width=0.3\textwidth]{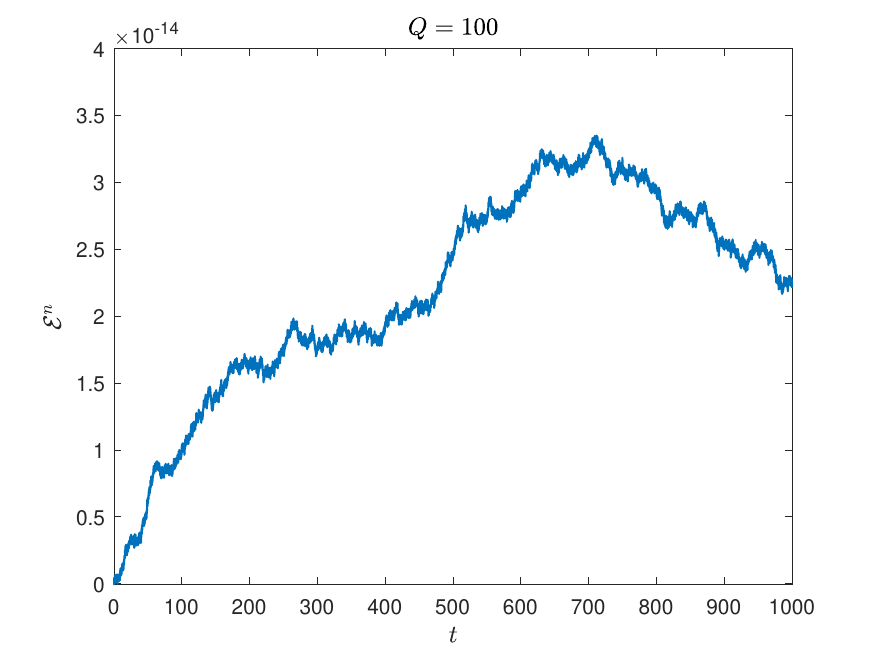}
		\label{fig:e1:mass:L:100}
	}
	\caption{The mass errors with $N_{x}=N_{y}=50$, $\De t_{c} = 0.005$ and different $Q$ for Example \ref{exm:e1}.}
	\label{fig:e1:mass:L}
   \setlength{\belowcaptionskip}{0.05cm}  
\end{figure}

\begin{itemize}
    \item[(i)]  
    The proposed MC-MTS-CBCFD scheme can reach second-order temporal accuracy and fourth-order spatial accuracy for the concentration $c$, pressure $p$ and velocity $\bm{u}$. This is totally consistent with the theoretical conclusion established in Theorem \ref{thm:err}, which shows the strong merits of the developed method that can simultaneously approximate all variables with the same fourth-order accuracy. Thus, the proposed method can avoid the reduction in convergence order that occurs in traditional direct FD methods to approximate the velocity $\bm{u}$ by differentiation of the pressure $p$.
    
    \item[(ii)]  The implementation of the multi-time-step strategy demonstrates a remarkable improvement in computational efficiency, as clearly shown in Tables \ref{tab:e1:cpu1}--\ref{tab:e1:cpu2}. Specifically, the proposed algorithm achieves a substantial reduction in computational cost for the velocity/pressure system, requiring only $1/Q$ of the CPU time compared to the conventional single-time-step approach (where $Q=1$). Thus, the proposed multi-time-step algorithm is particular suitable for long-term simulation of the incompressible miscible displacement model.

    \item[(iii)] It is worth noting that the multi-time-step algorithm can not only improve the computational efficiency, but also generate the same magnitude computational errors as the single-time-step algorithm, see Tables \ref{tab:e1:1}--\ref{tab:e1:20}. This favorable property remains valid in long-term simulations, see Table \ref{tab:e1:cpu2}, where nearly identical errors are produced with $Q=10$, $Q=100$ and the single-time-step algorithm ($Q=1$) at $T=1000$. This again shows strong reliability of the MC-MTS-CBCFD scheme.

    \item[(iv)] The proposed method is indeed mass-conserving. As seen from Figure \ref{fig:e1:mass}, the mass errors almost reach the machine precision for both the single-time-step approach (i.e., $Q=1$) and the multi-time-step approach (i.e., $Q=10$, $Q=20$). Furthermore, Figure \ref{fig:e1:mass:L} illustrates that the proposed method can still maintain the mass conservation even for long-term simulations.
\end{itemize}

\begin{example}\label{exm:e2} 
In this example, we consider a generalization of the proposed MC-MTS-CBCFD algorithm applied to the incompressible miscible displacement model \eqref{mod:PDE} under no-flow/homogeneous Neumann boundary conditions \eqref{mod:bc}.
As discussed in \cite{LBY13,LWZ25} that the boundary conditions \eqref{mod:bc} imply that $\bm{v}\cdot \bm{n}= \nabla c \cdot \bm{n} = 0$ on $\p \Omega \times [0,T]$. In this case, we have to slightly modify the fourth-order difference operator and interpolation operators defined in \eqref{op:L}--\eqref{op:h} near the boundary as follows:
    \begin{align*}
       [ \hat{\mal{L}}_{x} \omega ]_{i,j} &:= \frac{1}{24}\left\{
         \begin{aligned}
           &26\omega_{1,j} - 5\omega_{2,j} + 4\omega_{3,j} - \omega_{4,j}, & i = 1, \\
           &24[\mal{L}_{x} \omega]_{i,j}, & 2\leq i \leq N_{x}-1, \\
           &26\omega_{N_{x},j} - 5\omega_{N_{x}-1,j} + 4\omega_{N_{x}-2,j} - \omega_{N_{x}-3,j}, & i = N_{x},
        \end{aligned} \right.\\
      [\hat{\mal{T}}_{x} \omega ]_{i+1/2} &:=\frac{1}{16} \left\{
         \begin{aligned}
         &35\omega_{1} - 35 \omega_{2} + 21 \omega_{3}  -5 \omega_{4}, & i = 0, \\
           &5\omega_{1} + 15 \omega_{2} - 5\omega_{3} + \omega_{4}, & i = 1, \\
           &16[ \mal{T}_{x}\omega ] _{i+1/2}, & 2\leq i \leq N_{x}-2, \\
           &5\omega_{N_{x}} + 15\omega_{N_{x}-1} - 5\omega_{N_{x}-2} + \omega_{N_{x}-3}, & i = N_{x}-1,\\
           &35\omega_{N_{x}} - 35\omega_{N_{x}-1} + 21\omega_{N_{x}-2} - 5\omega_{N_{x}-3}, & i = N_{x},
        \end{aligned} \right.\\
        [\hat{\mal{T}}_{x}^{*} \omega ]_{i} &:= \frac{1}{16} \left\{
         \begin{aligned}
           &5\omega_{1/2} + 15 \omega_{3/2} -  5\omega_{5/2} +\omega_{7/2}, & i = 1, \\
           &16[ \mal{T}_{x}^{*}\omega ] _{i}, & 2\leq i \leq N_{x}-1, \\
           &5\omega_{N_{x}+1/2} + 15\omega_{N_{x}-1/2}  - 5\omega_{N_{x}-3/2} + \omega_{N_{x}-5/2}, & i = N_{x}.
        \end{aligned} \right.
    \end{align*}
 Besides, the operators $\hat{\mal{L}}_{y} $, $\hat{\mal{T}}_{y}$ and $\hat{\mal{T}}_{y}^{*}$ are defined similarly. For simplicity, we denote $\hat{\mal{L}}:=\hat{\mal{L}}_{x}\hat{\mal{L}}_{y}$, $\hat{\mal{H}}_{x} := \hat{\mal{T}}_{y}^{*}  \hat{\mal{T}}_{x}  =   \hat{\mal{T}}_{x} \hat{\mal{T}}_{y}^{*}$ and $\hat{\mal{H}}_{y} := \hat{\mal{T}}_{x}^{*}  \hat{\mal{T}}_{y}  =   \hat{\mal{T}}_{y} \hat{\mal{T}}_{x}^{*}$.
 \end{example}
    
Define $ \Pi_{x}^{\circ}:= \{x_{i+1/2}\}_{i=1}^{N_{x}-1}$ and $\Pi_{y}^{\circ}:= \{y_{j+1/2}\}_{j=1}^{N_{y}-1}$. We take Step 3 of the MC-MTS-CBCFD scheme as an example to illustrate the new scheme under no-flow boundary conditions: Find $\{C^{n+1}, V^{x,n+1}, V^{y,n+1}, W^{x,n+1}, W^{y,n+1}\} \in \mal{P}^{\circ} \times \mal{U}^{\circ} \times \mal{V}^{\circ}  \times \mal{U}^{\circ} \times \mal{V}^{\circ}$ such that
       \begin{align*}
            &\Big[\hat{\mal{L}}[\phi d_{t} C] + \hat{\mal{L}}_{y} \de_{x} \ovl{W}^{x} + \hat{\mal{L}}_{x} \de_{y} \ovl{W}^{y} - \hat{\mal{L}} [q_{P} \ovl{C}] \Big]_{i,j}^{n+1/2} =\hat{\mal{L}} [c_{I}q_{I}]_{i,j}^{n+1/2}, & \text{on}~ \Pi_{x}^{*}\times \Pi_{y}^{*}, \\
            &\Big[ \de_{x} C + \mal{L}_{x} V^{x} \Big]_{i+1/2,j}^{n+1} = 0, &\text{on}~ \Pi_{x}^{\circ}\times \Pi_{y}^{*},  \\
            &\Big[ \de_{y} C + \mal{L}_{y} V^{y} \Big]_{i,j+1/2}^{n+1} = 0,&\text{on}~ \Pi_{x}^{*}\times \Pi_{y}^{\circ}, \\
            &\Big[ W^{x} - U_{\#}^{x} \hat{\mal{T}}_{x} C - \bm{D}^{x}( U_{\#}^{x}, \hat{\mal{H}}_{x} U_{\#}^{y}) \cdot \big( V^{x}, 
                \hat{\mal{H}}_{x} V^{y} \big) \Big]_{i+1/2,j}^{n+1} = 0, 
             &\text{on}~ \Pi_{x}^{\circ}\times \Pi_{y}^{*}, \\
            &\Big[ W^{y} - U_{\#}^{y} \hat{\mal{T}}_{y} C - \bm{D}^{y}( \hat{\mal{H}}_{y} U_{\#}^{x},  U_{\#}^{y}) \cdot \big( \hat{\mal{H}}_{y} V^{x},  V^{y}  \big)\Big]_{i,j+1/2}^{n+1} = 0,
              &\text{on}~ \Pi_{x}^{*}\times \Pi_{y}^{\circ}, 
    \end{align*}
    where 
    \begin{equation*}
        \begin{aligned}
            & \mal{P}^{\circ}:= \left\lbrace g_{i,j} ~ | ~ (x_{i},y_{j})\in \Pi_{x}^{*}\times \Pi_{y}^{*}\right\rbrace,\\
            & \mal{U}^{\circ}:= \left\lbrace g_{i+1/2,j} ~ | ~   (x_{i+1/2},y_{j})\in \Pi_{x}\times \Pi_{y}^{*},~ g_{1/2,j}=g_{N_{x}+1/2,j}=0 \right\rbrace, \\
            & \mal{V}^{\circ}:= \left\lbrace g_{i,j+1/2} ~|~   (x_{i},y_{j+1/2})\in \Pi_{x}^{*}\times \Pi_{y},~ g_{i,1/2}= g_{i,N_{y}+1/2}=0~  \right\rbrace.
        \end{aligned}
    \end{equation*}

In this example, we choose the spatial domain $\Omega= (0,1)^2$, $T=1$, the manufactured exact solutions
 \[
  c(\bm{x},t) = 2e^{t}( x^2(x-1)^2 + y^2(y-1)^2),\quad
   p(\bm{x},t) = t^3\sin(\pi x)\sin(\pi y),
\]
\[
   \bm{u}(\bm{x},t) = t^3(x(x-1)(2x-1),~y(y-1)(2y-1))^{\top},
\]
and parameters
\[
     \phi(\bm{x}) = (x+y+1)^2/10,\quad
     k(\bm{x}) = (x+y+1)^3,\quad
     \mu(c) = 1 + c^2,\quad
\]
\[
    q_P(\bm{x},t) = \cos(2\pi(x+y+t)) -2,\quad
    \bm{D}(\bm{x},\bm{u}) = \phi(\bm{x}) ( 0.1\bm{I} + \bm{u} \bm{u}^{\top} ).
\]

\begin{table}[!htbp] \small
   \setlength{\abovecaptionskip}{0.05cm} 
\centering
\caption{Errors and convergence orders with  $N_{c} = N_{x}^{2}= N_{y}^{2}$ and $Q=1$ for Example \ref{exm:e2}.}
\label{tab:e2:1}
\setlength{\tabcolsep}{2.5mm}
\begin{tabular}{ccccccccc}
\toprule
{$N_{x}$} & $\|e_{c}\|_{\rm M}$ & Order  & $\|e_{p}\|_{\rm M}$ & Order & $\|\bm{e}_{\bm{u}}\|_{\rm T}$ & Order & $|e_{p}|_{1}$ & Order \\
\midrule
10  & 7.75e-04 & ---  & 2.00e-05 & ---  & 9.61e-06 & ---  & 6.56e-05 & ---\\
20  & 3.86e-05 & 4.33 & 1.18e-06 & 4.08 & 6.06e-07 & 3.99 & 4.67e-06 & 3.81\\
30  & 8.60e-06 & 3.70 & 2.20e-07 & 4.14 & 1.19e-07 & 4.01 & 9.60e-07 & 3.90\\
40  & 2.59e-06 & 4.17 & 6.70e-08 & 4.13 & 3.48e-08 & 4.27 & 3.09e-07 & 3.94\\
50  & 9.63e-07 & 4.43 & 2.68e-08 & 4.11 & 1.31e-08 & 4.38 & 1.28e-07 & 3.95\\
\bottomrule
\end{tabular}
\end{table} 
\begin{table}[!htbp] \small
   \setlength{\abovecaptionskip}{0.05cm} 
\centering
\caption{Errors and convergence orders with  $N_{c} = N_{x}^{2}= N_{y}^{2}$ and $Q=10$ for Example \ref{exm:e2}.}
\label{tab:e2:10}
\setlength{\tabcolsep}{2.5mm}
\begin{tabular}{ccccccccc}
\toprule
{$N_{x}$} & $\|e_{c}\|_{\rm M}$ & Order & $\|e_{p}\|_{\rm M}$ & Order & $\|\bm{e}_{\bm{u}}\|_{\rm T}$ & Order & $|e_{p}|_{1}$ & Order \\
\midrule
10  & 1.39e-03 & ---  & 1.53e-05 & ---  & 2.11e-05 & ---  & 6.56e-05 & --- \\
20  & 9.68e-05 & 3.84 & 9.23e-07 & 4.05 & 1.22e-06 & 4.11 & 4.67e-06 & 3.81\\
30  & 2.02e-05 & 3.86 & 1.72e-07 & 4.14 & 2.34e-07 & 4.07 & 9.60e-07 & 3.90\\
40  & 6.40e-06 & 4.00 & 5.22e-08 & 4.14 & 7.22e-08 & 4.09 & 3.09e-07 & 3.94\\
50  & 2.58e-06 & 4.07 & 2.08e-08 & 4.12 & 2.89e-08 & 4.10 & 1.28e-07 & 3.95\\
\bottomrule
\end{tabular}
\end{table}
\begin{table}[!htbp] \small
   \setlength{\abovecaptionskip}{0.05cm} 
\centering
\caption{Errors and convergence orders with  $N_{c} = N_{x}^{2}= N_{y}^{2}$ and $Q=20$ for Example \ref{exm:e2}.}
\label{tab:e2:20}
\setlength{\tabcolsep}{2.5mm}
\begin{tabular}{ccccccccc}
\toprule
{$N_{x}$} & $\|e_{c}\|_{\rm M}$ & Order & $\|e_{p}\|_{\rm M}$ & Order & $\|\bm{e}_{\bm{u}}\|_{\rm T}$ & Order & $|e_{p}|_{1}$ & Order \\
\midrule
10  & 5.60e-03 & ---  & 1.36e-05 & ---  & 6.81e-05 & ---  & 6.56e-05 & --- \\
20  & 3.40e-04 & 4.04 & 7.32e-07 & 4.22 & 3.76e-06 & 4.18 & 4.67e-06 & 3.81\\
30  & 6.66e-05 & 4.02 & 1.43e-07 & 4.03 & 7.09e-07 & 4.11 & 9.60e-07 & 3.90\\
40  & 2.09e-05 & 4.03 & 4.53e-08 & 4.00 & 2.20e-07 & 4.07 & 3.09e-07 & 3.94\\
50  & 8.53e-06 & 4.02 & 1.87e-08 & 3.97 & 8.91e-08 & 4.05 & 1.28e-07 & 3.95\\
\bottomrule
\end{tabular}
\end{table}
The errors and convergence orders for problem with nonlinear velocity-dependent diffusion-dispersion coefficient are presented in Tables \ref{tab:e2:1}--\ref{tab:e2:20}. We can see that the convergence orders for the approximations of $c$, $p$ and $\bm{u}$ are indeed second-order in time and fourth-order in space. This indicates that our algorithm is also applicable to the model problem with no-flow/homogeneous Neumann boundary conditions and general nonlinear diffusion coefficients.

\subsection{'Real' simulations}
In this subsection, we simulate various incompressible displacement problems in a water-oil system within porous media to systematically assess the numerical performance of the MC-MTS-CBCFD scheme \eqref{scm:pu:1}--\eqref{comP}. We consider one quadrant of a standard five-spot well pattern for the horizontal reservoir of unit thickness with no-flow/homogeneous Neumann boundary conditions, where the injection and production wells are located at opposite corners \cite{WH00,CD07,VPV21,Eymard19}. The spatial domain is $\Omega = (0, 1000)^2$ $\text{ft}^2$, the time period is $[0, T] = [0, 3600]$ days, and the viscosity of the oil is $\mu_0 = 1.0$ cp. The injection well is located at the upper-right corner $(1000,1000)$ of the domain with an injection rate $q_{I} = 30$ $\text{ft}^2/\text{day}$, and the production well is located at the lower-left corner $(0,0)$ with a production rate $q_{P} = - 30$ $\text{ft}^2/\text{day}$. The initial concentration and injection concentration of invading fluid are $c^{o} = 0$ and $c_{I}=1$, respectively. In all simulations, we use a fairly coarse spatial grids with $h^{x}=h^{y}=20$ ft, and take $\De t_{c} = 10$ days, $\De t_{p}=30$ days. 

As shown in \eqref{mod:mu}, the viscosity behavior is primarily governed by the mobility ratio $\mal M$. The bigger $\mal M > 1$, the more mobile invading water may make multiple unstable channels into the oil and tend to form long fingers that grow in length toward the production wells. This phenomenon is called viscous fingering \cite{Ewing83}. Once the channel composed of water extends from the injection well to the production well, the production well will mainly produce water, thereby reducing oil recovery. Meanwhile, dispersion and heterogeneity are also controlling factors in the presence of viscous fingering phenomenon \cite{Ewing83}.
Therefore, four distinct examples with different mobility ratio, dispersion and permeability are tested to observe the fingering phenomenon and displacement efficiency.

\begin{example}[\cite{WH00,CD07,VPV21,Eymard19}]\label{exm:e3}
    In this example, we consider numerical simulation of homogeneous and isotropic porous media with the following parameter configuration:
    \[
        \phi(\bm{x}) = 0.1,\quad k(\bm{x}) = 80~ {\rm md}, \quad \alpha_{m} = 10~ {\rm ft}^{2}/{\rm day}, \quad \alpha_{l} = \alpha_{t} = 0~ {\rm ft}/{\rm day}, \quad \mal M=1.
    \]
\end{example}

Surface and contour plots of the invading fluid (water) concentration at $t = 3$ and $10$ years are presented in Figure \ref{fig:e3}. As shown in Figure \ref{fig:e3:1}--\ref{fig:e3:2}, the concentration distribution of the invading fluid (water) exhibits a distinct pattern of concentric circular contours at $t = 3$ years. These results are physically reasonable for the following reasons: constant viscosity $\mu(c) = \mu_0$, homogeneous permeability field $k$, and only molecular diffusion effects. Furthermore, due to the effect of no-flow boundary conditions and the positions of injection/production wells, as time marches, the invading fluid moves faster along the diagonal (flow direction) of the reservoir, which can be observed in Figure \ref{fig:e3:3}--\ref{fig:e3:4} for $t=10$ years. These results are similar to those obtained in \cite{WH00,CD07,VPV21,Eymard19}.
\begin{figure}[!t]
  \centering
	\subfigure[Surface plot at $t=3$ years]{
		\includegraphics[width=0.44\textwidth]{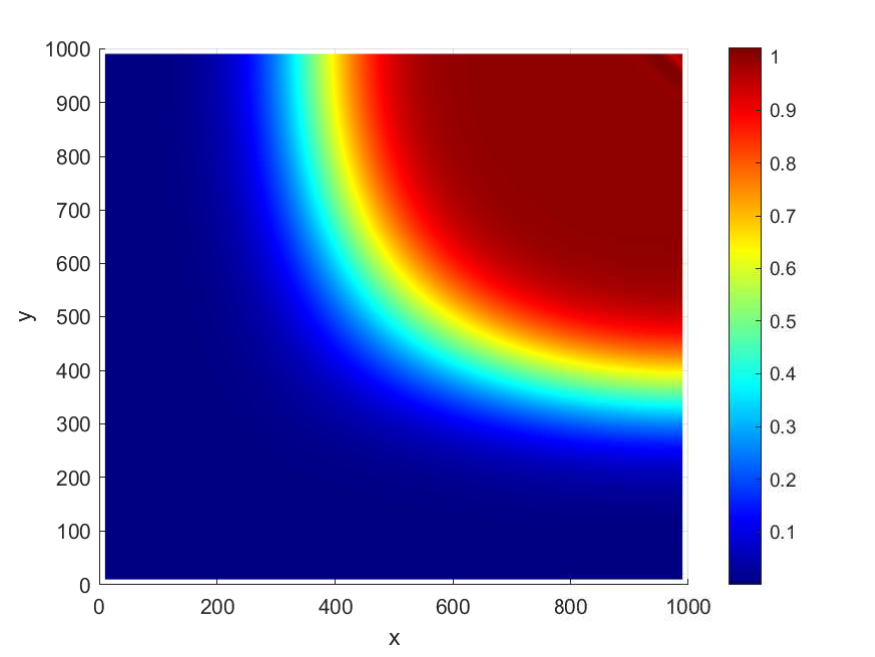}
		\label{fig:e3:1}
	}
	\subfigure[Contour plot at $t=3$ years]{
		\includegraphics[width=0.44\textwidth]{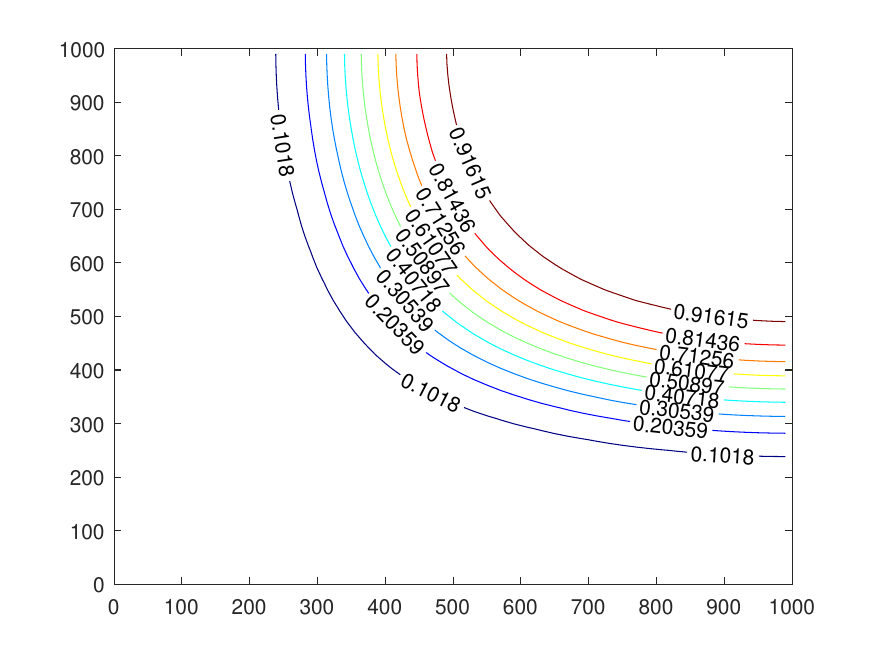}
		\label{fig:e3:2}
	}
	\subfigure[Surface plot at $t = 10$ years]{
		\includegraphics[width=0.44\textwidth]{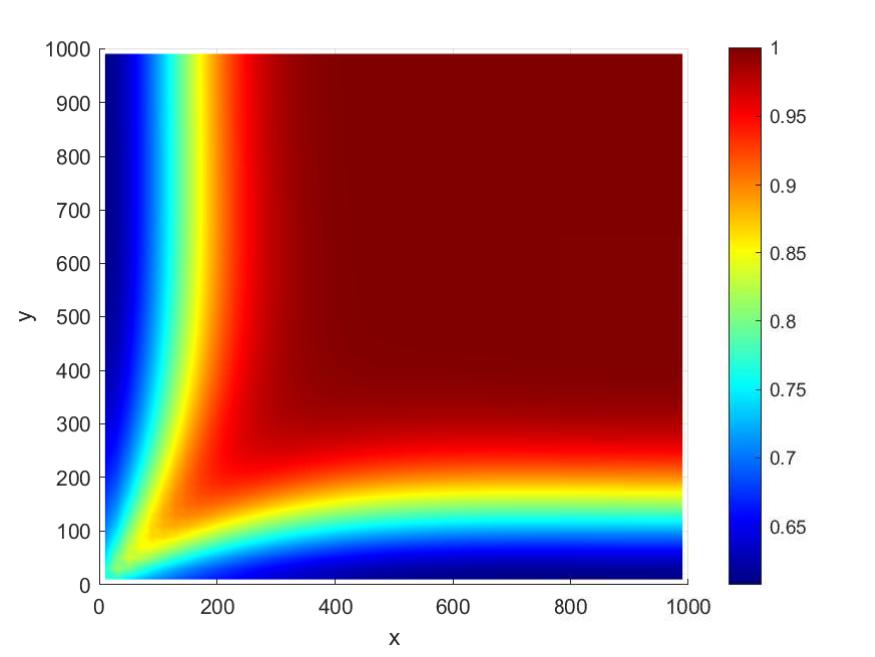}
		\label{fig:e3:3}
	}
	\subfigure[Contour plot at $t = 10$ years]{
		\includegraphics[width=0.44\textwidth]{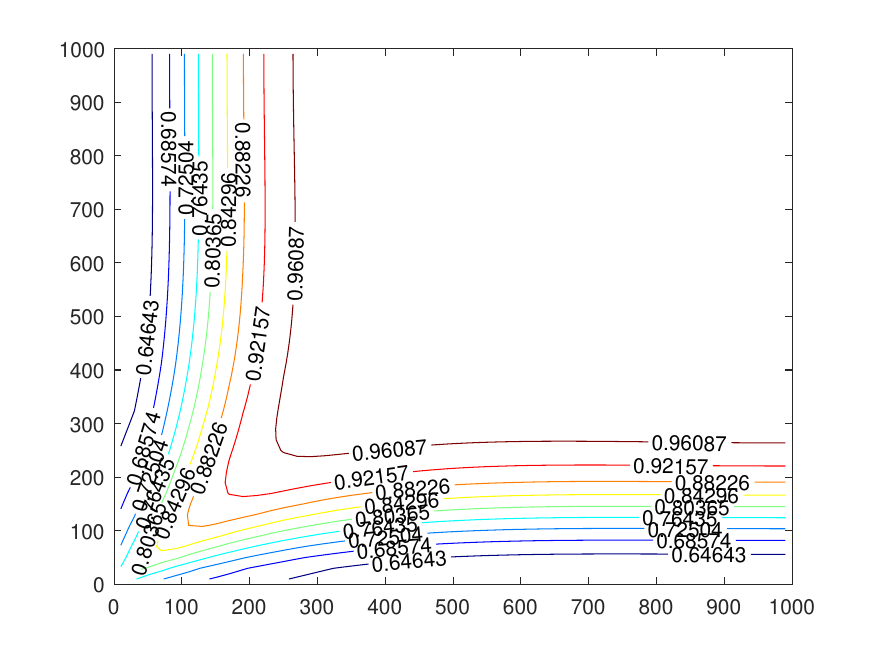}
		\label{fig:e3:4}
	}
	\caption{The concentration of the invading fluid at $t=3$ and $10$ years for Example \ref{exm:e3}.}	\label{fig:e3}
\end{figure}

\begin{example}[\cite{WH00,CD07,VPV21,Eymard19}]\label{exm:e4}	
    In this example, we simulate a challenging displacement scenario with an adverse mobility ratio $\mal M=41$ and an anisotropic dispersion in tensor form. The physical parameters are chosen as follows: 
    \[
        \phi(\bm{x}) = 0.1,\quad k(\bm{x}) = 80~ {\rm md}, \quad \alpha_{m} = 5~ {\rm ft}^{2}/{\rm day}, \quad \alpha_{l} = 50~ {\rm ft}/{\rm day},\quad  \alpha_{t} = 5~ {\rm ft}/{\rm day}.
    \]
\end{example}

Surface and contour plots of the invading fluid concentration at $t = 3$ and $10$ years are depicted in Figure \ref{fig:e4}. 
Due to the effect of the large adverse mobility ratio, the viscosity $\mu(c)$ given by the expression \eqref{mod:mu} changes rapidly across the steep fluid interface. Consequently, the Darcy velocity has a rapid change across the fluid interface. Therefore, as seen, the viscous fingering phenomena occurs, where the concentration front propagates significantly faster along the diagonal direction compared to the reference case $\mal M=1$ in Example \ref{exm:e3}.
This causes premature water breakthrough in the production well and reduces the swept area of injected water. As a result, the recovery rate declines.
\begin{figure}[!t]
	\centering
	\subfigure[Surface plot at $t=3$ years]{
		\includegraphics[width=0.45\textwidth]{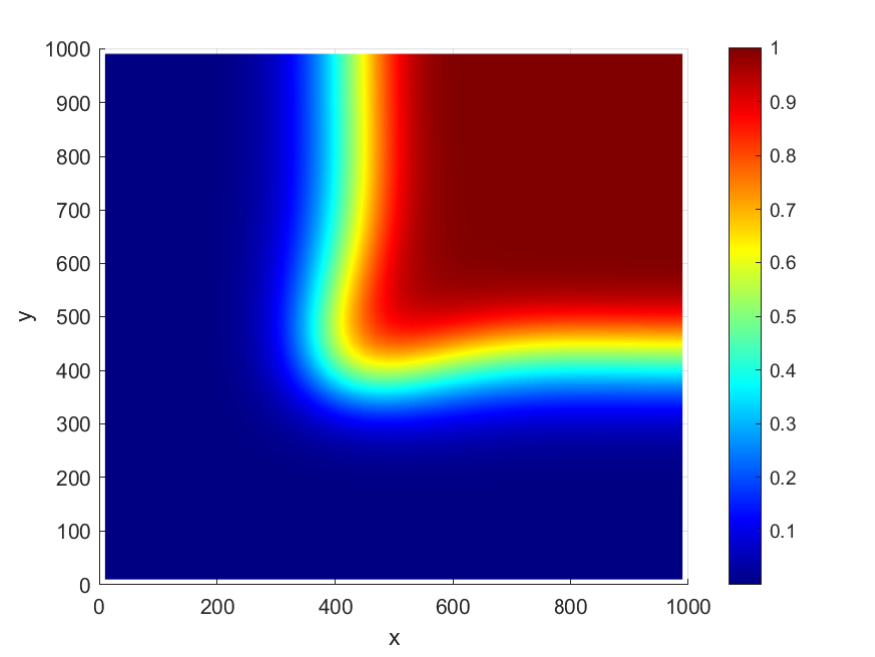}
		\label{fig:e4:1}
	}
	\subfigure[Contour plot at $t=3$ years]{
		\includegraphics[width=0.45\textwidth]{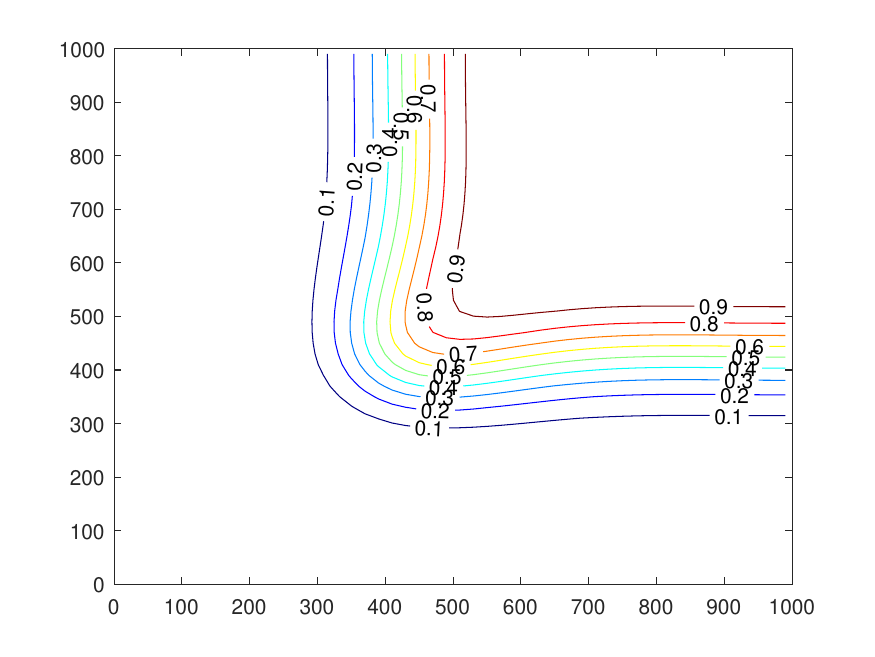}
		\label{fig:e4:2}
	}
	\subfigure[Surface plot at $t = 10$ years]{
		\includegraphics[width=0.45\textwidth]{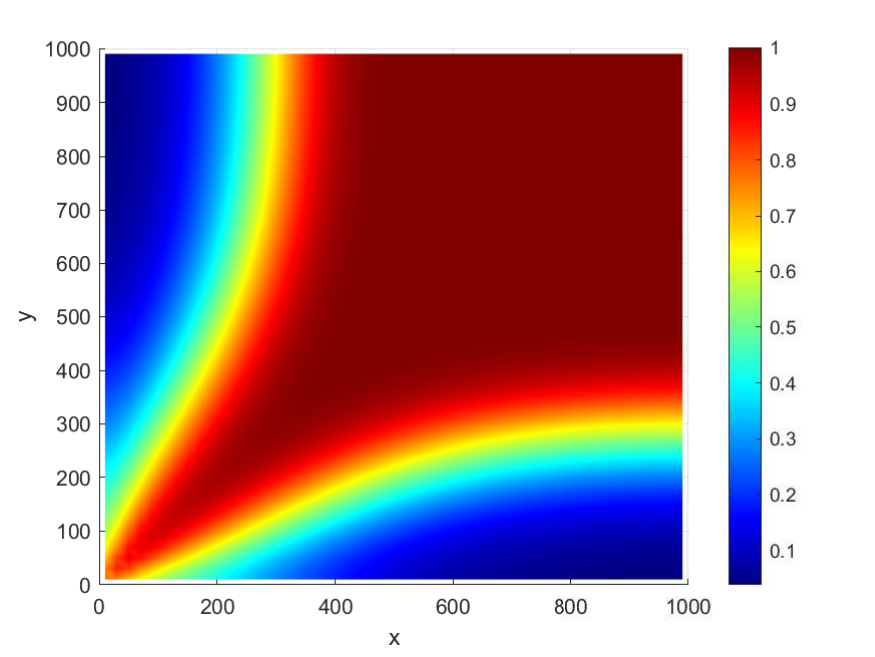}
		\label{fig:e4:3}
	}
	\subfigure[Contour plot at $t = 10$ years]{
		\includegraphics[width=0.45\textwidth]{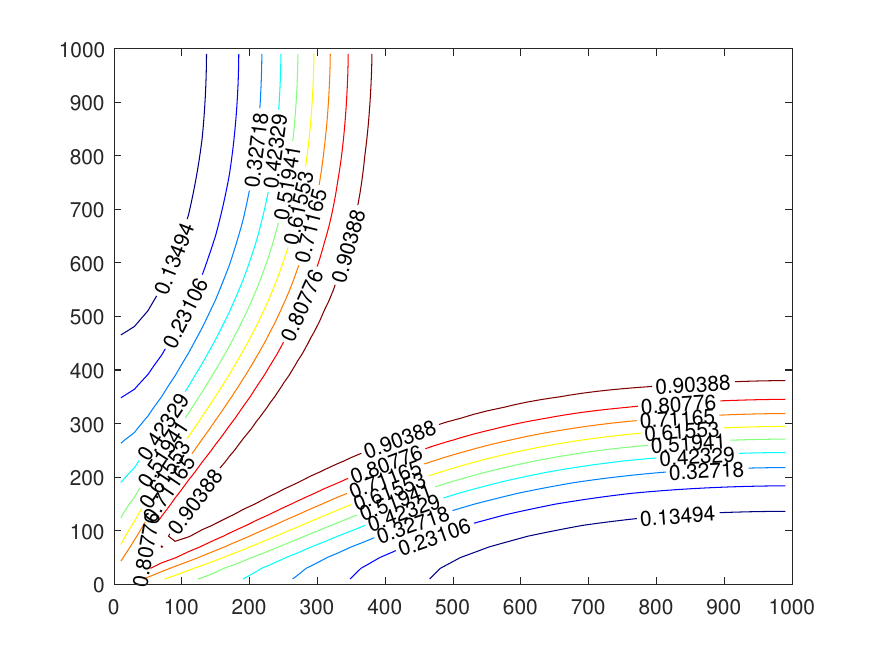}
		\label{fig:e4:4}
	}
	\caption{The concentration of the invading fluid at $t=3$ and $10$ years for Example \ref{exm:e4}.}
	\label{fig:e4}
\end{figure}

\begin{example}[\cite{WH00,CD07}]\label{exm:e5}
    In this example, we consider the numerical simulation of a miscible displacement in a heterogeneous porous media with a discontinuous permeability 
    \[
       k(\bm{x}) = \left\{ 
          \begin{aligned}
              & 80~ {\rm md}, && \Omega_{L}:=(0,1000)\times (0,500),\\
              & 20~ {\rm md}, && \Omega_{U}:=(0,1000)\times (500,1000),
          \end{aligned}
        \right.
    \]
    which is often encountered in many real production process. Other parameters remain consistent with those in Example \ref{exm:e4}.
\end{example}

Figure \ref{fig:e5} displays surface and contour plots of the invading fluid concentration at $t = 3$ and $10$ years. The concentration front initially propagates faster in the vertical direction than in the horizontal direction, as clearly depicted in Figure \ref{fig:e5:1}--\ref{fig:e5:2} .  This behavior stems from the higher permeability in the subdomain $\Omega_{L}$, which generates a higher Darcy velocity than in the subdomain $\Omega_{U}$. Figure \ref{fig:e5:3}--\ref{fig:e5:4} reveals that when the invading fluid reaches $\Omega_{L}$, it exhibits significantly accelerated horizontal movement relative to $\Omega_{U}$.
\begin{figure}[!t]
	\centering
	\subfigure[Surface plot at $t=3$ years]{
		\includegraphics[width=0.45\textwidth]{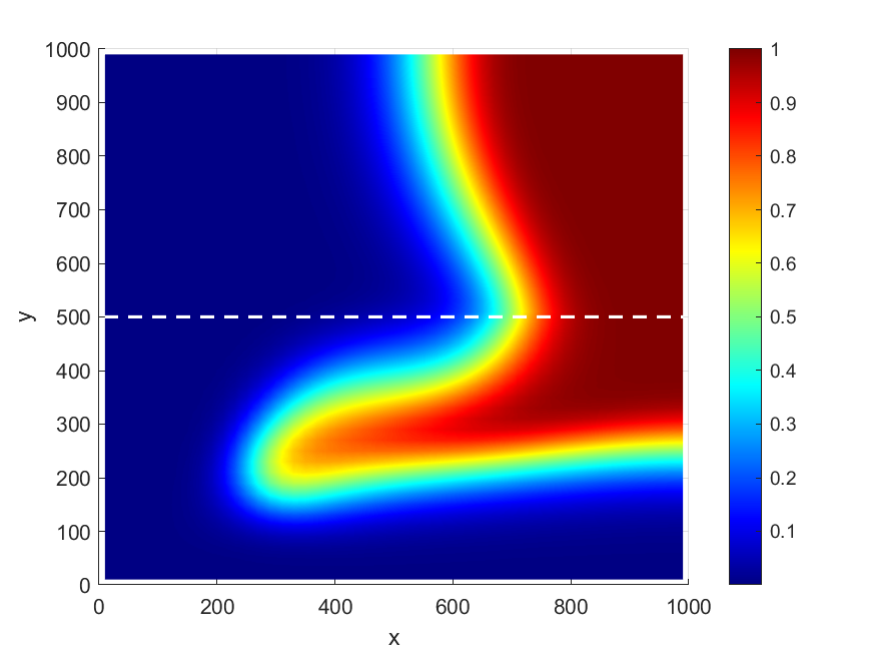}
		\label{fig:e5:1}
	}
	\hfill
	\subfigure[Contour plot at $t=3$ years]{
		\includegraphics[width=0.45\textwidth]{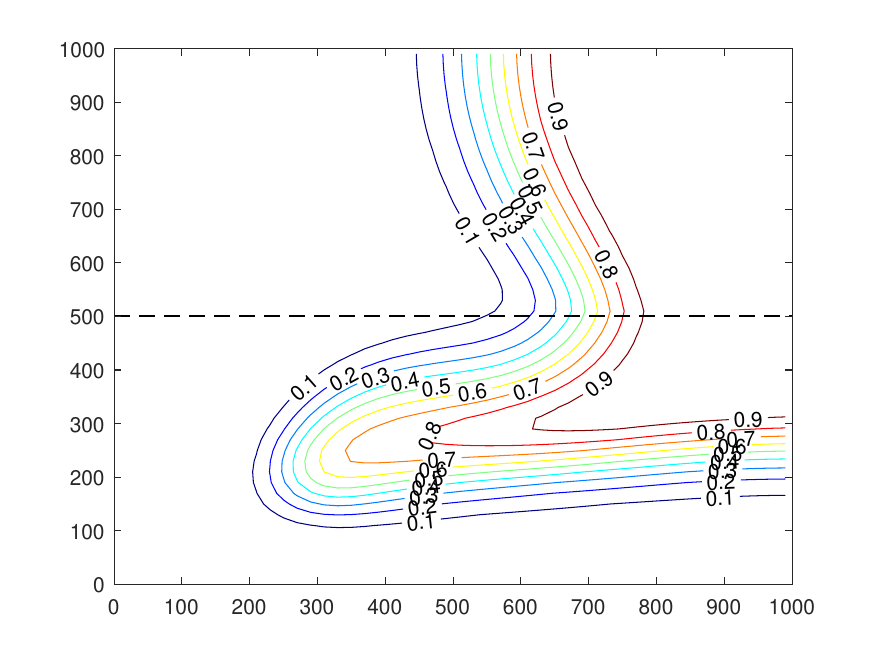}
		\label{fig:e5:2}
	}
	\vspace{0.1cm}
	\subfigure[Surface plot at $t = 10$ years]{
		\includegraphics[width=0.45\textwidth]{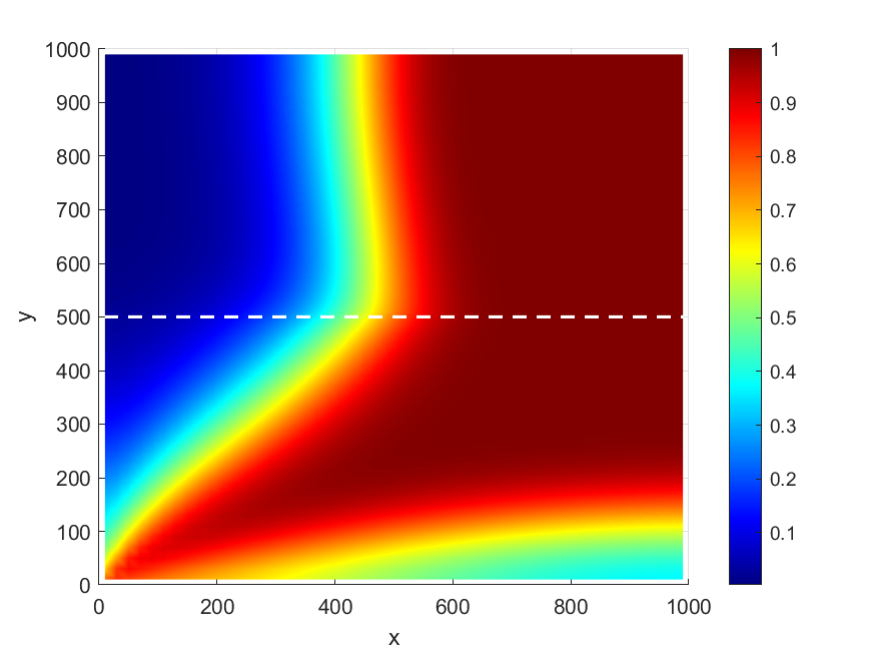}
		\label{fig:e5:3}
	}
	\hfill
	\subfigure[Contour plot at $t = 10$ years]{
		\includegraphics[width=0.45\textwidth]{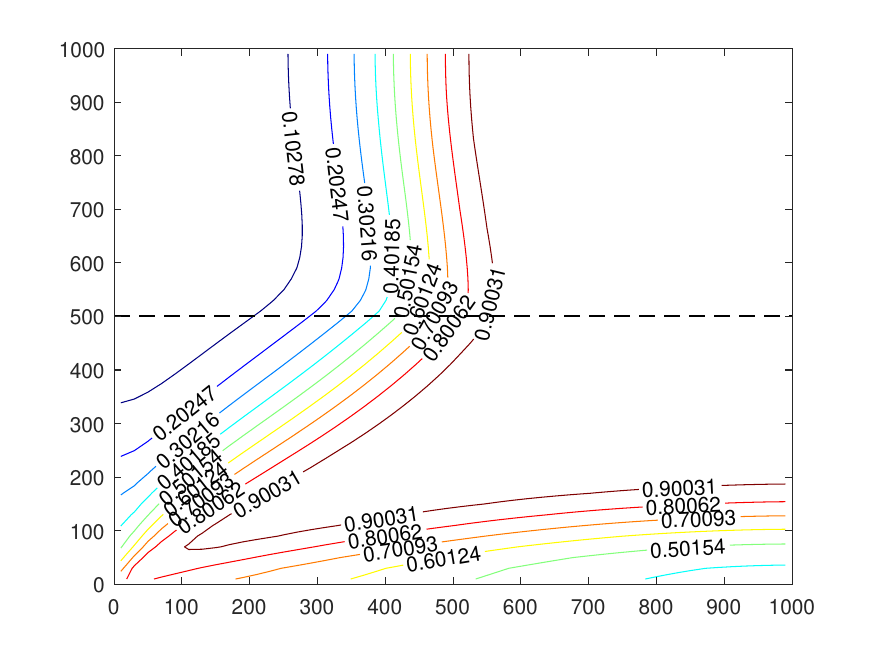}
		\label{fig:e5:4}
	}
	\caption{The concentration of the invading fluid at $t=3$ and $10$ years for Example \ref{exm:e5}.}
	\label{fig:e5}
\end{figure}

\begin{example}[\cite{WH00}]\label{exm:e6}
    In this example, we consider the porous media with piecewise structures, where the permeability and porosity of the media are respectively as
    \[
     \phi(\bm{x}) = \left\{ \begin{aligned}
         & 0.09, &\Omega_{1}:=(150,550)\times (150,550),\\
         & 0.1,  & \Omega_{2}:=\Omega \setminus \Omega_{1},
     \end{aligned} \right. \quad 
     k(\bm{x}) = \left\{ \begin{aligned}
         & 25~{\rm md}, &\Omega_{1},\\
         & 80~{\rm md},  & \Omega_{2},
     \end{aligned} \right. 
    \]
while all other parameters are identical to those of Example \ref{exm:e4}. 
\end{example}

Surface and contour plots of the invading fluid concentration at $t = 3$, $5$, $7$ and $10$ years are presented in Figure \ref{fig:e6}. Compared with Example \ref{exm:e4}, the injected fluid in this example exhibits much higher swept efficiency. Consequently, this improves the displacement efficiency and enhances the oil recovery. This indicates that the production well should preferably be placed in low-permeability zones if possible, thereby enhancing swept efficiency of the injected fluids. This is because well drilling (including injectors and producers) constitutes the primary cost in petroleum operations. In addition, this example demonstrates the feasibility of enhancing recovery rates through polymer-induced permeability alteration in reservoir porous media.
Given their high viscosity, polymers can selectively block or reduce permeability in certain pores or flow channels, effectively guiding fluid movement to optimize hydrocarbon recovery. In this context, the original fluid and porous media can be viewed as shown in Example \ref{exm:e4}, while by injecting polymers in a controlled manner, the properties of the porous media are modified to those presented in the current example. 
\begin{figure}[htbp]
	\centering
	\subfigure[Surface plot at $t=3$ years]{
		\includegraphics[width=0.45\textwidth]{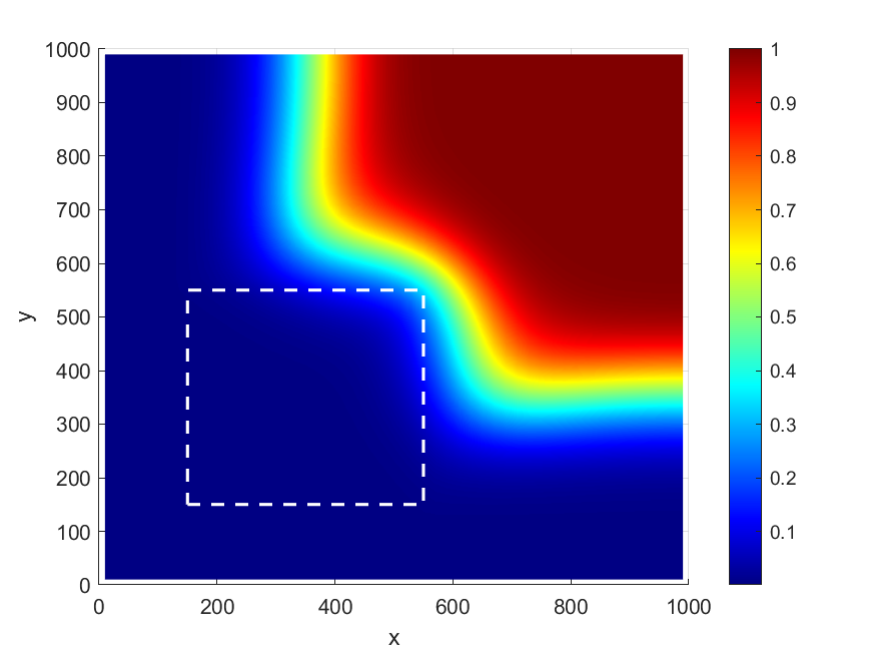}
		\label{fig:e6:1}
	}
	\subfigure[Contour plot at $t=3$ years]{
		\includegraphics[width=0.45\textwidth]{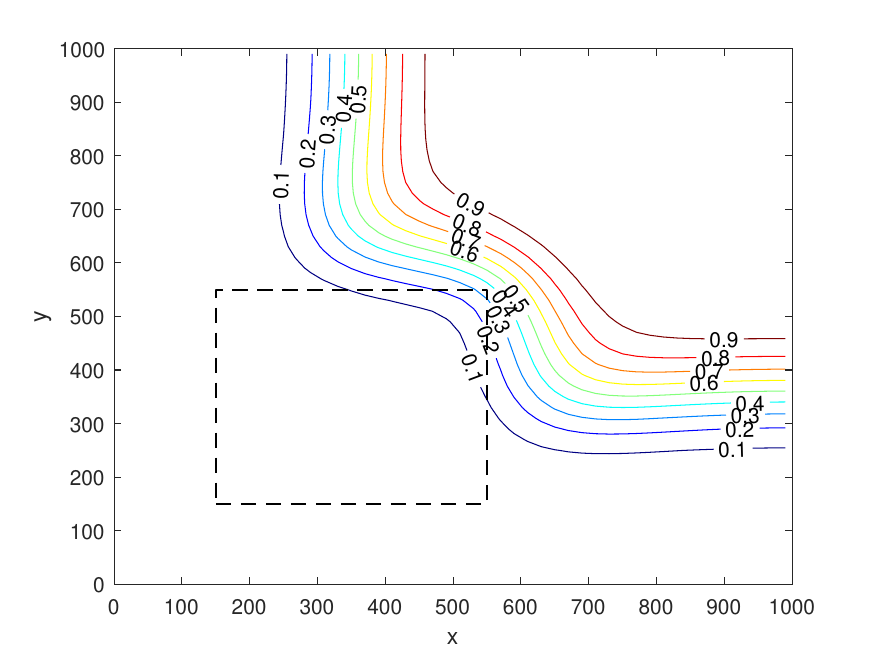}
		\label{fig:e6:2}
	}
	\subfigure[Surface plot at $t = 5$ years]{
		\includegraphics[width=0.45\textwidth]{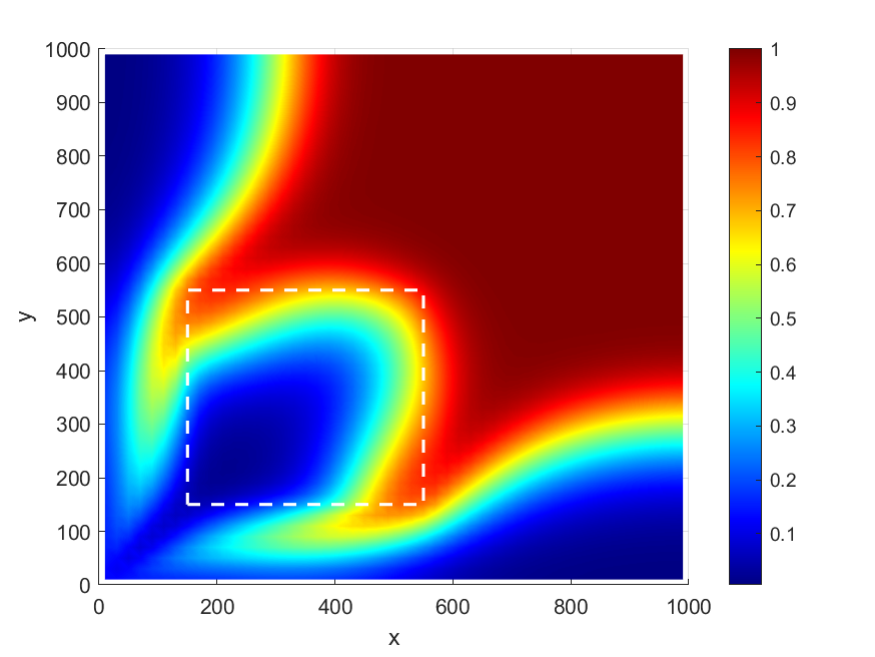}
		\label{fig:e6:3}
	}
	\subfigure[Contour plot at $t = 5$ years]{
		\includegraphics[width=0.45\textwidth]{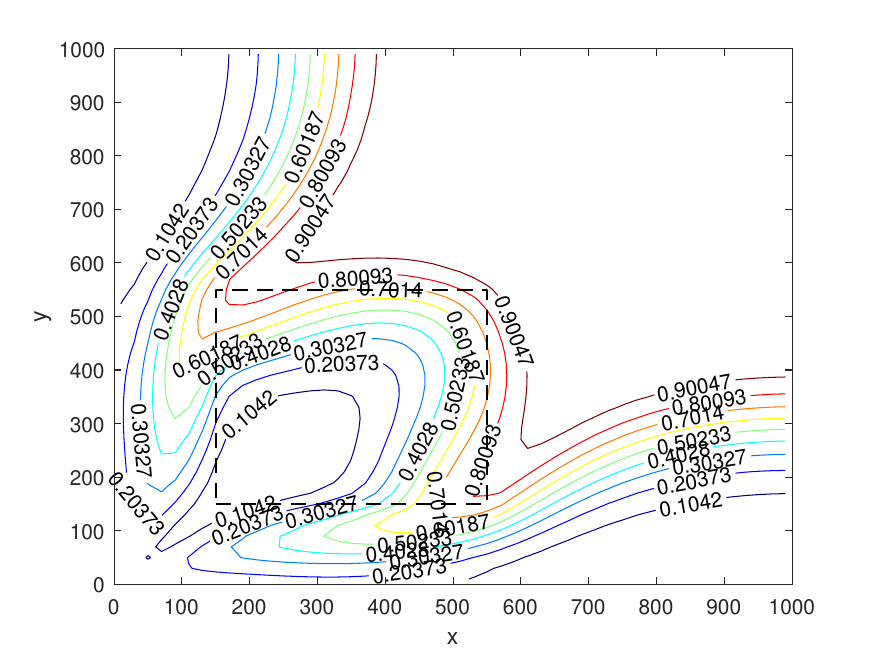}
		\label{fig:e6:4}
	}
	\subfigure[Surface plot at $t = 7$ years]{
		\includegraphics[width=0.45\textwidth]{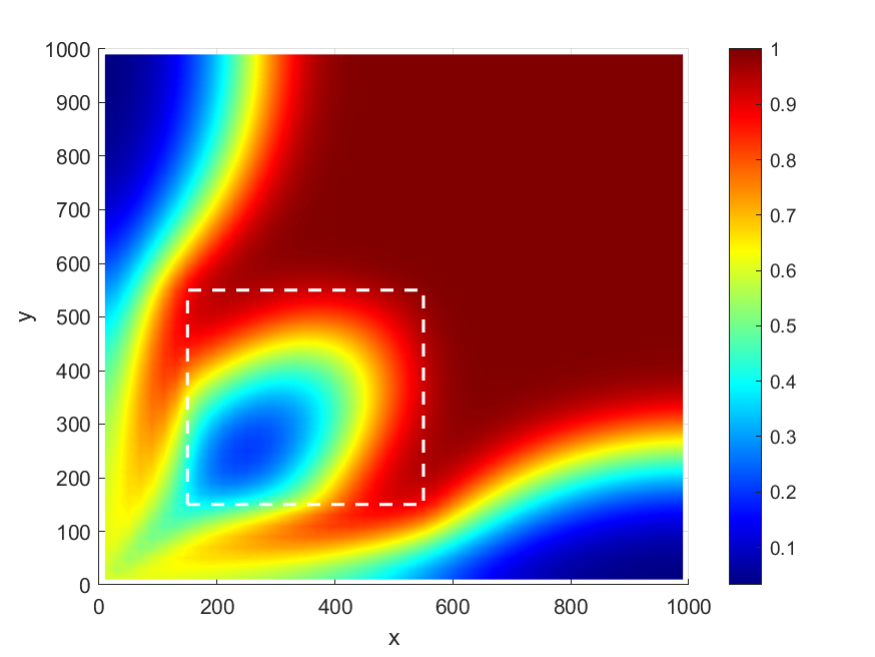}
		\label{fig:e6:5}
	}
	\subfigure[Contour plot at $t = 7$ years]{
		\includegraphics[width=0.45\textwidth]{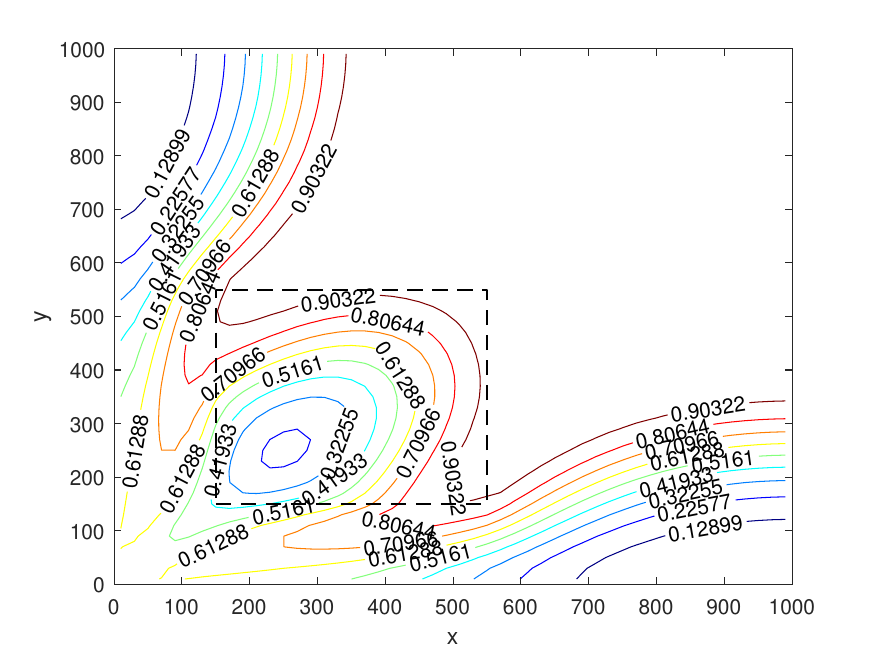}
		\label{fig:e6:6}
	}
	\subfigure[Surface plot at $t = 10$ years]{
		\includegraphics[width=0.45\textwidth]{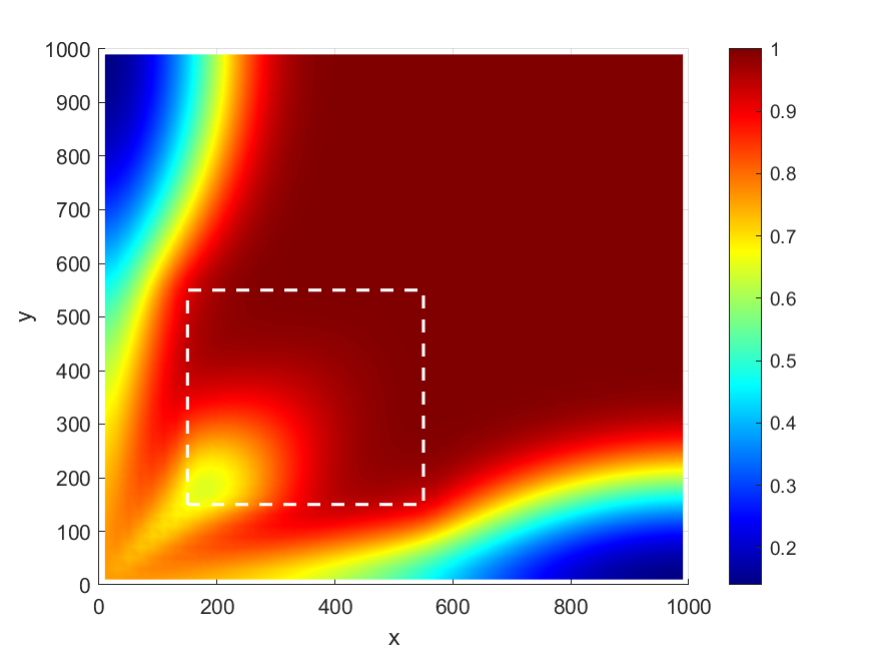}
		\label{fig:e6:7}
	}
	\hfill
	\subfigure[Contour plot at $t = 10$ years]{
		\includegraphics[width=0.45\textwidth]{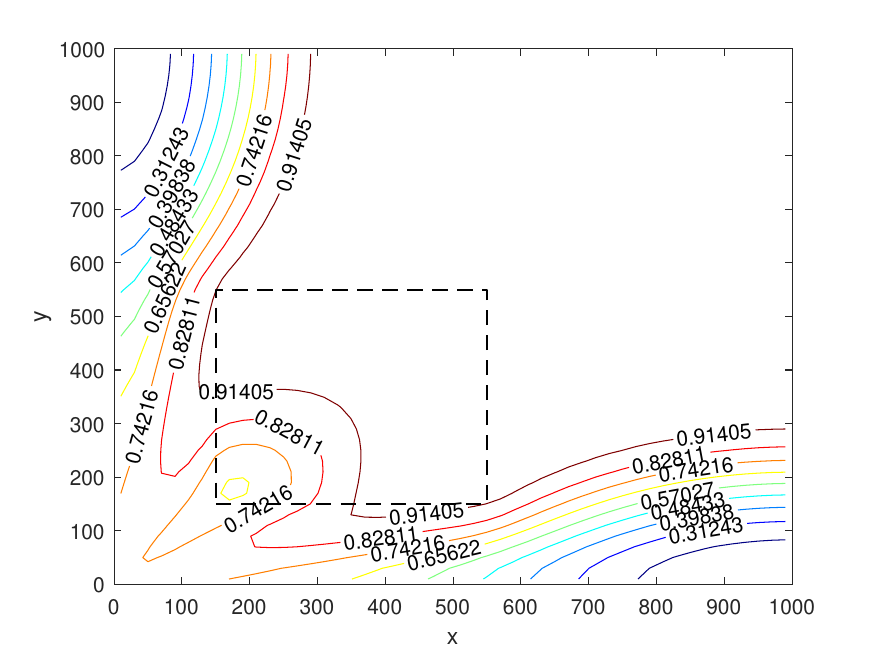}
		\label{fig:e6:8}
	}
	\caption{The concentration of the invading fluid at $t=3, 5, 7$ and $10$ years for Example \ref{exm:e6}.}
	\label{fig:e6}
\end{figure}

Finally, we also test the absolute mass errors for all the above examples. As illustrated in Figure \ref{fig:mass}, it demonstrates that our numerical scheme maintains stringent mass conservation when applied to real-world physical applications.
In summary, the numerical results in Examples \ref{exm:e3}--\ref{exm:e6} show that the proposed MC-MTS-CBCFD scheme can effectively simulate fluid flow in porous media with highly complex structures. Even when relatively coarse spatial grids and time steps are employed, it can still yield accurate and physically reasonable numerical solutions. The developed methodology significantly improves computational efficiency while simultaneously providing valuable support to predict and enhance oil recovery performance. In addition, the simulated results and physical phenomena are very similar to those obtained in \cite{WH00,CD07,VPV21,Eymard19}, which also demonstrate the reliability of the proposed method.
\begin{figure}[!t] 
    \centering
    \includegraphics[width=0.5\textwidth]{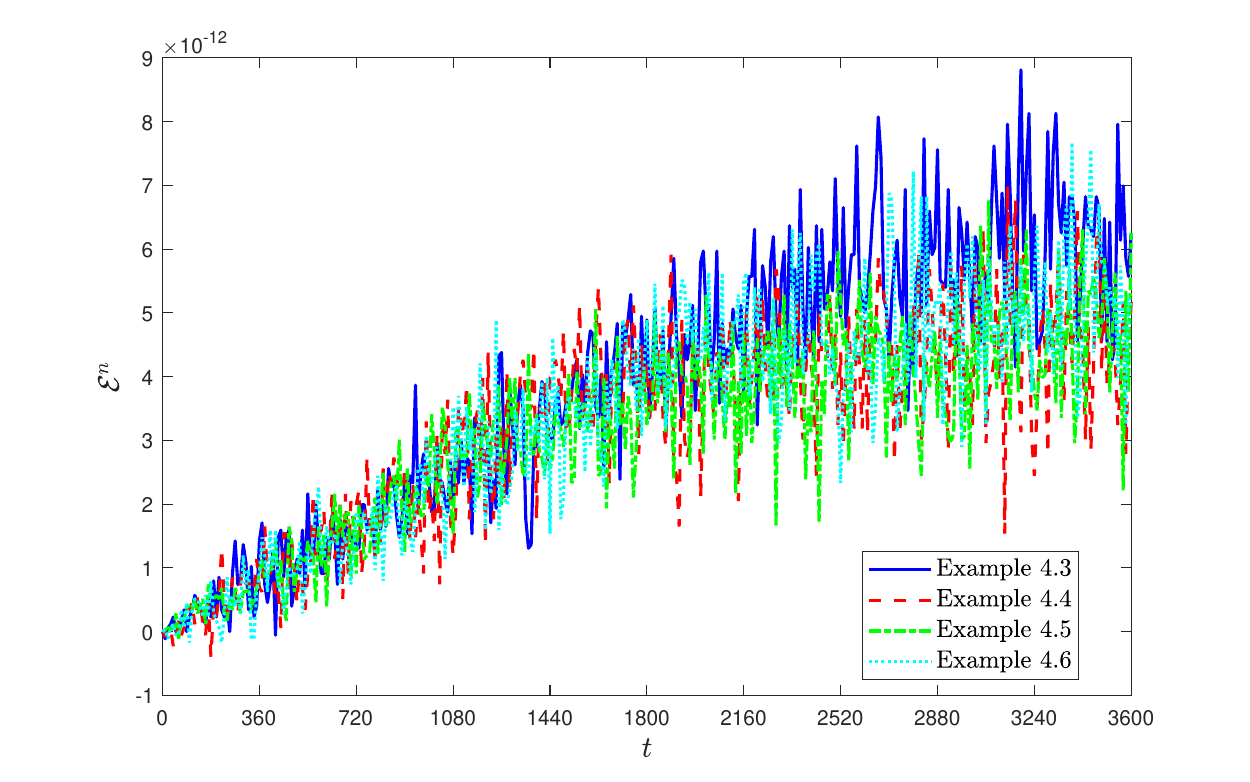}
	\caption{The mass errors for Examples \ref{exm:e3}--\ref{exm:e6}.}
	\label{fig:mass}
\end{figure}

\section{Concluding remarks}\label{sec:conclusion}
Second-order staggered FD methods have been well established for reservoir numerical simulation \cite{Rui20,LA22,LA24}, however, high-order FD methods have not yet been successfully discussed. In this paper, we developed a multi-time-step compact BCFD algorithm, namely MC-MTS-CBCFD, for the incompressible miscible displacement problem in porous media. The proposed approach possesses four distinguished advantages: (i) it can achieve fourth-order approximations for all physical variables, effectively eliminating the convergence order reduction for the velocity commonly encountered in traditional FD methods; (ii) it completely circumvents the need to solve saddle-point systems and the inversion of the tensor-form nonlinear diffusion coefficient, thereby avoiding the computational challenges inherent in classical mixed formulations; (iii) it preserves discrete mass conservation, a crucial property for practical long-term simulation; (iv) it significantly improves computational efficiency by using the multi-time-step approach, which is consistent with the actual multi-time scale behavior that the evolution of the velocity field is observed to be considerably slower than that of the concentration field. We established rigorous error analysis for the proposed MC-MTS-CBCFD scheme, and proved its unique solvability through the Browder’s fixed point theorem. Numerical experiments demonstrate the high accuracy, computational efficiency, and strong potential in 'real' simulation of various incompressible displacement scenarios. 

There are still many issues to be addressed. For example, the concentration $c$, defined as the volume fraction of injected fluid within the mixed phase, must satisfy the physical constraint $0 \leq c  \leq 1$. This bound-preserving property is expected to be maintained at the discrete level, see \cite{CXY19,FGKY22} and related applications \cite{ZXX18,XY19,SJ22,KJS24}. Our future work will focus on constructing efficient and high-order bound-preserving numerical methods for both the incompressible and compressible miscible displacement models. In addition, as seen from Tables \ref{tab:e1:cpu1}--\ref{tab:e1:cpu2}, the CPU time is greatly reduced for the velocity/pressure simulation via the multi-time-step strategy, however, it is still computationally expensive for the concentration equation. Thus, improving the computational efficiency of the concentration simulation is also a key research focus. Some possibilities include the development of Laplace-modified stabilized technique \cite{Dendy75,FYH25} or adoption of implicit pressure explicit concentration (IMPEC) time-marching approach \cite{Chen19,FGKY22,KXY24} for compressible miscible displacements in porous media. 

\section*{CRediT authorship contribution statement}
\textbf{Xiaoying Wang}: Methodology, Formal analysis, Software, Writing- Original draft.
\textbf{Hongxing Rui}: Methodology, Writing- Reviewing and Editing, Funding acquisition.
\textbf{Hongfei Fu}: Conceptualization, Supervision, Methodology, Writing- Reviewing and Editing, Funding acquisition.

\section*{Declaration of competing interest}
The authors declare that they have no competing interests.

\section*{Data availability } Data will be made available on request.

\section*{Acknowledgments}
This work was supported in part by the National Natural Science Foundation of China (No. 12131014), by the Shandong Provincial Natural Science Foundation (No. ZR2024MA023), by the Fundamental Research Funds for the Central Universities (No. 202264006) and by the OUC Scientific Research Program for Young Talented Professionals.

\appendix

\section{Proof of Lemma \ref{lem:err:pre}}\label{App:A}
\setcounter{equation}{0}
\renewcommand\theequation{A.\arabic{equation}}

First, applying the operator $\mal{L}^{-1}$ to both sides of \eqref{err:cu:p:1}, and taking the inner product for the resulting equation with $\ovl{e}_{c,p}^{1/2,*}$ in the sense of $(\cdot,\cdot)_{\rm M}$, we have
\begin{equation}\label{errcu:p:1}
   \begin{aligned}
		& \left( \phi D_{t} e_{c,p}^{1/2,*}  ,\ovl{e}_{c,p}^{1/2,*} \right)_{\rm M} + \left( \big[ \mal{L}_{x}^{-1} \de_{x} \ovl{e}_{w^{x},p} + \mal{L}_{y}^{-1} \de_{y} \ovl{e}_{w^{y},p} \big]^{1/2,*} ,\ovl{e}_{c,p}^{1/2,*}\right)_{\rm M}  \\
		&= \left( [q_{P} \ovl{e}_{c,p}]^{1/2,*}  ,\ovl{e}_{c,p}^{1/2,*}\right)_{\rm M} + \left( \mal{L}^{-1} E_{1,p}^{1/2,*},\ovl{e}_{c,p}^{1/2,*} \right)_{\rm M} \leq K\| e_{c,p}^{1,*} \|_{\rm M}^{2} +  K\left( (\De t_{p})^{4} + h^{8} \right), 
	\end{aligned}
\end{equation}
where $e_{c}^{0}=0$, the truncation error in \eqref{err:truc:p}, and Lemma \ref{lem:op:L} are used in the last step.

Analogous to \eqref{erru:2}, using Lemma \ref{lem:dis} and error equations \eqref{err:cu:p:2}--\eqref{err:cu:p:5}, the second term on the left-hand side of \eqref{errcu:p:1} can be written as
\begin{equation}\label{errcu:p:2}
	\begin{aligned}
		&\left( \big[ \mal{L}_{x}^{-1} \de_{x} \ovl{e}_{w^{x},p} + \mal{L}_{y}^{-1} \de_{y} \ovl{e}_{w^{y},p} \big]^{1/2,*} ,\ovl{e}_{c,p}^{1/2,*}\right)_{\rm M}\\
		&= - \left( \ovl{e}_{w^{x},p}^{1/2,*},  \mal{L}_{x}^{-1} \de_{x} \ovl{e}_{c,p}^{1/2,*}\right)_{x} - \left(   \ovl{e}_{w^{y},p}^{1/2,*} , \mal{L}_{y}^{-1} \de_{y} \ovl{e}_{c,p}^{1/2,*}\right)_{y}\\
        & = \left(  \ovl{\bm{e}}_{\bm{w},p}^{1/2,*},  D^{-1} \ovl{\bm{e}}_{\bm{w},p}^{1/2,*} \right)_{\rm T} + I_{2}  =: \|\ovl{\bm{e}}_{\bm{w},p}^{1/2,*}\|_{{\rm T},D^{-1}}^{2} + I_{2},
	\end{aligned}
\end{equation}
where
\begin{equation*}
    \begin{aligned}
        & I_{2}
        :=\left(   \ovl{e}_{\bm{w},p}^{1/2,*} , D^{-1}\ovl{\bm{\eta}}_{p}^{1/2,*}\right)_{\rm T} 
                 - \left(   \ovl{e}_{w^{x},p}^{1/2,*} ,   \mal{L}_{x}^{-1} \ovl{E}_{2,p}^{x,1/2,*}\right)_{x}
		- \left(   \ovl{e}_{w^{y},p}^{1/2,*},  \mal{L}_{y}^{-1} \ovl{E}_{2,p}^{y,1/2,*} \right)_{y},\\
        & \bm{\eta}_{p} =(\eta_{p}^{x},~\eta_{p}^{y}):= (  U_{\#}^{x} \mal{T}_{x} C_{p} - u_{\#}^{x} c_{p} - E_{3,p}^{x} ,~    U_{\#}^{y} \mal{T}_{y} C_{p} - u_{\#}^{y} c_{p} - E_{3,p}^{y} ).
    \end{aligned}
\end{equation*}

Utilizing Corollary \ref{col:erru0} and the inverse inequality, it can be deduced that  
\begin{equation}\label{bound:u0}
	\| \bm{U}^{0}\|_{\infty} \leq \| \bm{u}^{0}\|_{\infty} + \| \bm{e}_{\bm{u}}^{0}\|_{\infty} 
    \leq K_{1} + Kh^{-1} \| \bm{e}_{\bm{u}}^{0}\|_{\rm T} 
    \leq K_{1} + Kh^{3}  \leq K_{1} + 1,
\end{equation}
for sufficiently small $h$, which implies that $\bm{U}^{0}$ is uniformly bounded. 
Further, Lemmas \ref{lem:op:t}--\ref{lem:t:err} and Corollary \ref{col:erru0} imply that
\begin{equation}\label{errcu:p:3}
	\begin{aligned}
		\|\ovl{\bm{\eta}}_{p}^{1/2,*}\|_{\rm T} 
        & = \sum_{\kappa=x,y} \|\ovl{c}_{p}^{1/2}e_{u^{\kappa}}^{0} + U^{\kappa,0}[\ovl{c}_{p}-\mal{T}_{\kappa}\ovl{c}_{p}]^{1/2} + U^{\kappa,0}\mal{T}_{\kappa}\ovl{e}_{c,p}^{1/2,*} + \ovl{E}_{3,p}^{\kappa,1/2,*}\|_{\kappa} \\
		& \leq K\|e_{c,p}^{1,*}\|_{\rm M} + K \De t_{p} + Kh^{4}.
	\end{aligned}
\end{equation}
Then, utilizing estimate \eqref{errcu:p:3} and the Cauchy-Schwarz inequality, the $I_{2}$ term can be estimated as
\begin{equation}\label{errcu:p:4}
	\begin{aligned}
		I_{2}
		& \leq \left(\|\ovl{e}_{w^{x},p}^{1/2,*}\|_{x} + \|\ovl{e}_{w^{y},p}^{1/2,*}\|_{y}\right) \left( K\|e_{c,p}^{1,*}\|_{\rm M} + K\De t_{p} + Kh^{4}  \right) \\
		& \leq \f{1}{2} \|\ovl{\bm{e}}_{\bm{w},p}^{1/2,*}\|_{{\rm T},D^{-1}}^{2} + K \|e_{c,p}^{1,*}\|_{\rm M}^{2} + K \De t_{p}\|\ovl{\bm{e}}_{\bm{w},p}^{1/2,*}\|_{\rm T} + Kh^{8}.
	\end{aligned}
\end{equation}
Thus, substituting \eqref{errcu:p:2} and \eqref{errcu:p:4} into \eqref{errcu:p:1}, and using $e_{c}^{0}=0$, we obtain
\begin{equation}\label{errcu:p:5}
    \begin{aligned}
        &\f{\phi_{*}}{2\De t_{p}}\|e_{c,p}^{1,*}\|_{\rm M}^{2} + \f{1}{2} \|\ovl{\bm{e}}_{\bm{w},p}^{1/2,*}\|_{{\rm T},D^{-1}}^{2} \\
        &\leq K\|e_{c,p}^{1,*}\|_{\rm M}^{2}  + K \De t_{p}\|\ovl{\bm{e}}_{\bm{w},p}^{1/2,*}\|_{\rm T} +  K\left( ( \De t_{p} )^{4} + h^{8} \right).
    \end{aligned}
\end{equation}

Next, in order to achieve global second-order temporal accuracy, applying the operator $\mal{L}^{-1}$ to both sides of \eqref{err:cu:p:1}, and taking the inner product for the resulting equation with $D_{t} e_{c,p}^{1/2,*}$ in the sense of $(\cdot,\cdot)_{\rm M}$ gives us
\begin{equation}\label{errcu:p:6}
	\begin{aligned}
		& \left( \phi D_{t} e_{c,p}^{1/2,*}, D_{t} e_{c,p}^{1/2,*} \right)_{\rm M}  
        + \left( \big[ \mal{L}_{x}^{-1} \de_{x} \ovl{e}_{w^{x},p} + \mal{L}_{y}^{-1} \de_{y} \ovl{e}_{w^{y},p} \big]^{1/2,*}, D_{t} e_{c,p}^{1/2,*} \right)_{\rm M} \\
		& = \left( [q_{P} \ovl{e}_{c,p}]^{1/2,*}, D_{t}e_{c,p}^{1/2,*} \right)_{\rm M}  
          + \left( \mal{L}^{-1}E_{1,p}^{1/2,*}, D_{t} e_{c,p}^{1/2,*} \right)_{\rm M} \\
		& \leq \f{ \phi_{*} }{2}\|D_{t} e_{c,p}^{1/2,*}\|_{\rm M}^{2} 
        + K\| e_{c,p}^{1,*}\|_{\rm M}^{2} 
        + K\left( ( \De t_{p} )^{4} + h^{8} \right),
	\end{aligned}
\end{equation}
where the truncation error in \eqref{err:truc:p} and Lemma \ref{lem:op:L}  are utilized in the last step.

Similar to \eqref{errcu:p:2}, the second left-hand side term of \eqref{errcu:p:6} can be transformed into
\begin{equation}\label{errcu:p:7}
	\begin{aligned}
		&\left( \big[ \mal{L}_{x}^{-1} \de_{x} \ovl{e}_{w^{x},p} + \mal{L}_{y}^{-1} \de_{y} \ovl{e}_{w^{y},p} \big]^{1/2,*}, 
                D_{t} e_{c,p}^{1/2,*}\right)_{\rm M}\\
		& = - \left( \ovl{e}_{w^{x},p}^{1/2,*}, \mal{L}_{x}^{-1} \de_{x} D_{t} e_{c,p}^{1/2,*}\right)_{x}
            - \left( \ovl{e}_{w^{y},p}^{1/2,*}, \mal{L}_{y}^{-1} \de_{y} D_{t} e_{c,p}^{1/2,*}\right)_{y}\\
		&= \f{1}{2\De t_{p}} \left[ \|\bm{e}_{\bm{w},p}^{1,*}\|_{{\rm T},D^{-1}}^{2} - \|\bm{e}_{\bm{w},p}^{0}\|_{{\rm T},D^{-1}}^{2} \right] + I_{3},
	\end{aligned}
\end{equation}
where 
\begin{equation*}
    I_{3}:= \left( \ovl{e}_{\bm{w},p}^{1/2,*},  D^{-1} D_{t} \bm{\eta}_{p}^{1/2,*} \right)_{\rm T} 
          - \left( \ovl{e}_{w^{x},p}^{1/2,*},   \mal{L}_{x}^{-1} D_{t} E_{2,p}^{x,1/2,*} \right)_{x}
		 - \left( \ovl{e}_{w^{y},p}^{1/2,*},   \mal{L}_{y}^{-1} D_{t} E_{2,p}^{y,1/2,*} \right)_{y}.
\end{equation*}
Analogous to \eqref{errcu:p:3}, we have
\begin{equation}\label{errcu:p:8}
	\begin{aligned}
		 \|D_{t} \bm{\eta}_{p}^{1/2,*}\|_{\rm T} 
        & = \sum_{\kappa=x,y}\|D_{t} c_{p}^{1/2}e_{u^{\kappa}}^{0} + U^{\kappa,0}D_{t}[ c_{p}-\mal{T}_{\kappa}c_{p} ]^{1/2} + U^{\kappa,0}\mal{T}_{\kappa} D_{t} e_{c,p}^{1/2,*} + D_{t}E_{3,p}^{\kappa,1/2,*}\|_{\kappa} \\
		&\leq K\| D_{t} e_{c,p}^{1/2,*}\|_{\rm M} + K \De t_{c} + Kh^{4},
	\end{aligned}
\end{equation}
and then similar to \eqref{errcu:p:4}, combining the estimates \eqref{errcu:p:8} and \eqref{err:truc}, the $I_{3}$ term can be estimated as
\begin{equation}\label{errcu:p:9}
	\begin{aligned}
		I_{3}
		& \leq \left(\|\ovl{e}_{w^{x},p}^{1/2,*}\|_{x} + \|\ovl{e}_{w^{y},p}^{1/2,*}\|_{y}\right) 
               \left( K\| D_{t} e_{c,p}^{1/2,*}\|_{\rm M} + K\De t_{p} + Kh^{4}  \right) \\
		& \leq K \|\ovl{\bm{e}}_{\bm{w},p}^{1/2,*}\|_{{\rm T},D^{-1}}^{2} 
            + \f{\phi_{*}}{2} \|D_{t} e_{c,p}^{1/2,*}\|_{\rm M}^{2} 
            + K \De t_{p} \|\ovl{\bm{e}}_{\bm{w},p}^{1/2,*}\|_{\rm T} + Kh^{8}.
	\end{aligned}
\end{equation}

By inserting the estimates \eqref{errcu:p:7} and \eqref{errcu:p:9} into \eqref{errcu:p:6}, and using Lemma \ref{lem:op:L}, we get
\begin{equation}\label{errcu:p:10}
      \begin{aligned}
          &\f{1}{2\De t_{p}} \left[ \|\bm{e}_{\bm{w},p}^{1,*}\|_{{\rm T},D^{-1}}^{2} - \|\bm{e}_{\bm{w}}^{0}\|_{{\rm T},D^{-1}}^{2} \right] \\
	   & \leq K \|\ovl{\bm{e}}_{\bm{w},p}^{1/2,*}\|_{{\rm T},D^{-1}}^{2} 
          + K\De t_{p}\|\ovl{\bm{e}}_{\bm{w},p}^{1/2,*}\|_{\rm T}
          + K\| e_{c,p}^{1,*}\|_{\rm M}^{2}  +  K\left( ( \De t_{p} )^{4} + h^{8} \right).
      \end{aligned}
\end{equation}
Now, adding \eqref{errcu:p:10} and \eqref{errcu:p:5} together and multiplying both sides of the resulting equation by $2\De t_{p}$, for $\De t_{p} \leq \tau_{0}$ small enough, the following estimate holds
\begin{equation}\label{err:c1}
	\|e_{c,p}^{1,*}\|_{\rm M}^{2} +  \De t_{p} \|\bm{e}_{\bm{w},p}^{1/2,*}\|_{\rm T}^{2} + \|\bm{e}_{\bm{w},p}^{1,*}\|_{\rm T}^{2}   \leq K\left( ( \De t_{p} )^{4} + h^{8} \right).
\end{equation}

Finally, similar to the proof of Lemma \ref{lem:err:pu}, using error equations \eqref{err:cu:p:6}--\eqref{err:cu:p:8} and truncation errors \eqref{err:truc:p}, and along with \eqref{err:c1}, we derive
\begin{equation*}
	\|e_{p}^{1,*}\|_{\rm M}^{2} + |e_{p}^{1,*}|_{1}^{2} + \| 
	\bm{e}_{\bm{u}}^{1,*} \|_{\rm T}^{2} 
	\leq  K\| e_{c,p}^{1,*} \|_{\rm M}^{2} + Kh^{8} \leq K\left( (\De t_{p})^{4}  +  h^{8}\right).
\end{equation*}
Thus, we complete the proof.

\bibliographystyle{spmpsci}   
\bibliography{reference} 

\end{document}